\documentclass[11pt,twoside, reqno]{article}

\usepackage{amssymb}
\usepackage{amsmath}
\usepackage{mathrsfs}
\usepackage{amsthm}
\usepackage{txfonts}

\usepackage[colorlinks=true, pdfstartview=FitV, linkcolor=blue, citecolor=blue, urlcolor=blue]{hyperref}

\allowdisplaybreaks

\def\rr{{\mathbb R}}
\def\rn{{{\rr}^n}}
\def\zz{{\mathbb Z}}

\def\nn{{\mathbb N}}
\def\fz{\infty}

\def\ccc{{\mathbb C}}

\def\cs{{\mathcal S}}

\def\az{\alpha}

\renewcommand\tilde{\widetilde}

\def\supp{{\rm{\,supp\,}}}
\def\esup{\mathop\mathrm{\,ess\,sup\,}}

\def\ls{\lesssim}
\def\lz{\lambda}

\def\ez{\varepsilon}

\def\vz{\varphi}

\def\sz{\sigma}
\def\boz{\Omega}

\def\hs{\hspace{0.3cm}}

\def\dint{\displaystyle\int}

\def\r{\right}
\def\lf{\left}

\def\bint{{\ifinner\rlap{\bf\kern.30em--}
\int\else\rlap{\bf\kern.35em--}\int\fi}\ignorespaces}

\def\sbint{{\ifinner\rlap{\bf\kern.32em--}
\hspace{0.078cm}\int\else\rlap{\bf\kern.45em--}\int\fi}\ignorespaces}

\def\dfrac{\displaystyle\frac}
\def\dsup{\displaystyle\sup}

\newtheorem{theorem}{Theorem}[section]
\newtheorem{lemma}[theorem]{Lemma}

\newtheorem{proposition}[theorem]{Proposition}

\theoremstyle{definition}
\newtheorem{remark}[theorem]{Remark}
\newtheorem{definition}[theorem]{Definition}
\numberwithin{equation}{section}

\numberwithin{equation}{section}

\numberwithin{equation}{section}

\begin{document}

\arraycolsep=1pt

\title{\Large\bf Real-Variable Theory of Anisotropic Musielak-Orlicz-Lorentz Hardy
Spaces with Applications to Calder\'{o}n-Zygmund Operators
\footnotetext{\hspace{-0.35cm}
{\it 2020 Mathematics Subject Classification}.
{Primary 42B35; Secondary 42B30, 35J25, 42B37, 42B25.}
\endgraf{\it Key words and phrases.} Musielak-Orlicz-Lorentz Hardy
space, atomic decomposition, anisotropic, Calder\'{o}n-Zygmund operator.
\endgraf X. Liu is supported by
the Gansu Province Education Science and Technology Innovation Project (Grant No. 224040),
and
the Foundation for Innovative Fundamental Research Group Project of Gansu Province (Grant No. 25JRRA805).
W. Wang is supported by the China Postdoctoral Science Foundation (Grant No. 2024M754159), and the Postdoctoral Fellowship Program of CPSF (Grant No. GZB20230961).
\endgraf $^\ast$Corresponding author}}
\author{Xiong Liu and Wenhua Wang$^\ast$}
\date{  }
\maketitle

\vspace{-0.8cm}

\begin{center}
\begin{minipage}{13cm}\small
{\noindent{\bf Abstract.}
Let $\varphi: \mathbb{R}^{n}\times[0,\infty)\rightarrow[0,\infty)$ be a Musielak-Orlicz function satisfying the
uniformly anisotropic Muckenhoupt condition and be of uniformly lower type $p^-_{\varphi}$ and of uniformly
upper type $p^+_{\varphi}$ with $0<p^-_{\varphi}\leq p^+_{\varphi}<\infty$, $q\in(0,\infty]$, and $A$ be a general expansive matrix on $\mathbb{R}^{n}$.
In this paper, the authors first study the anisotropic Musielak-Orlicz-Lorentz Hardy space $H^{\varphi,q}_A(\mathbb{R}^{n})$,
which coincides with the known anisotropic weak Musielak-Orlicz Hardy space $H^{\varphi,\infty}_A(\mathbb{R}^{n})$ when $q=\infty$,
and then establish their atomic and molecular decompositions.
As some applications, the authors obain the boundedness of anisotropic Calder\'{o}n-Zygmund operators on $H^{\varphi,q}_A(\mathbb{R}^{n})$
when $q\in(0,\infty)$ and from the anisotropic Musielak-Orlicz Hardy space $H^{\varphi}_A(\mathbb{R}^{n})$ to $H^{\varphi,\infty}_A(\mathbb{R}^{n})$
in the critical case. All the ranges of the exponents considered are shown to be the best possible, significantly improving upon existing results for
$H^{\varphi,\infty}_A(\mathbb{R}^{n})$ via widening the original assumption $0<p^-_{\varphi}\leq p^+_{\varphi}\leq1$
into the full range $0<p^-_{\varphi}\leq p^+_{\varphi}<\infty$, and all the results for  $q\in(0,\infty)$ are novel
and generalize from isotropic setting to anisotropic frameworks.}
\end{minipage}
\end{center}

%
%
%
%
%

\section{Introduction}\label{s1}
\hskip\parindent
It is well known that Musielak-Orlicz functions generalize classical Orlicz functions by allowing variation in both the spatial and growth variables (see e.g. \cite{d05,k14,lhy12,m83}). The motivation for studying Musielak-Orlicz spaces from their wide range of applications across various fields of mathematics and physics (see e.g. \cite{d05,k14,lhy12,m83}). Notably, certain special cases of Musielak-Orlicz (Hardy) spaces naturally arise in the study of pointwise multipliers on BMO-type spaces (see e.g. \cite{n93,ny85,ny97}), endpoint estimates for the div-curl lemma (see e.g. \cite{bgk12,bfg10,yyz21}), boundedness of commutators of singular integral operators in critical cases (see e.g. \cite{k13,lsuyy12,lyz24}), and bilinear decompositions of products involving Hardy spaces and their dual spaces (see e.g. \cite{bgk12,bijz07,lyz24,yyz21}).

A series of works (see e.g.  \cite{bckyy13b,ccyy14,ccyy16,cfyy16,hyy13,lhy12,ly15,ylk17,yy15,yy12,yy14,yy16,yy19}) have revealed that Musielak-Orlicz Hardy spaces possess finer structural properties than classical Hardy spaces, making them more delicate and intriguing objects of study. Under the assumption that $\varphi$ is a growth function, the atomic and finite atomic characterizations of the Musielak-Orlicz Hardy space $H^\varphi(\mathbb{R}^n)$ were first established in \cite{k14}, and later extended to the anisotropic setting $H^{\varphi}_A(\mathbb{R}^{n})$ in \cite{fhly17,lfy15,lffy16,lyy14,lhyy20}.

In parallel, real-variable characterizations--including those via atoms, molecules, and maximal functions-have been developed for various related function spaces. These include the anisotropic weighted Hardy space, the (anisotropic) weak Musielak-Orlicz Hardy space, the anisotropic (variable) Hardy(-Lorentz) space, and the anisotropic mixed-norm Hardy space, with corresponding results found in \cite{blyz08,blyz10,lby14}, \cite{lsll20,lyj16,qzl18,sll19,zql17}, \cite{b03,bw13,lbyz10,l21,lwyy18,lwyy19,lyy16,lyy17,lyz24,lql20,w23,wlwl20}, and \cite{hlyy19,hlyy20}, respectively.
Furthermore, in \cite{ins22}, the atomic and wavelet characterizations of $H^\varphi(\mathbb{R}^n)$ were obtained through an alternative approach that does not require the assumption that $\varphi$ is a growth function, there by improving upon the results in \cite{k14}.

More recently, the atomic and molecular decompositions of the Musielak-Orlicz-Lorentz Hardy spaces $H^{\varphi,q}(\mathbb{R}^{n})$ for any $q \in (0,\infty]$ were established in \cite{jwyyz23}, by using a novel approach that leverages the relationship between weighted Lebesgue spaces and Musielak-Orlicz spaces. This was achieved without assuming the boundedness of the powered Hardy-Littlewood maximal operator on the associated Musielak-Orlicz spaces or the growth condition on $\varphi$. For more information about the real-variable characterizations and applications of $H^{\varphi,q}(\mathbb{R}^{n})$, we refer the reader to \cite{jwyyz23,jcwyy25}.

On the other hand, there has been growing interest in extending classical function spaces arising in harmonic analysis from the Euclidean setting to various anisotropic contexts. For example, in 2003,
Bownik \cite{b03} introduced and investigated the anisotropic Hardy spaces $H_A^p(\rn)$
with $p\in(0,\,\fz)$, where $A$ is a general expansive matrix on $\rn$. Since then, variants
 of classical Hardy spaces on the anisotropic Euclidean space have been
studied and their real-variable theories have been well developed. To be precise,
 Bownik et al. \cite{blyz08} further extended the anisotropic Hardy space to the
weighted setting. Li et al. \cite{lfy15} introduced the anisotropic Musielak-Orlicz type Hardy
spaces $H_A^{\varphi}(\rn)$, where $\varphi$ is an anisotropic Musielak-Orlicz function, and characterized these spaces by several maximal functions and atoms.
For more studies about Hardy-type function spaces on
the anisotropic Euclidean spaces, we refer the reader to \cite{blyz08,blyz10,bw13,fhly17,hlyy19,hlyy20,lby14}.

Let $\varphi: \mathbb{R}^{n}\times[0,\infty)\rightarrow[0,\infty)$ be a Musielak-Orlicz function satisfying the
uniformly anisotropic Muckenhoupt condition and be of uniformly lower type $p^-_{\varphi}$ and of uniformly
upper type $p^+_{\varphi}$ with $0<p^-_{\varphi}\leq p^+_{\varphi}<\infty$, $q\in(0,\infty]$, and $A$ be a general expansive matrix on $\mathbb{R}^{n}$.
Motivated by \cite{jwyyz23,llll21,qzl18,sll19,zql17}, in this article, we will introduce the anisotropic Musielak-Orlicz-Lorentz Hardy space $H^{\varphi,q}_A(\mathbb{R}^{n})$
which coincides with the known anisotropic weak Musielak-Orlicz Hardy space $H^{\varphi,\infty}_A(\mathbb{R}^{n})$ when $q=\infty$ (see \cite{qzl18,sll19,zql17}),
and then establish atomic and molecular characterizations of $H^{\varphi,q}_A(\mathbb{R}^{n})$.
As applications, we prove the boundedness of anisotropic Calder\'{o}n-Zygmund operators on $H^{\varphi,q}_A(\mathbb{R}^{n})$
when $q\in(0,\infty)$ and from the anisotropic Musielak-Orlicz Hardy space $H^{\varphi}_A(\mathbb{R}^{n})$ to $H^{\varphi,\infty}_A(\mathbb{R}^{n})$
in the critical case. The ranges of all the
exponents under consideration are the best possible admissible ones which particularly improve all the known corresponding results of
$H^{\varphi,\infty}_A(\mathbb{R}^{n})$ in \cite{sll19,zql17} via widening the original assumption $0<p^-_{\varphi}\leq p^+_{\varphi}\leq1$
into the full range $0<p^-_{\varphi}\leq p^+_{\varphi}<\infty$, and all the results when $q\in(0,\infty)$ are new
and generalize from isotropic setting to anisotropic frameworks.

To state the main results of this article, we first recall some necessary concepts.
Recall that a function $\Phi:[0,\infty)\rightarrow[0,\infty)$ is called an \emph{Orlicz function} if it is non-decreasing,
$\Phi(0)=0$, $\Phi(t)>0$ for any $t\in(0,\infty)$, and $\lim_{t\rightarrow\infty}\Phi(t)=0$ (see e.g. \cite{m83,ylk17}).
Now we present the Musielak-Orlicz function considered in this article as follows (see, for instance, \cite[Definition 1.1.4]{ylk17}).
A function $\varphi:\rn\times[0,\infty)\rightarrow[0,\infty)$ is called a \emph{Musielak-Orlicz function} if, for any
$x\in\rn$, $\varphi(x,\cdot)$ is an Orlicz function and, for any $t\in[0,\infty)$, $\varphi(\cdot,t)$
is a measurable function on $\rn$. Furthermore, the function $\varphi$ is said to be of \emph{uniformly upper (resp., lower)} type
$p\in(0,\infty)$ if there exists a positive constant $C_{(p)}$, depending
on $p$, such that, for any $x\in\rn$, $t\in[0,\infty)$, and $s\in[1,\infty)$ (resp., $s\in[0,1]$),
$\varphi(x,st)\leq C_{(p)}s^p\varphi(x,t)$. Let
\begin{align}\label{e1.1}
i(\varphi):=\sup\{p\in(0,\infty):\varphi\ {\rm{is\ of\ uniformly\ lower\ type}}\ p\}
\end{align}
and
\begin{align}\label{e1.2}
I(\varphi):=\inf\{p\in(0,\infty):\varphi\ {\rm{is\ of\ uniformly\ upper\ type}}\ p\}.
\end{align}
In the remainder of the whole article, we always assume that any Musielak-Orlicz function $\varphi$ satisfies that, for
any fixed $x\in\rn$, $\varphi(x,\cdot)$ is continuous and strictly increasing (see, for instance, \cite[Remark 2.2]{jwyyz23}).

Next we recall the notion of expansive dilations
on $\rn$ (see, for instance, \cite[p.\,5]{b03}). A real $n\times n$ matrix $A$ is called an {\it
expansive dilation}, shortly a {\it dilation}, if
$\min_{\lz\in\sz(A)}|\lz|>1$, where $\sz(A)$ denotes the set of
all {\it eigenvalues} of $A$. Let $\lz_-$ and $\lz_+$ be two {\it positive numbers} satisfying that
$$1<\lz_-<\min\{|\lz|:\ \lz\in\sz(A)\}\le\max\{|\lz|:\
\lz\in\sz(A)\}<\lz_+.$$ In the case when $A$ is diagonalizable over
$\mathbb C$, we can even take $$\lz_-:=\min\{|\lz|:\ \lz\in\sz(A)\} \ \ \mathrm{and} \ \
\lz_+:=\max\{|\lz|:\ \lz\in\sz(A)\}.$$
From \cite[p.\,5, Lemma 2.2]{b03}, we know that, for a given dilation $A$,
there exist a number $r\in(1,\,\fz)$ and an open ellipsoid $\Delta:=\{x\in\rn:\,|Px|<1\}$, where $P$ is some non-degenerate $n\times n$ matrix, such that $\Delta\subset r\Delta\subset A\Delta,$ and one can and
do additionally
assume that $|\Delta|=1$, where $|\Delta|$ denotes the
$n$-dimensional Lebesgue measure of the set $\Delta$. Let
$B_k:=A^k\Delta$ for $k\in \zz.$ Then $B_k$ is open, $B_k\subset
rB_k\subset B_{k+1}$ and $|B_k|=b^k$, here and hereafter, $b:=|\det A|$.
An ellipsoid $x+B_k$ for some $x\in\rn$ and $k\in\zz$ is called a {\it dilated ball}.
Denote by $\mathfrak{B}$ the set of all such dilated balls, namely,
\begin{equation}\label{e1.3}
\mathfrak{B}(\rn):=\{x+B_k:\ x\in \rn,\,k\in\zz\}.
\end{equation}
Throughout the whole paper, let $\tau$ be the {\it smallest integer} such that $2B_0\subset A^\tau B_0$,
and, for any subset $E$ of $\rn$, let $E^\complement:=\rn\setminus E$. Then,
for all $k,\,j\in\zz$ with $k\le j$, it holds true that
\begin{equation}\label{e1.4}
B_k+B_j\subset B_{j+\tau}
\end{equation}
and
\begin{equation}\label{e1.5}
B_k+(B_{k+\tau})^\complement\subset(B_k)^\complement,
\end{equation}
where $E+F$ denotes the {\it algebraic sum} $\{x+y:\ x\in E,\,y\in F\}$
of  sets $E,\, F\subset \rn$.

\begin{definition}\label{d1.1}
 A \textit{quasi-norm}, associated with an
expansive matrix $A$, is a Borel measurable mapping
$\rho_{A}:\rr^{n}\to [0,\infty)$, for simplicity, denoted by
$\rho$, satisfying
\begin{enumerate}
\item[\rm{(i)}] $\rho(x)>0$ for all $x \in \rn\setminus\{ \vec 0_n\}$,
here and hereafter, $\vec 0_n$ denotes the origin of $\rn$;
\item[\rm{(ii)}] $\rho(Ax)= b\rho(x)$ for all $x\in \rr^{n}$, where, as above, $b:=|\det A|$;
\item[\rm{(iii)}] $ \rho(x+y)\le H\lf[\rho(x)+\rho(y)\r]$ for
all $x,\, y\in \rr^{n}$, where $H\in[1,\,\fz)$ is a constant independent of $x$ and $y$.
\end{enumerate}
\end{definition}
\begin{remark}
\begin{enumerate}
\item[\rm{(i)}]
In the standard dyadic case $A:=2{\rm I}_{n\times n}$, $\rho(x):=|x|^n$ for all $x\in\rn$ is
an example of homogeneous quasi-norms associated with $A$, where
$|\cdot|$ always denotes the {\it Euclidean norm} in $\rn$.
\item[\rm{(ii)}]
From \cite[p.\,6, Lemma 2.4]{b03}, we know that all homogeneous quasi-norms associated with a given dilation
$A$ are equivalent. Hence, for a given expansive dilation $A$, in what follows, for simplicity, we
always use the {\it{step homogeneous quasi-norm}} $\rho$ defined by setting,  for all $x\in\rn$,
\begin{equation*}
\rho(x):=\sum_{k\in\zz}b^k\chi_{B_{k+1}\setminus B_k}(x)\ {\rm
if} \ x\ne \vec 0_n,\hs {\mathrm {or\ else}\hs } \rho(\vec 0_n):=0.
\end{equation*}
Obviously, $(\rn,\, \rho,\, dx)$ is a space of homogeneous type in the sense of Coifman and Weiss \cite{cw71,cw77},
where $dx$ denotes the {\it $n$-dimensional Lebesgue measure}.
\end{enumerate}
\end{remark}

Next we recall the definition of the anisotropic Muckenhoupt
condition (see, for instance, \cite[Definition 2]{lyy14}).
\begin{definition}\label{d1.2}
Let $p\in[1,\infty)$. A function $\varphi:\rn\times[0,\infty)\rightarrow[0,\infty)$ is said to satisfy the \emph{uniform anisotropic Muckenhoupt
condition} $\mathbb{A}_p(A)$, denoted by $\varphi\in\mathbb{A}_p(A)$, if there exists a
positive constant $C$ such that, for any $t\in(0,\infty)$, when $p\in(1,\infty)$,
$$[\varphi]_{\mathbb{A}_p(A)}:=\sup_{x\in\rn}\sup_{k\in\zz}\lf[\frac{1}{b^k}\int_{x+B_k}\vz(x,\,t)\,dx\r]
\lf\{\frac{1}{b^k}\int_{x+B_k}\frac{1}{[\vz(x,\,t)]^{\frac{1}{p-1}}}\,dx\r\}^{p-1}\le C$$
and, when $p=1$,
$$[\varphi]_{\mathbb{A}_1(A)}:=\sup_{x\in\rn}\sup_{k\in\zz}\lf[\frac{1}{b^k}\int_{x+B_k} \vz(x,\,t)\,dx\r]\lf\{\esup_{x\in x+B_k} \frac{1}{\vz(x,\,t)}\r\}\le C.$$
\end{definition}

Define $\mathbb{A}_\infty(A):=\bigcup_{p\in[1,\infty)}\mathbb{A}_p(A)$ and
\begin{equation}\label{e1.6}
q(\varphi):=\inf\lf\{p\in[1,\infty):\varphi\in\mathbb{A}_p(A)\r\}.
\end{equation}
Moreover, assume that $\vz$ is a Musielak-Orlicz function. Recall that the
\emph{Musielak-Orlicz space} $L^{\vz}(\rn)$ is defined to be the set of all measurable functions $f$ on $\rn$ such that
$$\|f\|_{L^\vz(\rn)}:=\inf\lf\{ \lz\in(0,\fz): \int_\rn \vz\lf(x, \frac{|f(x)|}{\lz}\r)dx\le 1\r\}$$
is finite (see, for instance, \cite[p.121]{k14}).

For any measurable subset $E$ of $\rn$, we denote the \emph{set} $\rn\setminus E$ by $E^\complement$
and its \emph{characteristic function} by $\mathbf{1}_{E}$.
The following Musielak-Orlicz-Lorentz space $L^{\varphi,q}(\rn)$ and the absolutely continuous part of
$L^{\varphi,\infty}(\rn)$ investigated by Jiao et al. in \cite[Definition 2.4]{jwxy21}.
\begin{definition}\label{d1.3}
Let $q\in(0,\infty]$ and $\vz$ be a Musielak-Orlicz function. The \emph{Musielak-Orlicz-Lorentz space} $L^{\varphi,q}(\rn)$ is
defined to be the set of all measurable functions $f$ on $\rn$ such that
\begin{eqnarray*}
\|f\|_{L^{\varphi,q}(\rn)}:=
\lf\{ \begin{array}{ll}
\lf[\dint_0^\infty \lambda^q \lf\|\mathbf{1}_{\{x\in\rn:|f(x)|>\lambda\}}\r\|^q_{L^{\varphi}(\rn)}
\dfrac {d\lambda}{\lambda}\r]^{\frac1{q}},& \ q\in(0,\,\fz),\\
\dsup_{\lambda\in(0,\,\fz)}\lf[\lambda\lf\|\mathbf{1}_{\{x\in\rn:|f(x)|>\lambda\}}\r\|_{L^{\varphi}(\rn)}\r], & \ q=\fz
\end{array}\r.
\end{eqnarray*}
is finite. Moreover, $f\in L^{\varphi,\infty}(\rn)$ is said to have an \emph{absolutely continuous quasi-norm} if
$$\lim_{\lambda\rightarrow\infty}\lf\|f\mathbf{1}_{\{x\in\rn:|f(x)|>\lambda\}}\r\|_{L^{\varphi,\infty}(\rn)}=0.$$
Furthermore, the \emph{absolutely continuous part} $\mathcal{L}^{\varphi,\infty}(\rn)$ of $L^{\varphi,\infty}(\rn)$ is defined to be
the set of all the $f\in L^{\varphi,\infty}(\rn)$ having an absolutely continuous quasi-norm.
\end{definition}

In what follows, we denote by $\cs(\rn)$ the \emph{space of all Schwartz functions} and $\cs'(\rn)$
its \emph{dual space} (namely, the \emph{space of all tempered distributions}).
Moreover, For any $m\in\zz_+$, define $\cs_m(\rn)$ as
\begin{eqnarray*}
\cs_m(\rn):=\{\varphi\in\cs(\rn):\ \|\varphi\|_{\cs_m(\rn)}\leq1\},
\end{eqnarray*}
where
\begin{eqnarray*}
\|\varphi\|_{\cs_m(\rn)}:=\sup_{\{\alpha\in\zz^n_+:|\alpha|\leq m+1\}}\,\sup_{x\in\rn}
\lf[\lf|\partial^\az_x\varphi(x)\r|\lf[1+\rho(x)\r]^{(m+2)(n+1)}\r]
\end{eqnarray*}
and, for any $\alpha:=(\alpha_1,\ldots,\alpha_n)\in\zz_+^n:=(\zz_+)^n$,
$\partial_x^\alpha:=(\frac{\partial}{\partial x_1})^{\alpha_1}\cdots(\frac{\partial}{\partial x_n})^{\alpha_n}$.

Now, we introduce the anisotropic Musielak-Orlicz-Lorentz Hardy space $H^{\varphi,q}_A(\mathbb{R}^{n})$ via the
non-tangential grand maximal function.
\begin{definition}\label{d1.4}
Let $\vz\in\mathbb{A}_p(A)$ be a Musielak-Orlicz function with $0<i(\varphi)\leq I(\varphi)<\infty$, where where $i(\varphi)$ and $I(\varphi)$ are as in
\eqref{e1.2} and \eqref{e1.1}, $q\in(0,\infty]$, and $m\in\zz_+$. Then, for any $f\in\cs'(\rn)$, the \emph{non-tangential grand maximal function} $f^*_m$ of $f$ is defined by setting,
for any $x\in\rn$,
$$f^*_m(x):=\sup_{\varphi\in \cs_m(\rn)}\,\sup_{k\in\zz}\,\sup_{y\in x+B_k}|f\ast\varphi_k(y)|,$$
here and thereafter, for any $\varphi\in \cs(\rn)$ and $k\in\zz$, $\varphi_k(\cdot):= b^{-k}\varphi(A^{-k}\cdot)$. Moreover, let
\begin{equation}\label{e1.7}
m(\varphi):=\max\lf\{\lf\lfloor\lf[\frac{q(\varphi)}{i(\varphi)}-1\r]\frac{\ln b}{\ln \lambda_-}\r\rfloor,0\r\}
\end{equation}
and denote $f^*_{m(\varphi)}$ simply by $f^*$, where $q(\varphi)$ and $i(\varphi)$ are as in \eqref{e1.6} and \eqref{e1.1}.
Furthermore, the \emph{anisotropic Musielak-Orlicz-Lorentz Hardy space} $H^{\varphi,q}_{A,m}(\mathbb{R}^{n})$
is defined by setting
$$H^{\varphi,q}_{A,m}(\mathbb{R}^{n}):=\lf\{f\in \cs'(\rn):\|f\|_{H^{\varphi,q}_{A,m}(\mathbb{R}^{n})}:=\lf\|f^*_m\r\|_{L^{\varphi,q}(\rn)}<\infty\r\}.$$
Meanwhile, the space $\mathcal{H}^{\varphi,q}_{A,m}(\mathbb{R}^{n})$ is defined to be the set of all the $f\in\cs'(\rn)$ such that
$f^*_m\in\mathcal{L}^{\varphi,\infty}(\rn)$.
In addition, we denote $H^{\varphi,q}_{A,m(\varphi)}(\mathbb{R}^{n})$ and $\mathcal{H}^{\varphi,q}_{A,m(\varphi)}(\mathbb{R}^{n})$, respectively, simply by
$H^{\varphi,q}_{A}(\mathbb{R}^{n})$ and $\mathcal{H}^{\varphi,q}_{A}(\mathbb{R}^{n})$, where $m(\varphi)$ is as in \eqref{e1.7}.
\end{definition}

Let the anisotropic $(\varphi,r,s,\ez)$-molecule be as in Definition \ref{d2.2} below.
Then we present the reconstruction theorem of $H^{\varphi,q}_{A,m}(\mathbb{R}^{n})$ via molecules.
\begin{theorem}\label{t1.1}
Let $\vz\in\mathbb{A}_\infty(A)$ be a Musielak-Orlicz function with $0<i(\varphi)\leq I(\varphi)<\infty$,
$q\in(0,\infty]$, $r\in(\max\{q(\varphi),I(\varphi)\},\infty]$, $s\in[m(\varphi),\infty)\cap\zz_+$, $m\in[s,\infty)\cap\zz_+$,
and $\varepsilon\in(\frac{\ln b}{\ln \lambda_-}+s+1,\infty)$, where $i(\varphi)$, $I(\varphi)$, $q(\varphi)$, and $m(\varphi)$ are, respectively, as in
\eqref{e1.1}, \eqref{e1.2}, \eqref{e1.6}, and \eqref{e1.7}. Assume that $c\in(0,1]$, $C_0\in[1,\infty)$, and $C_1\in(0,\infty)$.
Let $\{m_{k,j}\}_{k\in\zz,j\in\nn}$ be a sequence of anisotropic $(\varphi,r,s,\ez)$-molecules associated, respectively,
with the dilated balls $\{B_{k,j}\}_{k\in\zz,j\in\nn}\subset\mathfrak{B}(\rn)$, where $B_{k,j}:=x_{k,j}+B_{\ell_{k,j}}$
with $x_{k,j}\in\rn$ and $\ell_{k,j}\in\zz$, such that, for any $k\in\zz$, $\sum_{j\in\nn}\mathbf{1}_{cB_{k,j}}\leq C_0$,
and let $\lambda_{k,j}:=C_12^k\|\mathbf{1}_{B_{k,j}}\|_{L^{\varphi}(\rn)}$ for any $k\in\zz$ and $j\in\nn$. If
\begin{equation}\label{e1.8}
\lf[\sum_{k\in\zz}2^{kq}\lf\|\sum_{j\in\nn}\mathbf{1}_{B_{k,j}}\r\|^q_{L^\varphi(\rn)}\r]^{\frac1{q}}<\infty
\end{equation}
with the usual interpretation for $q=\infty$, then $f:=\sum_{k\in\zz}\sum_{j\in\nn}\lambda_{k,j}m_{k,j}$ converges
in $\cs'(\rn)$ and
\begin{equation}\label{e1.9}
\|f\|_{H^{\varphi,q}_{A,m}(\mathbb{R}^{n})}\lesssim\lf[\sum_{k\in\zz}2^{kq}\lf\|\sum_{j\in\nn}\mathbf{1}_{B_{i,j}}\r\|^q_{L^\varphi(\rn)}\r]^{\frac1{q}},
\end{equation}
where the implicit positive constant is independent of $\{m_{k,j}\}_{k\in\zz,j\in\nn}$ and $\{B_{k,j}\}_{k\in\zz,j\in\nn}$.
Moreover, if $q\in(0,\infty)$, then $f:=\sum_{k\in\zz}\sum_{j\in\nn}\lambda_{k,j}m_{k,j}$ converges
in $H^{\varphi,q}_{A,m}(\mathbb{R}^{n})$. If $q=\infty$ and
\begin{equation}\label{e1.10}
\lim_{|k|\rightarrow\infty}2^k\lf\|\sum_{j\in\nn}\mathbf{1}_{B_{k,j}}\r\|_{L^\varphi(\rn)}=0,
\end{equation}
then $f:=\sum_{k\in\zz}\sum_{j\in\nn}\lambda_{k,j}m_{k,j}$ converges in $\mathcal{H}^{\varphi,\infty}_{A,m}(\mathbb{R}^{n})$.
\end{theorem}

Let the anisotropic $(\varphi,r,s)$-atom be as in Definition \ref{d2.1} below.
By Theorem \ref{t1.1} and the facts that, for any $s\in[m(\varphi),\infty)$, an anisotropic $(\varphi,r,m(\varphi))$-atom
is also an anisotropic $(\varphi,r,m(\varphi),\ez)$-molecule and an anisotropic $(\varphi,r,s)$-atom
is also an anisotropic $(\varphi,r,m(\varphi))$-atom,
we have the following reconstruction theorem of $H^{\varphi,q}_{A}(\mathbb{R}^{n})$ via atoms. We omit the details here.
\begin{theorem}\label{t1.2}
Let $\vz\in\mathbb{A}_\infty(A)$ be a Musielak-Orlicz function with $0<i(\varphi)\leq I(\varphi)<\infty$,
$q\in(0,\infty]$, $r\in(\max\{q(\varphi),I(\varphi)\},\infty]$, and $s\in[m(\varphi),\infty)\cap\zz_+$,
where $i(\varphi)$, $I(\varphi)$, $q(\varphi)$, and $m(\varphi)$ are, respectively, as in
\eqref{e1.1}, \eqref{e1.2}, \eqref{e1.6}, and \eqref{e1.7}. Assume that $c\in(0,1]$, $C_0\in[1,\infty)$, and $C_1\in(0,\infty)$.
Let $\{a_{k,j}\}_{k\in\zz,j\in\nn}$ be a sequence of anisotropic $(\varphi,r,s)$-atoms associated, respectively,
with the dilated balls $\{B_{k,j}\}_{k\in\zz,j\in\nn}\subset\mathfrak{B}(\rn)$, where $B_{k,j}:=x_{k,j}+B_{\ell_{k,j}}$
with $x_{k,j}\in\rn$ and $\ell_{k,j}\in\zz$,
such that, for any $k\in\zz$, $\sum_{j\in\nn}\mathbf{1}_{cB_{k,j}}\leq C_0$,
and let $\lambda_{k,j}:=C_12^k\|\mathbf{1}_{B_{k,j}}\|_{L^{\varphi}(\rn)}$ for any $k\in\zz$ and $j\in\nn$. If
$$\lf[\sum_{k\in\zz}2^{kq}\lf\|\sum_{j\in\nn}\mathbf{1}_{B_{k,j}}\r\|^q_{L^\varphi(\rn)}\r]^{\frac1{q}}<\infty$$
with the usual interpretation for $q=\infty$, then $f:=\sum_{k\in\zz}\sum_{j\in\nn}\lambda_{k,j}a_{k,j}$ converges
in $\cs'(\rn)$ and
$$\|f\|_{H^{\varphi,q}_{A}(\mathbb{R}^{n})}\lesssim\lf[\sum_{k\in\zz}2^{kq}\lf\|\sum_{j\in\nn}\mathbf{1}_{B_{k,j}}\r\|^q_{L^\varphi(\rn)}\r]^{\frac1{q}},$$
where the implicit positive constant is independent of $\{a_{k,j}\}_{k\in\zz,j\in\nn}$ and $\{B_{k,j}\}_{k\in\zz,j\in\nn}$.
Moreover, if $q\in(0,\infty)$, then $f:=\sum_{k\in\zz}\sum_{j\in\nn}\lambda_{k,j}a_{k,j}$ converges
in $H^{\varphi,q}_{A}(\mathbb{R}^{n})$. If $q=\infty$ and
$$\lim_{|k|\rightarrow\infty}2^k\lf\|\sum_{j\in\nn}\mathbf{1}_{B_{k,j}}\r\|_{L^\varphi(\rn)}=0,$$
then $f:=\sum_{k\in\zz}\sum_{j\in\nn}\lambda_{k,j}a_{k,j}$ converges in $\mathcal{H}^{\varphi,\infty}_{A}(\mathbb{R}^{n})$.
\end{theorem}

\begin{remark}\label{r1.1}
Let all notation be as in Theorems \ref{t1.1} and \ref{t1.2}.
\begin{enumerate}
\item[\rm{(i)}] Let $q=\infty$. In this case, Theorems \ref{t1.1} and \ref{t1.2} improve the reconstruction part
of \cite[Theorem 2.8]{sll19} and \cite[Theorem 1]{zql17}, respectively, via weakening the assumption that $\varphi$ is a growth function to
the full range $0<i(\varphi)<I(\varphi)<\infty$.
\item[\rm{(ii)}] Let $p=\infty$ and, for any $(x,t)\in\rn\times[0,\infty)$, $\varphi(x,t):=t^p$. In this case,
$H^{\varphi,q}_{A}(\mathbb{R}^{n})$ is just the anisotropic Hardy-Lorentz space $H^{p,q}_{A}(\mathbb{R}^{n})$ and hence, in this case,
Theorems \ref{t1.1} and \ref{t1.2} coincide with the best known results of the reconstruction theorem via molecules and atoms
of $H^{p,q}_{A}(\mathbb{R}^{n})$, respectively (see, for instance, \cite[Theorems 3.9 and 3.6]{lyy16}).
\item[\rm{(iii)}] Let $A:=2\mathrm{I}_{n\times n}$, where $\mathrm{I}_{n\times n}$ always denotes the $n\times n$
unit matrix. In this case, $H^{\varphi,q}_{A}(\mathbb{R}^{n})$ is just the Musielak-Orlicz-Lorentz Hardy space $H^{\varphi,q}(\mathbb{R}^{n})$ and hence, in this case,
Theorems  \ref{t1.1} and \ref{t1.2} are anisotropic analogue of \cite[Theorems 3.6 and 3.29]{jwyyz23}, respectively.
\end{enumerate}
\end{remark}

Next we establish the following decomposition theorem of $H^{\varphi,q}_{A}(\mathbb{R}^{n})$ via atoms.
\begin{theorem}\label{t1.3}
Let $\vz\in\mathbb{A}_\infty(A)$ be a Musielak-Orlicz function with $0<i(\varphi)\leq I(\varphi)<\infty$,
$q\in(0,\infty]$, $s\in\zz_+$, and $m\geq s\geq\lfloor\frac{\ln b}{\ln\lambda_-}\frac{q(\varphi)}{i(\varphi)}\rfloor$,
where $i(\varphi)$, $I(\varphi)$, and $q(\varphi)$ are, respectively, as in \eqref{e1.1}, \eqref{e1.2}, and \eqref{e1.6}.
Then there exist positive constants $C, C_{(n)}$ (depending only on $n$), and $C_1$ such that, for any $f\in H^{\varphi,q}_{A,m}(\mathbb{R}^{n})$,
there exists a sequence $\{a_{k,j}\}_{k\in\zz,j\in\nn}$ of anisotropic $(\varphi,\infty,s)$-atoms supported, respectively,
in the dilated balls $\{B_{k,j}\}_{k\in\zz,j\in\nn}\subset\mathfrak{B}(\rn)$, where $B_{k,j}:=x_{k,j}+B_{\ell_{k,j}}$
with $x_{k,j}\in\rn$ and $\ell_{k,j}\in\zz$, such that $\sum_{j\in\nn}\mathbf{1}_{B_{k,j}}\leq C_{(n)}$,
\begin{equation}\label{e1.11}
f=\sum_{k\in\zz}\sum_{j\in\nn}C_12^k\lf\|\mathbf{1}_{B_{k,j}}\r\|_{L^\varphi(\rn)}a_{k,j}
\quad \mathrm{in} \ \cs'(\rn),
\end{equation}
and
\begin{equation}\label{e1.12}
\lf[\sum_{k\in\zz}2^{kq}\lf\|\sum_{j\in\nn}\mathbf{1}_{B_{k,j}}\r\|^q_{L^\varphi(\rn)}\r]^{\frac1{q}}
\leq C\|f\|_{H^{\varphi,q}_{A,m}(\mathbb{R}^{n})}
\end{equation}
with the usual interpretation for $q=\infty$. When $f\in\mathcal{H}^{\varphi,\infty}_{A,m}(\mathbb{R}^{n})$, then
\begin{equation}\label{e1.13}
\lim_{|k|\rightarrow\infty}2^k\lf\|\sum_{j\in\nn}\mathbf{1}_{B_{k,j}}\r\|_{L^\varphi(\rn)}=0.
\end{equation}
Moreover, when $f\in H^{\varphi,q}_{A,m}(\mathbb{R}^{n})$ with $q\in(0,\infty)$ [resp., $f\in\mathcal{H}^{\varphi,\infty}_{A,m}(\mathbb{R}^{n})$], then
\eqref{e1.11} holds true in $H^{\varphi,q}_{A,m}(\mathbb{R}^{n})$ [resp., $H^{\varphi,\infty}_{A,m}(\mathbb{R}^{n})$].
\end{theorem}

Then, as an application of Theorems \ref{t1.2} and \ref{t1.3}, we obtain that the space $H^{\varphi,q}_{A,m}(\mathbb{R}^{n})$
is independent of the choice of $m\geq s\geq\lfloor\frac{\ln b}{\ln\lambda_-}\frac{q(\varphi)}{i(\varphi)}\rfloor$.
\begin{theorem}\label{t1.4}
Let $\vz\in\mathbb{A}_\infty(A)$ be a Musielak-Orlicz function with $0<i(\varphi)\leq I(\varphi)<\infty$ and $q\in(0,\infty]$,
where $i(\varphi)$ and $I(\varphi)$ are, respectively, as in \eqref{e1.1} and \eqref{e1.2}. Then, for any
$m\geq s\geq\lfloor\frac{\ln b}{\ln\lambda_-}\frac{q(\varphi)}{i(\varphi)}\rfloor$, the anisotropic Musielak-Orlicz-Lorentz Hardy spaces
$$H^{\varphi,q}_{A}(\mathbb{R}^{n})=H^{\varphi,q}_{A,m}(\mathbb{R}^{n}) \quad
\mathrm{and} \quad \mathcal{H}^{\varphi,\infty}_{A}(\mathbb{R}^{n})=\mathcal{H}^{\varphi,\infty}_{A,m}(\mathbb{R}^{n})$$
with equivalent quasi-norms.
\end{theorem}

\begin{remark}\label{r1.2}
Let all notation be as in Theorems \ref{t1.3} and \ref{t1.4}.
\begin{enumerate}
\item[\rm{(i)}] Let $q=\infty$. In this case, Theorem \ref{t1.3} improves the decomposition part
of \cite[Theorem 1]{zql17} via weakening the assumption that $\varphi$ is a growth function to
the full range $0<i(\varphi)<I(\varphi)<\infty$. Moreover, Theorem \ref{t1.4} generalizes  \cite[Remark 2.7]{lsll20}
via weakening the assumption that $\varphi$ is a growth function to
the full range $0<i(\varphi)<I(\varphi)<\infty$.
\item[\rm{(ii)}] Let $p=\infty$ and, for any $(x,t)\in\rn\times[0,\infty)$, $\varphi(x,t):=t^p$. In this case,
$H^{\varphi,q}_{A}(\mathbb{R}^{n})$ is just the anisotropic Hardy-Lorentz space $H^{p,q}_{A}(\mathbb{R}^{n})$ and hence, in this case,
Theorem \ref{t1.3} coincides with the best known result of the decomposition theorem
of $H^{p,q}_{A}(\mathbb{R}^{n})$ (see, for instance, \cite[Theorem 3.6]{lyy16}).
\item[\rm{(iii)}] Let $A:=2\mathrm{I}_{n\times n}$, where $\mathrm{I}_{n\times n}$ always denotes the $n\times n$
unit matrix. In this case, $H^{\varphi,q}_{A}(\mathbb{R}^{n})$ is just the Musielak-Orlicz-Lorentz Hardy space $H^{\varphi,q}(\mathbb{R}^{n})$ and hence, in this case,
Theorems  \ref{t1.3} and \ref{t1.4} are anisotropic analogue of \cite[Theorems 3.31 and 3.36]{jwyyz23}, respectively.
\end{enumerate}
\end{remark}

As some applications, we will establish the boundedness of anisotropic Calder\'{o}n-Zygmund operators on $H^{\varphi,q}_{A}(\mathbb{R}^{n})$
with $q\in(0,\infty)$ and, in the critical case, from $H^{\varphi}_{A}(\mathbb{R}^{n})$ to $H^{\varphi,\infty}_{A}(\mathbb{R}^{n})$.
Assume that the anisotropic Calder\'{o}n-Zygmund operators are defined as in Definition \ref{d4.1} below. Then we have the following results.
\begin{theorem}\label{t1.5}
Let $\vz\in\mathbb{A}_\infty(A)$ be a Musielak-Orlicz function with $0<i(\varphi)\leq I(\varphi)<\infty$ and $q\in(0,\infty)$,
where $i(\varphi)$ and $I(\varphi)$ are, respectively, as in \eqref{e1.1} and \eqref{e1.2}.
Let $\delta\in(0,\frac{\ln\lambda_-}{\ln b}]$ and $T$ be an anisotropic Calder\'{o}n-Zygmund operator with regularity $\delta$ as in Definition \ref{d4.1}.
Then we have the following conclusions.
\begin{enumerate}
\item[\rm{(i)}] If $\frac{i(\varphi)}{q(\varphi)}\in(\frac1{1+\delta},\infty)$, then $T$ can be extended to a bounded linear operator
on $H^{\varphi,q}_{A}(\mathbb{R}^{n})$ and, moreover,
there exists a positive constant $C_1$ such that, for any $f\in H^{\varphi,q}_{A}(\mathbb{R}^{n})$,
$$\lf\|T(f)\r\|_{H^{\varphi,q}_{A}(\mathbb{R}^{n})}\leq C_1\|f\|_{H^{\varphi,q}_{A}(\mathbb{R}^{n})}.$$
\item[\rm{(ii)}] If $\varphi\in\mathbb{A}_1(A)$ and $i(\varphi)=\frac1{1+\delta}$ is attainable, then $T$ can be extended to a bounded linear operator
from $H^{\varphi}_{A}(\mathbb{R}^{n})$ to $H^{\varphi,\infty}_{A}(\mathbb{R}^{n})$ and, moreover,
there exists a positive constant $C_2$ such that, for any $f\in H^{\varphi}_{A}(\mathbb{R}^{n})$,
$$\lf\|T(f)\r\|_{H^{\varphi,\infty}_{A}(\mathbb{R}^{n})}\leq C_2\|f\|_{H^{\varphi}_{A}(\mathbb{R}^{n})}.$$
\end{enumerate}
\end{theorem}

\begin{remark}\label{r1.3}
We adopt the same notation as in Theorem \ref{t1.5}.
\begin{enumerate}
\item[\rm{(i)}] Under the assumption that $I(\varphi)\in(0,1)$, it is easy to obtain that Theorem \ref{t1.5}(ii) coincides with \cite[Theorem 4.2]{sll19}.
Thus, Theorem \ref{t1.5}(ii) essentially improves the known results in \cite[Theorem 4.2]{sll19} via weakening the assumption $I(\varphi)\in(0,1)$ to
the full range $I(\varphi)\in(0,\infty)$.
Moreover, Theorem \ref{t1.5}(ii) establishes the boundedness of $T$ from $H^{\varphi}_{A}(\mathbb{R}^{n})$ to $H^{\varphi,\infty}_{A}(\mathbb{R}^{n})$
in the critical case that $\varphi\in\mathbb{A}_1(A)$ and $i(\varphi)=\frac1{1+\delta}$ is attainable.
\item[\rm{(ii)}] Let $p=\infty$ and, for any $(x,t)\in\rn\times[0,\infty)$, $\varphi(x,t):=t^p$. In this case,
$H^{\varphi,q}_{A}(\mathbb{R}^{n})$ is just the anisotropic Hardy-Lorentz space $H^{p,q}_{A}(\mathbb{R}^{n})$ and hence, in this case,
Theorem \ref{t1.5}(i) coincides with \cite[Theorem 6.16(ii)]{lyy16}.
\item[\rm{(iii)}] If $A:=2\mathrm{I}_{n\times n}$ and $\rho(x):=|x|^n$, then $\frac{\ln\lambda_-}{\ln b}=\frac{1}{n}$, $\delta\in(0,\frac{1}{n}]$,
and Theorem \ref{t1.5} coincides with \cite[Theorem 5.2]{jwyyz23} with $s=0$. Meanwhile, we remark that the range
of $\delta$ belongs to $(0,\frac{1}{n}]$ in Theorem \ref{t1.5}, whereas the range of the corresponding
parameter $\delta$ in \cite[Theorem 5.2]{jwyyz23} belongs to $(0,1]$. Indeed, the range of these two
parameters are essentially coincident since we use the (quasi-)norm $\rho(x)$ which
is reduced to $|x|^n$ when $A:=2\mathrm{I}_{n\times n}$.
\end{enumerate}
\end{remark}

The organization of this article is as follows. In Section \ref{s2}, we present some notions and some basic properties
of atoms and molecules associated with the dilated ball $B$ used in this article,
and then we give the proof of Theorem \ref{t1.1}. Section \ref{s3} is devoted to the proofs of Theorems \ref{t1.3} and \ref{t1.4}.
The proof of Theorem \ref{t1.5} is presented in Section \ref{s4}.

At the end of this section, we make some conventions on notation. Let $\nn:=\{1,\, 2,\,\ldots\}$ and $\zz_+:=\{0\}\cup\nn$.
For any $\alpha:=(\alpha_1,\ldots,\alpha_n)\in\zz_+^n:=(\zz_+)^n$, let $|\alpha|:=\alpha_1+\cdots+\alpha_n$ and
$\partial^\alpha:=(\frac{\partial}{\partial x_1})^{\alpha_1}\cdots(\frac{\partial}{\partial x_n})^{\alpha_n}$.
Throughout the whole article, we denote by $C$ a \emph{positive
constant} which is independent of the main parameters, but it may
vary from line to line. The \emph{symbol} $f\ls g$ means that $f\le Cg$. If $f\ls
g$ and $g\ls f$, then we write $f\sim g$.
For any $a\in\rr$, denote by $\lfloor a\rfloor$ the \emph{largest integer} not greater than $a$.
For any measurable subset $E$ of $\rn$, we denote the \emph{set} $\rn\setminus E$ by $E^\complement$
and its \emph{characteristic function} by $\mathbf{1}_{E}$.
Moreover, denote by $\cs(\rn)$ the \emph{space of all Schwartz functions} and $\cs'(\rn)$
its \emph{dual space} (namely, the \emph{space of all tempered distributions}).
Finally, for any given $q\in[1,\fz]$, we denote by $q'$
its \emph{conjugate exponent}, namely, $1/q + 1/q'= 1$.



\section{Proof of Theorem \ref{t1.1}}\label{s2}
\hskip\parindent
In this section, we will give the proof of Theorem \ref{t1.1}. To do this, let us recall some definitions. The following is the definition of the space $L^p_\varphi(\rn)$ as follows (see, for instance, \cite[Definition 3.1]{jwyyz23}).
Let $p\in(1,\infty]$, $\vz$ be a Musielak-Orlicz function, and $E$ a measurable set of $\rn$.
The space $L^p_\varphi(E)$ is defined to be the set of all the measurable functions $f$ on $\rn$
such that $\supp (f)\subset E$ and
\begin{eqnarray*}
L^p_\varphi(E):=
\lf\{ \begin{array}{ll}
\dsup_{t\in(0,\infty)}\lf[\frac1{\varphi(E,t)}\int_{E}|f(x)|^p\varphi(x,t)\,dx\r]^{\frac1{p}}<\infty,& \ p\in(1,\infty);\\
\|f\|_{L^\infty(E)}<\infty, & \ p=\fz.
\end{array}\r.
\end{eqnarray*}

The followings is the definitions of atoms and molecules associated with the dilated ball $B$
(see, for instance, \cite[Definitions 2.5 and 2.6]{sll19}, \cite[Definition 2.6]{lffy16} and \cite[Definition 30]{lyy14}).
\begin{definition}\label{d2.1}
Let $\vz$ be a Musielak-Orlicz function, $r\in(1,\infty]$, and $s\in\zz_+$.
A measurable function $a$ on $\rn$ is called an {\it{anisotropic $(\varphi,r,s)$-atom}} if
\begin{enumerate}
\item[\rm{(i)}]  $\supp a\subset B$, where $B\in\mathfrak{B}(\rn)$ and $\mathfrak{B}(\rn)$ is as in \eqref{e1.3};
\item[\rm{(ii)}] $\|a\|_{L^r_\varphi(B)}\le \|\mathbf{1}_B\|^{-1}_{L^\varphi(\rn)}$;
\item[\rm{(iii)}] $\int_\rn a(x)x^\alpha dx=0$ for any $\alpha\in \mathbb{Z}^n_+$ with $|\alpha|\leq s$.
\end{enumerate}
\end{definition}
\begin{definition}\label{d2.2}
Let $\vz$ be a Musielak-Orlicz function, $r\in(1,\infty]$, $\varepsilon\in(0,\infty)$, and $s\in\zz_+$. A measurable function $m$ is called an
{\it{anisotropic $(\varphi,r,s,\ez)$-molecule}} associated with some dilated ball $x_0+B\in\mathfrak{B}(\rn)$ if
\begin{enumerate}
\item[\rm{(i)}] for each
$j\in\zz_+$, $\|m\|_{L^r_\varphi(U_j(x_0+B))}\le b^{-j\varepsilon}\|\mathbf{1}_{x_0+B}\|^{-1}_{L^\varphi(\rn)}$, where $U_0(x_0+B):=x_0+B$ and, for each $j\in\nn$, $U_j(x_0+B):=x_0+(A^jB)\setminus(A^{j-1}B)$;
\item[\rm{(ii)}] for all $\alpha\in \mathbb{Z}^n_+$ with $|\alpha|\leq s$, $\int_\rn m(x)x^\alpha dx=0$.
\end{enumerate}
\end{definition}

The following lemma comes from \cite[p.2]{ylk17} (see also \cite[Remark 2.2]{jwyyz23}) and \cite[Proposition 2.1]{blyz08}.
\begin{lemma}\label{l2.1}
\begin{enumerate}
\item[\rm{(i)}] Let $\varphi$ be a Musielak-Orlicz function of uniformly lower type $p^-_\varphi\in(0,\infty)$
and of uniformly upper type $p^+_\varphi\in(0,\infty)$. Then $p^-_\varphi\leq p^+_\varphi$ and, for any
$p_1\in(-\infty,p^-_\varphi)$ and $p_2\in(p^+_\varphi,+\infty)$, $\varphi$ is also of uniformly lower type $p_1$ and of uniformly upper type $p_2$.
\item[\rm{(ii)}] Let $p\in[1,\infty)$ and $\varphi\in \mathbb{A}_p(A)$. Then there exists a positive constant $C$ such that, for any $x\in\rn$,
$k\in\zz$, $E\subset(x+B_k)$, and $t\in(0,\infty)$,
$$\frac{\varphi(x+B_k,t)}{\varphi(E,t)}\leq C\frac{|x+B_k|^p}{|E|^p}.$$
Moreover, if $1\leq p\leq q\leq\infty$, then it is easy to obtain that $\mathbb{A}_1(A)\subset \mathbb{A}_p(A)\subset\mathbb{A}_q(A)$.
\end{enumerate}
\end{lemma}

In what follows, we use $L^1_{\mathrm{loc}}(\rn)$ to denote the set of all locally integrable functions on $\rn$
and use $M_{\mathrm{HL}}$ to denote the \emph{Hardy--Littlewood maximal operator} which is defined by
setting, for any $f\in L^1_{\mathrm{loc}}(\rn)$ and $x\in\rn$,
$$M_{\mathrm{HL}}(f)(x):=\sup_{x\in B\in\mathfrak{B}(\rn)}\frac1{|B|}\int_B|f(y)|dy,$$
where $\mathfrak{B}(\rn)$ is as in \eqref{e1.3}. Then, we have the following conclusion (see, for instance, \cite[Proposition 2.1]{blyz08}).
\begin{lemma}\label{l2.2}
\begin{enumerate}
\item[\rm{(i)}] Let $p\in[1,\infty)$ and $\omega\in \mathbb{A}_p(A)$. Then there exists a positive constant $C$ such that, for any $x\in\rn$,
$k, m\in\zz$ with $k\leq m$,
$$C^{-1}b^{\frac{m-k}{p}}\leq\frac{\omega(x+B_m)}{\omega(x+B_k)}\leq Cb^{\frac{m-k}{p}}.$$
\item[\rm{(ii)}] Let $p\in(1,\infty)$. Then the Hardy--Littlewood maximal operator $M_{\mathrm{HL}}$ is bounded on $L^p_\omega(\rn)$
if and only if $\omega\in \mathbb{A}_p(A)$.
\end{enumerate}
\end{lemma}

To prove Theorem \ref{t1.1}, we also need the following technical lemma,
whose proof is similar to that of \cite[Lemma 3.18]{jwxy21}. We omit the details here.
In what follows, for any $p\in(0,\infty)$, $\varphi\in \mathbb{A}_\infty(A)$, and $\lambda\in(0,\infty)$, the space
$L^p_{\varphi(\cdot,\lambda)}(\rn)$ is defined to be the set of all the measurable functions $f$ on $\rn$ such that
$$L^p_{\varphi(\cdot,\lambda)}(\rn):=\lf[\int_\rn|f(x)|^p\varphi(x,\lambda)dx\r]^{\frac1{p}}<\infty.$$

\begin{lemma}\label{l2.3}
Let $\vz\in\mathbb{A}_\infty(A)$ be a Musielak-Orlicz function, $b\in(0,\infty)$, $r\in(b,\infty)$, and $p\in[b,r)$.
Then there exists a positive constant $C$ such that, for any sequence $\{B^{(k)}\}_{k\in\nn}\subset\mathfrak{B}(\rn)$ of dilated
balls, numbers $\{\lambda_k\}_{k\in\nn}\subset\ccc$ and measurable functions $\{a_k\}_{k\in\nn}$ satisfying that, for each
$k\in\nn$, $\supp (a_k)\subset B^{(k)}$ and $\|a_k\|_{L^r_\varphi(B^{(k)})}\leq1$, and any $\lambda\in(0,\infty)$, it holds true that
$$\lf\|\lf(\sum_{k\in\nn}\lf|\lambda_ka_k\r|^b\r)^{\frac1{b}}\r\|_{L^p_{\varphi(\cdot,\lambda)}(\rn)}\leq C
\lf\|\lf(\sum_{k\in\nn}\lf|\lambda_k\mathbf{1}_{B^{(k)}}\r|^b\r)^{\frac1{b}}\r\|_{L^p_{\varphi(\cdot,\lambda)}(\rn)}.$$
\end{lemma}

Moreover, by \cite[Assumption 4.1]{syy22} and checking the proof of \cite[p.75, Theorem 8.3(i)]{syy22},
we have the following vector-valued Fefferman-Stein inequality of weighted Lebesgue spaces (see also \cite[Lemma 2.21]{lfy15}), the detail being omitted.
\begin{lemma}\label{l2.4}
Let $p\in(1,\infty)$, $\vz\in\mathbb{A}_p(A)$ be a Musielak-Orlicz function, and $r\in(1,\infty]$.
Then there exists a positive constant $C$ such that, for any $\lambda\in(0,\infty)$ and any sequence $\{f_k\}_{k\in\nn}\subset L^p_{\varphi(\cdot,\lambda)}(\rn)$,
$$\lf\|\lf\{\sum_{k\in\nn}\lf[M_{\mathrm{HL}}(f_k)\r]^r\r\}^{\frac1{r}}\r\|_{L^p_{\varphi(\cdot,\lambda)}(\rn)}\leq C
\lf\|\lf(\sum_{k\in\nn}|f_k|^r\r)^{\frac1{r}}\r\|_{L^p_{\varphi(\cdot,\lambda)}(\rn)}$$
with the usual modification made when $r=\infty$.
\end{lemma}

Then, by Lemma \ref{l2.4} and the fact that, for any dilated ball
$B\in\mathfrak{B}(\rn)$, where $\mathfrak{B}(\rn)$ is as in \eqref{e1.3}, and $r\in(0,1)$, $\mathbf{1}_{A^iB}\leq b^{i/r}[M_{\mathrm{HL}}(\mathbf{1}_B)]^{1/r}$,
we have the following lemma, the details are omitted.
\begin{lemma}\label{l2.5}
Let $p\in(1,\infty)$, $\vz\in\mathbb{A}_p(A)$ be a Musielak-Orlicz function, $i\in\zz_+$, and $r\in(1,\infty)$.
Then there exists a positive constant $C$ such that, for any dilated balls $\{B^{(k)}\}_{k\in\nn}\in\mathfrak{B}(\rn)$ and any $\lambda\in(0,\infty)$,
$$\lf\|\sum_{k\in\nn}\mathbf{1}_{A^iB^{(k)}}\r\|_{L^p_{\varphi(\cdot,\lambda)}(\rn)}\leq C b^{\frac{i}{r}}
\lf\|\sum_{k\in\nn}\mathbf{1}_{B^{(k)}}\r\|_{L^p_{\varphi(\cdot,\lambda)}(\rn)}.$$
\end{lemma}

The following two lemmas are just from \cite[Lemmas 3.20 and 3.27]{jwyyz23}, respectively.
\begin{lemma}\label{l2.6}
Let $\varphi$ be a Musielak-Orlicz function of uniformly lower type $p^-_\varphi\in(0,\infty)$ and of uniformly upper type $p^+_\varphi\in(0,\infty)$.
Then, for any $f\in L^\varphi(\rn)$ with $\|f\|_{L^\varphi(\rn)}\neq0$,
$$\int_\rn\varphi\lf(x,\frac{|f(x)|}{\|f\|_{L^\varphi(\rn)}}\r)dx=1.$$
\end{lemma}

\begin{lemma}\label{l2.7}
Let $p\in(0,\infty]$ and $\alpha,\beta,\widetilde{\alpha},\widetilde{\beta}\in(0,\infty)$. Then there exists a positive
constant $C$ such that, for any $\{\mu_k\}_{k\in\zz}\subset[0,\infty)$,
$$\lf[\sum_{k_0\in\zz}2^{-k_0p\alpha}\lf(\sum^{k_0-1}_{k=-\infty}2^{k\alpha\beta}\mu_k^\beta\r)^{\frac{p}{\beta}}
\r]^{\frac1{p}}\leq C\lf(\sum_{k\in\zz}\mu_k^p\r)^{\frac1{p}}$$
and
$$\lf[\sum_{k_0\in\zz}2^{k_0p\widetilde{\alpha}}\lf(\sum^\infty_{k=k_0}2^{-k\widetilde{\alpha}\widetilde{\beta}}
\mu_k^{\widetilde{\beta}}\r)^{\frac{p}{\widetilde{\beta}}}\r]^{\frac1{p}}\leq C\lf(\sum_{k\in\zz}\mu_k^p\r)^{\frac1{p}}$$
with the usual interpretation for $p=\infty$.
\end{lemma}

The next conclusion asserts that radial and nontangential grand maximal functions
are pointwise equivalent (see, for instance, \cite[Proposition 3.10]{b03}).
\begin{lemma}\label{l2.8}
Let $m\in\zz_+$. Then there exists a positive constant $C$, depending only on
$m$, such that, for any $f\in\cs'(\rn)$ and $x\in\rn$,
$$M^{(0)}_m(f)(x)\leq f^*_m(x)\leq CM^{(0)}_m(f)(x),$$
where, for any $x\in\rn$,
$$M^{(0)}_m(f)(x):=\sup_{\varphi\in\cs_m(\rn)}\sup_{t\in(0,\infty)}\lf|f\ast\varphi_t(x)\r|.$$
\end{lemma}

The following lemma is just from \cite[Lemma 3.3]{jwxy21} and \cite[p.2]{ylk17} (see also \cite[Lemma 2.4]{ins22}).
\begin{lemma}\label{l2.9}
\begin{enumerate}
\item[\rm{(i)}] Let $\varphi$ be a Musielak-Orlicz function. Then, for any $\theta\in(0,\infty)$ and $f\in L^{\varphi_\theta}(\rn)$,
$$\lf\||f|^\theta\r\|_{L^\varphi(\rn)}=\lf\|f\r\|^\theta_{L^{\varphi_\theta}(\rn)},$$
here and thereafter, for any $\theta\in(0,\infty)$, $t\in[0,\infty)$, and $x\in\rn$, $\varphi_\theta(x,t):=\varphi(x,t^\theta)$.
\item[\rm{(ii)}] Let $\varphi$ be a Musielak-Orlicz function of uniformly upper (resp., lower) type $p^+_\varphi\in(0,\infty)$ [resp., $p^-_\varphi\in(0,\infty)$].
Then, for any $\theta\in(0,\infty)$, $\varphi_\theta$ is a Musielak-Orlicz function of uniformly upper (resp., lower) type $\theta p^+_\varphi$
[resp., $\theta p^-_\varphi$] and, moreover, $i(\varphi_\theta)=\theta i(\varphi)$, $I(\varphi_\theta)=\theta I(\varphi)$, and $q(\varphi_\theta)=q(\varphi)$.
\end{enumerate}
\end{lemma}

The following lemma comes from \cite[Corollary 3.2.5]{hh19}.
\begin{lemma}\label{l2.10}
Let $\varphi$ be a Musielak-Orlicz function of uniformly lower type $p^-_\varphi\in[1,\infty)$.
Then there exists a positive constant $C$ such that, for any sequence $\{f_k\}_{k\in\zz}$,
$$\lf\|\sum_{k\in\zz}f_k\r\|_{L^\varphi(\rn)}\leq C\sum_{k\in\zz}\lf\|f_k\r\|_{L^\varphi(\rn)}.$$
\end{lemma}

The following vector-valued Fefferman-Stein inequality of Musielak-Orlicz spaces $L^\varphi(\rn)$
is a corollary of both \cite[Lemma 3.6]{lfy15} and the definition of $\|\cdot\|_{L^\varphi(\rn)}$, the details are omitted here.
\begin{lemma}\label{l2.11}
Let $p\in(1,\infty]$ and $\vz\in\mathbb{A}_\infty(A)$ be a Musielak-Orlicz function
with $q(\varphi)<i(\varphi)\leq I(\varphi)<\infty$, where $q(\varphi)$, $i(\varphi)$, and $I(\varphi)$
are, respectively, as in \eqref{e1.6}, \eqref{e1.1}, and \eqref{e1.2}.
Then there exists a positive constant $C$ such that, for any sequence $\{f_k\}_{k\in\nn}\subset L^\varphi(\rn)$,
$$\lf\|\lf\{\sum_{k\in\nn}\lf[M_{\mathrm{HL}}(f_k)\r]^p\r\}^{\frac1{p}}\r\|_{L^\varphi(\rn)}\leq C
\lf\|\lf(\sum_{k\in\nn}|f_k|^p\r)^{\frac1{p}}\r\|_{L^\varphi(\rn)}.$$
\end{lemma}
Then, by Lemmas \ref{l2.9} and \ref{l2.11}, the fact that, for any dilated ball
$B\in\mathfrak{B}(\rn)$, $i\in\zz$, and $\mu\in(0,\min\{1,\frac{i(\varphi)}{q(\varphi)}\})$,
where $\mathfrak{B}(\rn)$ is as in \eqref{e1.3}, $\mathbf{1}_{A^iB}\leq b^{i/\mu}[M_{\mathrm{HL}}(\mathbf{1}_B)]^{1/\mu}$,
and similar to the proof of \cite[Lemma 3.11]{jwyyz23}, we have the following lemma, the details are omitted.
\begin{lemma}\label{l2.12}
Let $\mu\in(0,\min\{1,\frac{i(\varphi)}{q(\varphi)}\})$, $i\in\zz$, and $\vz$ be a Musielak-Orlicz function
with $0<i(\varphi)\leq I(\varphi)<\infty$, where $q(\varphi)$, $i(\varphi)$, and $I(\varphi)$
are, respectively, as in \eqref{e1.6}, \eqref{e1.1}, and \eqref{e1.2}.
Then there exists a positive constant $C$ such that, for any dilated balls $\{B^{(k)}\}_{k\in\nn}\in\mathfrak{B}(\rn)$ and any sequence $\{\lambda_k\}_{k\in\nn}\subset[0,\infty)$,
$$\lf\|\sum_{k\in\nn}\lambda_k\mathbf{1}_{A^iB^{(k)}}\r\|_{L^\varphi(\rn)}\leq C b^{\frac{i}{\mu}}
\lf\|\sum_{k\in\nn}\lambda_k\mathbf{1}_{B^{(k)}}\r\|_{L^\varphi(\rn)}.$$
\end{lemma}

The following lemma is a special case of \cite[Lemma 3.6(i)]{jwxy21}.
\begin{lemma}\label{l2.13}
Let $q\in(0,\infty)$ and $\vz$ be a Musielak-Orlicz function. Then, for any $f\in L^{\varphi,q}(\rn)$,
$$\frac1{2}\|f\|_{L^{\varphi,q}(\rn)}\leq
\lf[\sum_{k\in\zz}2^{kq}\lf\|\mathbf{1}_{\{x\in\rn:\ |f(x)|>2^k\}}\r\|^q_{L^{\varphi}(\rn)}\r]^{\frac1{q}}
\leq2\|f\|_{L^{\varphi,q}(\rn)}$$
and, for any $f\in L^{\varphi,\infty}(\rn)$,
$$\frac1{2}\|f\|_{L^{\varphi,\infty}(\rn)}\leq
\sup_{k\in\zz}2^k\lf\|\mathbf{1}_{\{x\in\rn:\ |f(x)|>2^k\}}\r\|_{L^{\varphi}(\rn)}
\leq2\|f\|_{L^{\varphi,\infty}(\rn)}.$$
\end{lemma}

Now, we show Theorem \ref{t1.1} via using Lemmas \ref{l2.3}, \ref{l2.5}, \ref{l2.7}, and \ref{l2.12}.
\begin{proof}[Proof of Theorem \ref{t1.1}]
To prove this theorem, we first show that
\begin{equation}\label{e2.1}
\lf[\sum_{\widetilde{k}\in\zz}2^{\widetilde{k}q}\lf\|\mathbf{1}_{\{\sum_{k\in\zz}\sum_{j\in\nn}\lambda_{k,j}(m_{k,j})^*_m(x)>2^{\widetilde{k}}\}}\r\|^q_{L^\varphi(\rn)}\r]^{\frac1{q}}
\lesssim\lf[\sum_{k\in\zz}2^{kq}\lf\|\sum_{j\in\nn}\mathbf{1}_{B_{k,j}}\r\|^q_{L^\varphi(\rn)}\r]^{\frac1{q}}
\end{equation}
with the usual interpretation for $q=\infty$. Now we prove \eqref{e2.1}. Indeed, by checking the proof of \cite[Lemma 3.7]{lffy16} (see also \cite[(12)]{sll19}), we know
that, for any $k\in\zz$ and $j\in\nn$, there exists a sequence of multiples of anisotropic $(\varphi,r,s)$-atoms $\{a^l_{k,j}\}_{l\in\nn}$
associated with dilated balls $\{x_{k,j}+B_{l_{k,j}+l}\}_{l\in\nn}$, such that $m_{k,j}=\sum_{l\in\nn}a^l_{k,j}$ both pointwisely on $\rn$ and in $\cs'(\rn)$, and
\begin{equation}\label{e2.2}
\lf\|a^l_{k,j}\r\|_{L^r_\varphi(x_{k,j}+B_{l_{k,j}+l})}\lesssim b^{-l\varepsilon}\lf\|\mathbf{1}_{B_{k,j}}\r\|^{-1}_{L^\varphi(\rn)}.
\end{equation}
Therefore, for any $k\in\zz$ and $j\in\nn$, we have
\begin{align}\label{e2.3}
\lf(m_{k,j}\r)^*_m
=\sum_{l\in\nn}\lf(a^l_{k,j}\r)^*_m =\sum_{l\in\nn}\sum_{i\in\zz_+}\lf(a^l_{k,j}\r)^*_m\mathbf{1}_{U_i(x_{k,j}+B_{l_{k,j}+l})}
=:\sum_{l\in\nn}\sum_{i\in\zz_+}M^{k,j}_{l,i},
\end{align}
where $U_0(x_{k,j}+B_{l_{k,j}+l}):=x_{k,j}+B_{l_{k,j}+l}$ and $U_i(x_{k,j}+B_{l_{k,j}+l}):=(x_{k,j}+B_{l_{k,j}+l+i})\setminus (x_{k,j}+B_{l_{k,j}+l+i-1})$
for any $i\in\nn$. Furthermore, by the fact that $\|\cdot\|_{L^\varphi(\rn)}$ is a quasi-norm of $L^\varphi(\rn)$ (see, for instance, \cite[Lemma 3.2.2]{hh19})
and \eqref{e2.3}, we obtain that, for any $\widetilde{k}\in\zz$,
\begin{align}\label{e2.4}
2^{\widetilde{k}}\lf\|\mathbf{1}_{\{x\in\rn:\sum_{k\in\zz}\sum_{j\in\nn}\lambda_{k,j}(m_{k,j})^*_m(x)
>2^{\widetilde{k}+2}\}}\r\|_{L^\varphi(\rn)}
&\lesssim2^{\widetilde{k}}\lf\|\mathbf{1}_{\{x\in\rn:\sum_{k=-\infty}^{\widetilde{k}-1}\sum_{j\in\nn}\sum_{l\in\nn}\sum_{i=0}^2
\lambda_{k,j}M^{k,j}_{l,i}(x)>2^{\widetilde{k}}\}}\r\|_{L^\varphi(\rn)}\nonumber\\
&\hs\hs+2^{\widetilde{k}}\lf\|\mathbf{1}_{\{x\in\rn:\sum_{k=-\infty}^{\widetilde{k}-1}\sum_{j\in\nn}\sum_{l\in\nn}\sum_{i=3}^\infty
\lambda_{k,j}M^{k,j}_{l,i}(x)>2^{\widetilde{k}}\}}\r\|_{L^\varphi(\rn)}\nonumber\\
&\hs\hs+2^{\widetilde{k}}\lf\|\mathbf{1}_{\{x\in\rn:\sum_{k=\widetilde{k}}^\infty\sum_{j\in\nn}\sum_{l\in\nn}\sum_{i=0}^2
\lambda_{k,j}M^{k,j}_{l,i}(x)>2^{\widetilde{k}}\}}\r\|_{L^\varphi(\rn)}\nonumber\\
&\hs\hs+2^{\widetilde{k}}\lf\|\mathbf{1}_{\{x\in\rn:\sum_{k=\widetilde{k}}^\infty\sum_{j\in\nn}\sum_{l\in\nn}\sum_{i=3}^\infty
\lambda_{k,j}M^{k,j}_{l,i}(x)>2^{\widetilde{k}}\}}\r\|_{L^\varphi(\rn)}\nonumber\\
&=:{\rm{II_1+II_2+II_3+II_4}}.
\end{align}

First, we deal with the term ${\rm{II_1}}$. Notice that $r\in(\max\{q(\varphi),I(\varphi)\},\infty]$,
$s\in[\lfloor\frac{\ln b}{\ln \lambda_-}[\frac{q(\varphi)}{i(\varphi)}-1]\rfloor,\infty)\cap\zz_+$,
and $\varepsilon\in(\frac{\ln b}{\ln \lambda_-}+s+1,\infty)$, we can choose a $q_1\in(\frac{I(\varphi)}{i(\varphi)},\frac{r}{i(\varphi)})$
and an $r_1\in(\frac{\ln b}{q_1\varepsilon\ln\lambda_-},\min\{\frac{1}{q_1},\frac{i(\varphi)}{q(\varphi)}\})$.
Moreover, it is easy to find that there exists a $p^-_\varphi\in(0,i(\varphi))$, a $p^+_\varphi\in(I(\varphi),\infty)$,
and a $q_0\in(q(\varphi),\infty)$ such that $q_1\in(\frac{p^+_\varphi}{p^-_\varphi},\frac{r}{p^-_\varphi})$
and $r_1\in(\frac{\ln b}{q_1\varepsilon\ln\lambda_-},\min\{\frac{1}{q_1},\frac{p^-_\varphi}{q_0}\})$,
which further implies that
\begin{align}\label{e2.5}
p^-_\varphi\in\lf(r_1,\frac{r}{q_1}\r) \quad \mathrm{and} \quad
\frac{p^-_\varphi}{r_1}\in(q_0,\infty).
\end{align}
Furthermore, from Lemma \ref{l2.1}, it follows that $\varphi$ is of uniformly upper type $p^+_\varphi$
and of uniformly lower type $p^-_\varphi$ and $\varphi\in\mathbb{A}_{q_0}(A)$. Meanwhile, let $\widetilde{k}\in\zz$
and $a_1\in(0,1-\frac{p^+_\varphi}{q_1p^-_\varphi})$. By the fact that $q_1\in(1,\infty)$ and the H\"{o}lder inequality
for $\frac{1}{q_1}+\frac{1}{q'_1}=1$, we know that
\begin{align}\label{e2.6}
\sum_{k=-\infty}^{\widetilde{k}-1}\sum_{j\in\nn}\sum_{l\in\nn}\sum_{i=0}^2\lambda_{k,j}M^{k,j}_{l,i}
&\leq\lf(\sum_{k=-\infty}^{\widetilde{k}-1}2^{ka_1q'_1}\r)^{\frac{1}{q'_1}}
\lf\{\sum_{k=-\infty}^{\widetilde{k}-1}2^{-ka_1q_1}\lf[\sum_{j\in\nn}\sum_{l\in\nn}\sum_{i=0}^2\lambda_{k,j}M^{k,j}_{l,i}\r]^{q_1}\r\}^{\frac{1}{q_1}}\nonumber\\
&=\frac{2^{\widetilde{k}a_1}}{(2^{a_1q'_1}-1)^{\frac{1}{q'_1}}}
\lf\{\sum_{k=-\infty}^{\widetilde{k}-1}2^{-ka_1q_1}\lf[\sum_{j\in\nn}\sum_{l\in\nn}\sum_{i=0}^2\lambda_{k,j}M^{k,j}_{l,i}\r]^{q_1}\r\}^{\frac{1}{q_1}}.
\end{align}
In what follows, let $\sum_{j,l,i}:=\sum_{j\in\nn}\sum_{l\in\nn}\sum_{i=0}^2$ and
$$E_1:=\lf\{x\in\rn:\frac{2^{\widetilde{k}a_1q_1}}{(2^{a_1q'_1}-1)^{\frac{q_1}{q'_1}}}
\sum_{k=-\infty}^{\widetilde{k}-1}2^{-ka_1q_1}\lf[\sum_{j,l,i}\lambda_{k,j}M^{k,j}_{l,i}(x)\r]^{q_1}>2^{\widetilde{k}q_1}\r\}.$$
Thus, by the definition of $E_1$, \eqref{e2.6}, and \eqref{e2.4}, we find that
\begin{align}\label{e2.7}
{\rm{II_1}}\lesssim2^{\widetilde{k}}\lf\|\mathbf{1}_{E_1}\r\|_{L^\varphi(\rn)}.
\end{align}
Moreover, from the definition of $E_1$, we deduce that, for any $x\in E_1$,
$$2^{\widetilde{k}q_1p^-_\varphi(1-a_1)}\lesssim
\lf\{\sum_{k=-\infty}^{\widetilde{k}-1}2^{-ka_1q_1}\lf[\sum_{j,l,i}\lambda_{k,j}M^{k,j}_{l,i}(x)\r]^{q_1}\r\}^{p^-_\varphi},$$
which together with $r_1\in(0,1)$, $q_1r_1\in(0,1)$, $\lambda_{k,j}=C_12^k\|\mathbf{1}_{B_{k,j}}\|_{L^{\varphi}(\rn)}$, and the fact that
$(\sum_{k\in\nn}|\lambda_k|)^\theta\leq\sum_{k\in\nn}|\lambda_k|^\theta$ for any $\{\lambda_k\}_{k\in\nn}\subset\mathbb{C}$ and $\theta\in(0,1]$,
further implies that, for any  $\lambda\in(0,\infty)$,
\begin{align}\label{e2.8}
\mathbf{N}_1:
&=\int_{\rn}\varphi\lf(x,\frac{2^{\widetilde{k}}\mathbf{1}_{E_1}(x)}{\lambda}\r)dx\nonumber\\
&\lesssim2^{\widetilde{k}q_1p^-_\varphi(1-a_1)}\int_{\rn}
\lf\{\sum_{k=-\infty}^{\widetilde{k}-1}2^{-ka_1q_1}\lf[\sum_{j,l,i}\lambda_{k,j}M^{k,j}_{l,i}(x)\r]^{q_1}\r\}^{p^-_\varphi}
\varphi\lf(x,\frac{2^{\widetilde{k}}}{\lambda}\r)dx \nonumber\\
&=2^{\widetilde{k}q_1p^-_\varphi(1-a_1)}
\lf\|\sum_{k=-\infty}^{\widetilde{k}-1}2^{-kq_1(1-a_1)}\lf[\sum_{j,l,i}\lf\|\mathbf{1}_{B_{k,j}}\r\|_{L^{\varphi}(\rn)}M^{k,j}_{l,i}\r]^{q_1}\r\|^{p^-_\varphi}
_{L^{p^-_\varphi}_{\varphi(\cdot,2^{\widetilde{k}}/\lambda)}(\rn)}\nonumber\\
&=2^{\widetilde{k}q_1p^-_\varphi(1-a_1)}
\lf\|\lf\{\sum_{k=-\infty}^{\widetilde{k}-1}2^{-kq_1(1-a_1)}\lf[\sum_{j,l,i}\lf\|\mathbf{1}_{B_{k,j}}\r\|_{L^{\varphi}(\rn)}M^{k,j}_{l,i}\r]^{q_1}\r\}^{r_1}\r\|^{\frac{p^-_\varphi}{r_1}}
_{L^{p^-_\varphi/r_1}_{\varphi(\cdot,2^{\widetilde{k}}/\lambda)}(\rn)}\nonumber\\
&\leq2^{\widetilde{k}q_1p^-_\varphi(1-a_1)}
\lf\|\sum_{k=-\infty}^{\widetilde{k}-1}2^{-kq_1r_1(1-a_1)}\lf[\sum_{j,l,i}\lf\|\mathbf{1}_{B_{k,j}}\r\|_{L^{\varphi}(\rn)}M^{k,j}_{l,i}\r]^{q_1r_1}\r\|^{\frac{p^-_\varphi}{r_1}}
_{L^{p^-_\varphi/r_1}_{\varphi(\cdot,2^{\widetilde{k}}/\lambda)}(\rn)}\nonumber\\
&\leq2^{\widetilde{k}q_1p^-_\varphi(1-a_1)}
\lf\{\sum_{k=-\infty}^{\widetilde{k}-1}2^{kq_1r_1(1-a_1)}\sum_{l\in\nn}b^{-lq_1r_1\varepsilon} \r.\nonumber\\
&\hs\lf.\times\sum^2_{i=0}\lf\|\lf\{\sum_{j\in\nn}\lf[b^{l\varepsilon}\lf\|\mathbf{1}_{B_{k,j}}\r\|_{L^{\varphi}(\rn)}M^{k,j}_{l,i}\r]^{q_1r_1}\r\}
^{\frac{1}{r_1}}\r\|^{r_1}_{L^{p^-_\varphi}_{\varphi(\cdot,2^{\widetilde{k}}/\lambda)}(\rn)}\r\}^{\frac{p^-_\varphi}{r_1}}.
\end{align}
Furthermore, by the fact that $r\in(q(\varphi),\infty)$, $\vz\in\mathbb{A}_\infty(A)$, and Lemma \ref{l2.1}(ii),
we know that, $\vz\in\mathbb{A}_r(A)$. By this, the definition of $M^{k,j}_{l,i}$, the fact that $(a^l_{k,j})^*_m\lesssim M_{\mathrm{HL}}(a^l_{k,j})$,
$M_{\mathrm{HL}}$ is bounded on $L^r_\varphi(\rn)$ (see Lemma \ref{l2.2}(ii)), and \eqref{e2.4}, we conclude that, for any $k\in\zz$, $j\in\nn$, and $i\in\{0,1,2\}$,
\begin{align}\label{e2.9}
&\lf\|\lf[b^{l\varepsilon}\lf\|\mathbf{1}_{B_{k,j}}\r\|_{L^{\varphi}(\rn)}M^{k,j}_{l,i}\r]^{q_1}\r\|
_{L_\varphi^{\frac{r}{q_1}}(x_{k,j}+B_{l_{k,j}+l+2})}\nonumber\\
&\hs\leq\lf\|b^{l\varepsilon}\lf\|\mathbf{1}_{B_{k,j}}\r\|_{L^{\varphi}(\rn)}M_{\mathrm{HL}}\lf(a^l_{k,j}\r)\mathbf{1}_{x_{k,j}+B_{l_{k,j}+l+2}}\r\|^{q_1}
_{L_\varphi^{r}(x_{k,j}+B_{l_{k,j}+l+2})}\nonumber\\
&\hs\lesssim b^{l\varepsilon q_1} \lf\|\mathbf{1}_{B_{k,j}}\r\|^{q_1}_{L^{\varphi}(\rn)}
\sup_{t\in(0,\infty)}\lf[\frac{1}{\varphi(x_{k,j}+B_{l_{k,j}+l+2},t)}\int_{\rn}\lf\{M_{\mathrm{HL}}\lf(a^l_{k,j}\r)\r\}^r\varphi(x,t)dx\r]^{\frac{q_1}{r}}\nonumber\\
&\hs\lesssim b^{l\varepsilon q_1} \lf\|\mathbf{1}_{B_{k,j}}\r\|^{q_1}_{L^{\varphi}(\rn)}
\lf\|a^l_{k,j}\r\|^{q_1}_{L_\varphi^{r}(x_{k,j}+B_{l_{k,j}+l})}\lesssim1.
\end{align}
Thus, from \eqref{e2.8}, $r_1\in(\frac{\ln b}{q_1\varepsilon\ln\lambda_-},\min\{\frac{1}{q_1},\frac{p^-_\varphi}{q_0}\})$,
$\varphi\in\mathbb{A}_{q_0}(A)$, \eqref{e2.5}, Lemma \ref{l2.3} with \eqref{e2.9}, Lemma \ref{l2.5},
$\sum_{j\in\nn}\mathbf{1}_{cB_{k,j}}\lesssim1$, and the uniformly upper type $p^+_\varphi$ property of $\varphi$, it follows that, for
any given $u_1\in(\frac1{q_1r_1\varepsilon},1)$,
\begin{align}\label{e2.10}
\mathbf{N}_1
&\lesssim2^{-\widetilde{k}q_1p^-_\varphi(1-a_1)}
\lf\{\sum_{k=-\infty}^{\widetilde{k}-1}2^{kq_1r_1(1-a_1)}\sum_{l\in\nn}b^{-lq_1r_1\varepsilon}
\lf\|\lf(\sum_{j\in\nn}\mathbf{1}_{B_{l_{k,j}+l+2}}\r)
^{\frac{1}{r_1}}\r\|^{r_1}_{L^{p^-_\varphi}_{\varphi(\cdot,2^{\widetilde{k}}/\lambda)}(\rn)}\r\}^{\frac{p^-_\varphi}{r_1}}\nonumber\\
&\lesssim2^{-\widetilde{k}q_1p^-_\varphi(1-a_1)}
\lf\{\sum_{k=-\infty}^{\widetilde{k}-1}2^{kq_1r_1(1-a_1)}\sum_{l\in\nn}b^{-l(q_1r_1\varepsilon-\frac1{u_1})}
\lf\|\sum_{j\in\nn}\mathbf{1}_{cB_{l_{k,j}}}
\r\|_{L^{p^-_\varphi/r_1}_{\varphi(\cdot,2^{\widetilde{k}}/\lambda)}(\rn)}\r\}^{\frac{p^-_\varphi}{r_1}}\nonumber\\
&\lesssim2^{-\widetilde{k}q_1p^-_\varphi(1-a_1)}
\lf\{\sum_{k=-\infty}^{\widetilde{k}-1}2^{kq_1r_1(1-a_1)}\lf[\sum_{j\in\nn}\varphi\lf(cB_{l_{k,j}},\frac{2^{\widetilde{k}}}{\lambda}\r)
\r]^{\frac{r_1}{p^-_\varphi}}\r\}^{\frac{p^-_\varphi}{r_1}}\nonumber\\
&\lesssim2^{-\widetilde{k}[q_1p^-_\varphi(1-a_1)-p^+_\varphi]}
\lf\{\sum_{k=-\infty}^{\widetilde{k}-1}2^{kr_1[q_1(1-a_1)-\frac{p^+_\varphi}{p^-_\varphi}]}\lf[\sum_{j\in\nn}\varphi\lf(cB_{l_{k,j}},\frac{2^k}{\lambda}\r)
\r]^{\frac{r_1}{p^-_\varphi}}\r\}^{\frac{p^-_\varphi}{r_1}}.
\end{align}
Moreover, similarly to the proof of \cite[(3.30)]{jwyyz23}, we obtain that, for any $k\in\zz$ and $\lambda\in(0,\infty)$,
\begin{align}\label{e2.11}
\sum_{j\in\nn}\varphi\lf(cB_{l_{k,j}},\frac{2^k}{\lambda}\r)\lesssim
\lf\{ \begin{array}{ll}
\lf(\frac{2^k}{\lambda}\r)^{p^+_\varphi}\lf\|\sum\limits_{j\in\nn}\mathbf{1}_{B_{l_{k,j}}}\r\|_{L^\varphi(\rn)}^{p^+_\varphi},& \ \frac{2^k}{\lambda}\lf\|\sum\limits_{j\in\nn}\mathbf{1}_{B_{l_{k,j}}}\r\|_{L^\varphi(\rn)}\in[1,\infty);\\
\lf(\frac{2^k}{\lambda}\r)^{p^-_\varphi}\lf\|\sum\limits_{j\in\nn}\mathbf{1}_{B_{l_{k,j}}}\r\|_{L^\varphi(\rn)}^{p^-_\varphi},& \ \frac{2^k}{\lambda}\lf\|\sum\limits_{j\in\nn}\mathbf{1}_{B_{l_{k,j}}}\r\|_{L^\varphi(\rn)}\in(0,1),
\end{array}\r.
\end{align}
which together with \eqref{e2.10}, further implies that, for any $\widetilde{k}\in\zz$ and $\lambda\in(0,\infty)$,
\begin{align}\label{e2.12}
\lf(\mathbf{N}_1\r)^{\frac{r_1}{p^-_\varphi}}
&\lesssim2^{-\widetilde{k}r_1[q_1(1-a_1)-\frac{p^+_\varphi}{p^-_\varphi}]}\sum_{k=-\infty}^{\widetilde{k}-1}2^{kr_1[q_1(1-a_1)-\frac{p^+_\varphi}{p^-_\varphi}]}\nonumber\\
&\hs\times\max\lf\{\lf(\frac{2^k}{\lambda}\lf\|\sum\limits_{j\in\nn}\mathbf{1}_{B_{l_{k,j}}}\r\|_{L^\varphi(\rn)}\r)^{\frac{r_1p^+_\varphi}{p^-_\varphi}},
\lf(\frac{2^k}{\lambda}\lf\|\sum\limits_{j\in\nn}\mathbf{1}_{B_{l_{k,j}}}\r\|_{L^\varphi(\rn)}\r)^{r_1}\r\}.
\end{align}
Let $\lambda:=2^{\widetilde{k}}\|\mathbf{1}_{E_1}\|_{L^\varphi(\rn)}$. Then, from Lemma \ref{l2.6}, we deduce that $N_1=1$, which together with
\eqref{e2.7} and \eqref{e2.12}, further implies that
\begin{align*}
{\rm{II_1}}&\lesssim2^{\widetilde{k}}\lf\|\mathbf{1}_{E_1}\r\|_{L^\varphi(\rn)}=\lambda\\
&\lesssim\max\lf\{2^{-\frac{\widetilde{k}[q_1p^-_\varphi(1-a_1)-p^+_\varphi]}{p^+_\varphi}}\lf[
\sum_{k=-\infty}^{\widetilde{k}-1}2^{kr_1q_1(1-a_1)}
\lf\|\sum\limits_{j\in\nn}\mathbf{1}_{B_{l_{k,j}}}\r\|^{\frac{r_1p^+_\varphi}{p^-_\varphi}}_{L^\varphi(\rn)}\r]^{\frac{p^-_\varphi}{r_1p^+_\varphi}},\r.\\
&\quad \quad \quad \quad\lf.2^{-\frac{\widetilde{k}[q_1p^-_\varphi(1-a_1)-p^+_\varphi]}{p^-_\varphi}}
\lf[\sum_{k=-\infty}^{\widetilde{k}-1}2^{\frac{kr_1[q_1p^-_\varphi(1-a_1)-p^+_\varphi]}{p^-_\varphi}}2^{kr_1}
\lf\|\sum\limits_{j\in\nn}\mathbf{1}_{B_{l_{k,j}}}\r\|^{r_1}_{L^\varphi(\rn)}\r]^{\frac1{r_1}}\r\}.
\end{align*}
By this, $a_1\in(0,1-\frac{p^+_\varphi}{q_1p^-_\varphi})$, and Lemma \ref{l2.7}, we know that
\begin{align}\label{e2.13}
\lf[\sum_{\widetilde{k}\in\zz}\lf({\rm{II_{1}}}\r)^q\r]^{\frac1{q}}\ls
\lf[\sum_{k\in\zz}2^{kq}\lf\|\sum_{j\in\nn}\mathbf{1}_{B_{k,j}}\r\|^q_{L^\varphi(\rn)}\r]^{\frac1{q}},
\end{align}
which implies the desired result.

Next, we deal with the term ${\rm{II_{2}}}$. To this end, we first estimate the term $M^{k,j}_{l,i}$.
For any $k\in\zz_+$ and $j\in\nn$, by the vanishing moments of $a^l_{k,j}$, $m\in[s,\infty)\cap\zz_+$,
the Taylor remainder theorem, the definition of $\cs_m(\rn)$,
the H\"{o}lder inequality, $\varphi\in\mathbb{A}_{r}(A)$, and \eqref{e2.2},
we know that, for any $\varphi\in\cs_m(\rn)$, $k\in\zz$, $j\in\nn$, $t\in(0,\infty)$, $l\in\zz_+$,
$i\in\zz_+\setminus \{0,1,2\}$, and $x\in U_i(x_{k,j}+B_{l_{k,j}+l})$,  there exists a $\widetilde{y}\in x_{k,j}+B_{l_{k,j}+l}$ such that
\begin{align*}
&\lf|a^l_{k,j}\ast\varphi_t(x)\r|
=\lf|\int_{x_{k,j}+B_{l_{k,j}+l}}a^l_{k,j}(y)\varphi\lf(\frac{x-y}{A^t}\r)\frac{dy}{b^t}\r|\\
&\hs=\lf|\int_{x_{k,j}+B_{l_{k,j}+l}}a^l_{k,j}(y)\lf[\varphi\lf(\frac{x-y}{A^t}\r)
-\sum_{\{\gamma\in\zz^n_+:|\gamma|\leq s\}}\frac{\partial^\gamma_x\varphi(\frac{x-x_{k,j}}{A^t})}{\gamma!}\lf(\frac{x_{k,j}-y}{A^t}\r)^\gamma\r]\frac{dy}{b^t}\r|\\
&\hs=\lf|\int_{x_{k,j}+B_{l_{k,j}+l}}a^l_{k,j}(y)\lf[\sum_{\{\gamma\in\zz^n_+:|\gamma|=s+1\}}\frac{\partial^\gamma_x\varphi(\frac{x-\widetilde{y}}{A^t})}{\gamma!}
\lf(\frac{x_{k,j}-y}{A^t}\r)^\gamma\r]\frac{dy}{b^t}\r|\\
&\hs\lesssim\frac{b^{(l_{k,j}+l)(s+1)}}{b^{(l_{k,j}+l+i)(s+2)}}\int_{x_{k,j}+B_{l_{k,j}+l}}\lf|a^l_{k,j}(y)\r|dy\\
&\hs\lesssim\frac{b^{(l_{k,j}+l)(s+1)}}{b^{(l_{k,j}+l+i)(s+2)}}\lf[\int_{x_{k,j}+B_{l_{k,j}+l}}\lf|a^l_{k,j}(y)\r|^r\varphi(y,1)dy\r]^{\frac1{r}}
\lf\{\int_{x_{k,j}+B_{l_{k,j}+l}}\lf[\varphi(y,1)\r]^{-\frac{r'}{r}}dy\r\}^{\frac1{r'}}\\
&\hs\lesssim\frac{b^{(l_{k,j}+l)(s+1)}b^{l_{k,j}+l}}{b^{(l_{k,j}+l+i)(s+2)}}\lf[\int_{x_{k,j}+B_{l_{k,j}+l}}\lf|a^l_{k,j}(y)\r|^r\varphi(y,1)dy\r]^{\frac1{r}}
\lf[\varphi\lf(x_{k,j}+B_{l_{k,j}+l}\r)\r]^{-\frac1{r}}\\
&\hs=b^{-i(s+2)}\lf\|a^l_{k,j}\r\|_{L^r_\varphi(x_{k,j}+B_{l_{k,j}+l})}
\lesssim b^{-l\varepsilon-i(s+2)}\lf\|\mathbf{1}_{B_{k,j}}\r\|^{-1}_{L^\varphi(\rn)},
\end{align*}
which, combined with Lemma \ref{l2.8}, further implies that, for any
$k\in\zz$, $j\in\nn$, $l\in\zz_+$, $i\in\zz_+\setminus \{0,1,2\}$, and $x\in U_i(x_{k,j}+B_{l_{k,j}+l})$,
\begin{align}\label{e2.14}
M^{k,j}_{l,i}(x)&\lesssim\sup_{\varphi\in\cs_m(\rn)}\sup_{t\in(0,\infty)}\lf|a^l_{k,j}\ast\varphi_t(x)\r|\nonumber\\
&\lesssim b^{-l\varepsilon-i(s+2)}\lf\|\mathbf{1}_{B_{k,j}}\r\|^{-1}_{L^\varphi(\rn)}.
\end{align}
Furthermore, notice that $s\geq\lfloor\frac{\ln b}{\ln \lambda_-}[\frac{q(\varphi)}{i(\varphi)}-1]\rfloor$,
we can choose a $q_2\in(\frac{\ln b/\ln \lambda_-}{\ln b/\ln \lambda_-+s+1},\frac{i(\varphi)}{q(\varphi)})\cap(0,1)$.
Thus, by Lemmas \ref{l2.9} and \ref{l2.10}, $\lambda_{k,j}=C_12^k\|\mathbf{1}_{B_{k,j}}\|_{L^{\varphi}(\rn)}$,
\eqref{e2.14}, Lemma \ref{l2.12}, and $\varepsilon\in(\frac{\ln b}{\ln \lambda_-}+s+1,\infty)$,
we know that, for any $\vartheta_1\in(\frac{\ln b/\ln \lambda_-}{\ln b/\ln \lambda_-+s+1},\frac{i(\varphi)}{q_2q(\varphi)})\cap(0,1)$, $\varepsilon>\frac1{\vartheta_1}$ and
\begin{align*}
{\rm{II_2}}
&=2^{\widetilde{k}}\lf\|\mathbf{1}_{\{x\in\rn:\sum_{k=-\infty}^{\widetilde{k}-1}\sum_{j\in\nn}\sum_{l\in\nn}\sum_{i=3}^\infty
\lambda_{k,j}M^{k,j}_{l,i}(x)>2^{\widetilde{k}}\}}\r\|^{\frac1{q_2}}_{L^{\varphi_{1/q_2}}(\rn)}\\
&\leq2^{\widetilde{k}(1-\frac1{q_2})}\lf\|\sum_{k=-\infty}^{\widetilde{k}-1}\sum_{j\in\nn}\sum_{l\in\nn}\sum_{i=3}^\infty
\lambda_{k,j}M^{k,j}_{l,i}\r\|^{\frac1{q_2}}_{L^{\varphi_{1/q_2}}(\rn)}\\
&\lesssim2^{\widetilde{k}(1-\frac1{q_2})}\lf\{\sum_{l\in\nn}b^{-l\varepsilon}\sum_{i=3}^\infty b^{-i(s+2)}
\sum_{k=-\infty}^{\widetilde{k}-1}2^k\lf\|\sum_{j\in\nn}\mathbf{1}_{U_i(x_{k,j}+B_{l_{k,j}+l})}
\r\|_{L^{\varphi_{1/q_2}}(\rn)}\r\}^{\frac1{q_2}}\\
&\lesssim2^{\widetilde{k}(1-\frac1{q_2})}\lf\{\sum_{l\in\nn}b^{-l(\varepsilon-\frac1{\vartheta_1})}\sum_{i=3}^\infty
b^{-i[(s+2)-\frac1{\vartheta_1}]}\sum_{k=-\infty}^{\widetilde{k}-1}2^k\lf\|\sum_{j\in\nn}\mathbf{1}_{x_{k,j}+B_{l_{k,j}}}
\r\|_{L^{\varphi_{1/q_2}}(\rn)}\r\}^{\frac1{q_2}}\\
&\lesssim2^{\widetilde{k}(1-\frac1{q_2})}\lf\{\sum_{k=-\infty}^{\widetilde{k}-1}2^k
\lf\|\sum_{j\in\nn}\mathbf{1}_{x_{k,j}+B_{l_{k,j}}}\r\|^{q_2}_{L^{\varphi}(\rn)}\r\}^{\frac1{q_2}}.
\end{align*}
By this, $q_2\in(0,1)$ and Lemma \ref{l2.7} with $\alpha:=\frac1{q_2}-1$, $\beta:=q_2$, and
$\mu_k:=2^k\|\mathbf{1}_{x_{k,j}+B_{l_{k,j}}}\|_{L^{\varphi}(\rn)}$, we conclude that
\begin{align}\label{e2.15}
\lf[\sum_{\widetilde{k}\in\zz}\lf({\rm{II_{2}}}\r)^q\r]^{\frac1{q}}\ls
\lf[\sum_{k\in\zz}2^{kq}\lf\|\sum_{j\in\nn}\mathbf{1}_{B_{k,j}}\r\|^q_{L^\varphi(\rn)}\r]^{\frac1{q}},
\end{align}
which implies the desired result.

Now we deal with the term ${\rm{II_{3}}}$. Indeed, notice that
$s\geq\lfloor\frac{\ln b}{\ln \lambda_-}[\frac{q(\varphi)}{i(\varphi)}-1]\rfloor$,
we obtain that there exists a $p_0\in(q(\varphi),\infty)$ and a $\widetilde{p}^-_\varphi\in(0,i(\varphi))$
such that $s>\frac{\ln b}{\ln \lambda_-}(\frac{p_0}{\widetilde{p}^-_\varphi}-1)-1$. From this and
$\varepsilon\in(\frac{\ln b}{\ln \lambda_-}+s+1,\infty)$, we can choose a
$b\in(\frac{\ln b/\ln \lambda_-}{\varepsilon},\min\{1,\frac{\widetilde{p}^-_\varphi}{p_0}\})$,
a $q_3\in(\frac{\ln b/\ln \lambda_-}{b\varepsilon},1)$, and a $\widetilde{p}^+_\varphi\in(I(\varphi),\infty)$.
Thus, by Lemma \ref{l2.1}, we know that $\varphi$ is also of uniformly lower type $\widetilde{p}^-_\varphi$ and of uniformly upper type $\widetilde{p}^+_\varphi$ and $\varphi\in \mathbb{A}_{p_0}(A)$. Moreover, by $q_3\in(0,1)$
and the fact that $(\sum_{k\in\nn}|\lambda_k|)^\theta\leq\sum_{k\in\nn}|\lambda_k|^\theta$ for any $\{\lambda_k\}_{k\in\nn}\subset\mathbb{C}$ and $\theta\in(0,1]$, we know that
$$\sum_{k=\widetilde{k}}^\infty\sum_{j,l,i}\lambda_{k,j}M^{k,j}_{l,i}
\leq\lf\{\sum_{k=\widetilde{k}}^\infty\sum_{j,l,i}\lf[\lambda_{k,j}M^{k,j}_{l,i}\r]^{q_3}\r\}^{\frac1{q_3}},$$
which, together with
$$E_2:=\lf\{x\in\rn:\sum_{k=\widetilde{k}}^\infty\sum_{j,l,i}
\lf[\lambda_{k,j}M^{k,j}_{l,i}(x)\r]^{q_3}>2^{\widetilde{k}q_3}\r\}$$
and the definition of  $\|\cdot\|_{L^\varphi(\rn)}$, further implies that
\begin{align}\label{e2.16}
{\rm{II_3}}\leq2^{\widetilde{k}}\lf\|\mathbf{1}_{E_2}\r\|_{L^\varphi(\rn)}
=\inf\lf\{ \lz\in(0,\fz): \int_\rn \vz\lf(x, \frac{2^{\widetilde{k}}\mathbf{1}_{E_2}(x)}{\lz}\r) dx\le 1\r\}.
\end{align}
Furthermore, by the fact that, for any $E_2$,
$$2^{\widetilde{k}q_3\widetilde{p}^-_\varphi}\leq
\lf\{\sum_{k=\widetilde{k}}^{\infty}\sum_{j,l,i}\lf[\lambda_{k,j}M^{k,j}_{l,i}(x)\r]^{q_3}\r\}^{\widetilde{p}^-_\varphi},$$
$\lambda_{k,j}=C_12^k\|\mathbf{1}_{B_{k,j}}\|_{L^{\varphi}(\rn)}$, and
$b\in(\frac{\ln b/\ln \lambda_-}{\varepsilon},\min\{1,\frac{\widetilde{p}^-_\varphi}{p_0}\})
\subset(0,\min\{1,\widetilde{p}^-_\varphi\})$, we conclude that, for any  $\lambda\in(0,\infty)$,
\begin{align}\label{e2.17}
\mathbf{N}_2:
&=\int_{\rn}\varphi\lf(x,\frac{2^{\widetilde{k}}\mathbf{1}_{E_2}(x)}{\lambda}\r)dx\nonumber\\
&\leq2^{-\widetilde{k}q_3\widetilde{p}^-_\varphi}\int_{\rn}
\lf\{\sum_{k=\widetilde{k}}^{\infty}\sum_{j,l,i}\lf[\lambda_{k,j}M^{k,j}_{l,i}(x)\r]^{q_3}\r\}^{\widetilde{p}^-_\varphi}
\varphi\lf(x,\frac{2^{\widetilde{k}}}{\lambda}\r)dx \nonumber\\
&=2^{-\widetilde{k}q_3\widetilde{p}^-_\varphi}
\lf\|\sum_{k=\widetilde{k}}^{\infty}2^{kq_3}\sum_{j,l,i}\lf[\lf\|\mathbf{1}_{B_{k,j}}\r\|_{L^{\varphi}(\rn)}M^{k,j}_{l,i}
\r]^{q_3}\r\|^{\widetilde{p}^-_\varphi}_{L^{\widetilde{p}^-_\varphi}_{\varphi(\cdot,2^{\widetilde{k}}/\lambda)}(\rn)}\nonumber\\
&=2^{-\widetilde{k}q_3\widetilde{p}^-_\varphi}
\lf\|\lf\{\sum_{k=\widetilde{k}}^{\infty}2^{kq_3}\sum_{j,l,i}\lf[\lf\|\mathbf{1}_{B_{k,j}}\r\|_{L^{\varphi}(\rn)}
M^{k,j}_{l,i}\r]^{q_3}\r\}^{b}\r\|^{\frac{\widetilde{p}^-_\varphi}{b}}
_{L^{\widetilde{p}^-_\varphi/b}_{\varphi(\cdot,2^{\widetilde{k}}/\lambda)}(\rn)}\nonumber\\
&\leq2^{-\widetilde{k}q_3\widetilde{p}^-_\varphi}
\lf\{\sum_{k=\widetilde{k}}^{\infty}2^{kq_3b}\sum_{l\in\nn}b^{-lq_3b\varepsilon}
\sum^2_{i=0}\lf\|\sum_{j\in\nn}\lf[b^{l\varepsilon}\lf\|\mathbf{1}_{B_{k,j}}
\r\|_{L^{\varphi}(\rn)}M^{k,j}_{l,i}\r]^{q_3b}\r\|_{L^{\widetilde{p}^-_\varphi/b}
_{\varphi(\cdot,2^{\widetilde{k}}/\lambda)}(\rn)}\r\}^{\frac{\widetilde{p}^-_\varphi}{b}}\nonumber\\
&=2^{-\widetilde{k}q_3\widetilde{p}^-_\varphi}
\lf\{\sum_{k=\widetilde{k}}^{\infty}2^{kq_3b}\sum_{l\in\nn}b^{-lq_3b\varepsilon}
\sum^2_{i=0}\lf\|\lf\{\sum_{j\in\nn}\lf[b^{l\varepsilon}\lf\|\mathbf{1}_{B_{k,j}}
\r\|_{L^{\varphi}(\rn)}M^{k,j}_{l,i}\r]^{q_3b}\r\}^{\frac1{b}}\r\|^b_{L^{\widetilde{p}^-_\varphi}
_{\varphi(\cdot,2^{\widetilde{k}}/\lambda)}(\rn)}\r\}^{\frac{\widetilde{p}^-_\varphi}{b}}.
\end{align}
In addition, by the definition of $M^{k,j}_{l,i}$, Lemma \ref{l2.2}(i), $q_3\in(0,1)$,
$\vz\in\mathbb{A}_r(A)$, $r\in(q(\varphi),\infty)$, Lemma \ref{l2.2}(ii), and \eqref{e2.2},
we obtain that, for any $k\in\zz$, $j\in\nn$, $l\in\nn$, and $i\in\{0,1,2\}$,
\begin{align}\label{e2.18}
&\lf\|\lf[b^{l\varepsilon}\lf\|\mathbf{1}_{B_{k,j}}\r\|_{L^{\varphi}(\rn)}M^{k,j}_{l,i}\r]^{q_3}\r\|
_{L_\varphi^{\frac{r}{q_3}}(x_{k,j}+B_{l_{k,j}+l+2})}\nonumber\\
&\hs\leq\lf\|b^{l\varepsilon}\lf\|\mathbf{1}_{B_{k,j}}\r\|_{L^{\varphi}(\rn)}M_{\mathrm{HL}}\lf(a^l_{k,j}\r)
\mathbf{1}_{x_{k,j}+B_{l_{k,j}+l+2}}\r\|^{q_3}_{L_\varphi^{r}(x_{k,j}+B_{l_{k,j}+l+2})}\nonumber\\
&\hs\lesssim b^{l\varepsilon q_3} \lf\|\mathbf{1}_{B_{k,j}}\r\|^{q_3}_{L^{\varphi}(\rn)}
\sup_{t\in(0,\infty)}\lf[\frac{1}{\varphi(x_{k,j}+B_{l_{k,j}+l+2},t)}
\int_{\rn}\lf\{M_{\mathrm{HL}}\lf(a^l_{k,j}\r)\r\}^r\varphi(x,t)dx\r]^{\frac{q_3}{r}}\nonumber\\
&\hs\lesssim b^{l\varepsilon q_3} \lf\|\mathbf{1}_{B_{k,j}}\r\|^{q_3}_{L^{\varphi}(\rn)}
\lf\|a^l_{k,j}\r\|^{q_3}_{L_\varphi^{r}(x_{k,j}+B_{l_{k,j}+l})}\lesssim1.
\end{align}
Therefore, by \eqref{e2.18}, $q_3\in(\frac{\ln b/\ln \lambda_-}{b\varepsilon},1)$,
$\widetilde{p}^-_\varphi\in(b,\frac{r}{q_3})$, $\varphi\in\mathbb{A}_{p_0}(A)$,
$\frac{\widetilde{p}^-_\varphi}{b}>p_0$, Lemma \ref{l2.3} with \eqref{e2.18}, Lemma \ref{l2.5},
$\sum_{j\in\nn}\mathbf{1}_{cB_{k,j}}\lesssim1$, and the uniformly lower type $\widetilde{p}^-_\varphi$ property of $\varphi$,
we find that, for any given $u\in(\frac1{q_3b\varepsilon},1)$,
\begin{align}\label{e2.19}
\mathbf{N}_2
&\lesssim2^{-\widetilde{k}q_3\widetilde{p}^-_\varphi}
\lf\{\sum_{k=\widetilde{k}}^{\infty}2^{kq_3b}\sum_{l\in\nn}b^{-lq_3b\varepsilon}
\lf\|\lf(\sum_{j\in\nn}\mathbf{1}_{B_{l_{k,j}+l+2}}\r)^{\frac{1}{b}}\r\|^{b}_{L^{\widetilde{p}^-_\varphi}
_{\varphi(\cdot,2^{\widetilde{k}}/\lambda)}(\rn)}\r\}^{\frac{\widetilde{p}^-_\varphi}{b}}\nonumber\\
&=2^{-\widetilde{k}q_3\widetilde{p}^-_\varphi}
\lf\{\sum_{k=\widetilde{k}}^{\infty}2^{kq_3b}\sum_{l\in\nn}b^{-lq_3b\varepsilon}
\lf\|\sum_{j\in\nn}\mathbf{1}_{B_{l_{k,j}+l+2}}\r\|_{L^{\widetilde{p}^-_\varphi/b}
_{\varphi(\cdot,2^{\widetilde{k}}/\lambda)}(\rn)}\r\}^{\frac{\widetilde{p}^-_\varphi}{b}}\nonumber\\
&\lesssim2^{-\widetilde{k}q_3\widetilde{p}^-_\varphi}
\lf\{\sum_{k=\widetilde{k}}^{\infty}2^{kq_3b}\sum_{l\in\nn}b^{-l(q_3b\varepsilon-\frac1{u})}
\lf\|\lf(\sum_{j\in\nn}\mathbf{1}_{cB_{l_{k,j}}}\r)^{\frac{1}{b}}\r\|^{b}_{L^{\widetilde{p}^-_\varphi}
_{\varphi(\cdot,2^{\widetilde{k}}/\lambda)}(\rn)}\r\}^{\frac{\widetilde{p}^-_\varphi}{b}}\nonumber\\
&\sim2^{-\widetilde{k}q_3\widetilde{p}^-_\varphi}
\lf\{\sum_{k=\widetilde{k}}^{\infty}2^{kq_3b}\lf[\sum_{j\in\nn}\varphi\lf(cB_{l_{k,j}},
\frac{2^{\widetilde{k}}}{\lambda}\r)\r]^{\frac{b}{\widetilde{p}^-_\varphi}}
\r\}^{\frac{\widetilde{p}^-_\varphi}{b}}\nonumber\\
&\lesssim2^{-\widetilde{k}\widetilde{p}^-_\varphi(q_3-1)}
\lf\{\sum_{k=\widetilde{k}}^{\infty}2^{kb(q_3-1)}\lf[\sum_{j\in\nn}\varphi\lf(B_{l_{k,j}},
\frac{2^k}{\lambda}\r)\r]^{\frac{b}{\widetilde{p}^-_\varphi}}
\r\}^{\frac{\widetilde{p}^-_\varphi}{b}}.
\end{align}
By this and similarly to the proofs of \eqref{e2.11} and \eqref{e2.12}, we know that, for any $k\in\zz$ and $\lambda\in(0,\infty)$,
\begin{align}\label{e2.20}
\lf(\mathbf{N}_2\r)^{\frac{b}{p^-_\varphi}}
&\lesssim2^{-\widetilde{k}b(q_3-1)}\sum_{k=\widetilde{k}}^{\infty}2^{kb(q_3-1)}\nonumber\\
&\hs\times\max\lf\{\lf(\frac{2^k}{\lambda}\lf\|\sum\limits_{j\in\nn}\mathbf{1}_{B_{l_{k,j}}}\r\|
_{L^\varphi(\rn)}\r)^{\frac{b\widetilde{p}^+_\varphi}{\widetilde{p}^-_\varphi}},
\lf(\frac{2^k}{\lambda}\lf\|\sum\limits_{j\in\nn}\mathbf{1}_{B_{l_{k,j}}}\r\|_{L^\varphi(\rn)}\r)^{b}\r\}.
\end{align}
Moreover, let $\lambda:=2^{\widetilde{k}}\|\mathbf{1}_{E_2}\|_{L^\varphi(\rn)}$. Then, by
the definition of $N_2$ and Lemma \ref{l2.6},
we conclude that $N_2=1$, which together with \eqref{e2.16} and \eqref{e2.20}, further implies that
\begin{align*}
{\rm{II_3}}&\lesssim2^{\widetilde{k}}\lf\|\mathbf{1}_{E_2}\r\|_{L^\varphi(\rn)}=\lambda\\
&\lesssim\max\lf\{2^{-\frac{\widetilde{k}\widetilde{p}^-_\varphi(q_3-1)}{\widetilde{p}^+_\varphi}}
\lf[\sum_{k=\widetilde{k}}^{\infty}2^{kb(q_3-1)}
\lf\{2^k\lf\|\sum\limits_{j\in\nn}\mathbf{1}_{B_{l_{k,j}}}\r\|_{L^\varphi(\rn)}\r\}
^{\frac{b\widetilde{p}^+_\varphi}{\widetilde{p}^-_\varphi}}\r]^{\frac{\widetilde{p}^-_\varphi}{b\widetilde{p}^+_\varphi}},\r.\\
&\quad \quad \quad \quad\lf.2^{-\widetilde{k}(q_3-1)}
\lf[\sum_{k=\widetilde{k}}^{\infty}2^{kb(q_3-1)}\lf\{2^k
\lf\|\sum\limits_{j\in\nn}\mathbf{1}_{B_{l_{k,j}}}\r\|_{L^\varphi(\rn)}\r\}^b\r]^{\frac1{b}}\r\}.
\end{align*}
From this, $q_3\in(0,1)$, and Lemma \ref{l2.7}, we deduce that
\begin{align}\label{e2.21}
\lf[\sum_{\widetilde{k}\in\zz}\lf({\rm{II_{3}}}\r)^q\r]^{\frac1{q}}\ls
\lf[\sum_{k\in\zz}2^{kq}\lf\|\sum_{j\in\nn}\mathbf{1}_{B_{k,j}}\r\|^q_{L^\varphi(\rn)}\r]^{\frac1{q}},
\end{align}
which implies the desired result.

Now we turn to estimate the term ${\rm{II_{4}}}$. By the fact that $(\sum_{k\in\nn}|\lambda_k|)^\theta\leq\sum_{k\in\nn}|\lambda_k|^\theta$ for any $\{\lambda_k\}_{k\in\nn}\subset\mathbb{C}$ and $\theta\in(0,1]$, we know that, for any $\widetilde{a}\in(0,1)$,
\begin{align*}
\sum_{k=\widetilde{k}}^\infty\sum_{j\in\nn}\sum_{l\in\nn}\sum_{i=3}^\infty\lambda_{k,j}M^{k,j}_{l,i}
\leq\lf\{\sum_{k=\widetilde{k}}^\infty\sum_{j\in\nn}\sum_{l\in\nn}\sum_{i=3}^\infty
\lf[\lambda_{k,j}M^{k,j}_{l,i}\r]^{\widetilde{a}}\r\}^{\frac1{\widetilde{a}}}.
\end{align*}
By this, \eqref{e2.14}, and $\lambda_{k,j}=C_12^k\|\mathbf{1}_{B_{k,j}}\|_{L^{\varphi}(\rn)}$, we obtain that
\begin{align}\label{e2.22}
{\rm{II_4}}
&\leq2^{\widetilde{k}}\lf\|\mathbf{1}_{\{x\in\rn:\sum_{k=\widetilde{k}}^\infty\sum_{j\in\nn}\sum_{l\in\nn}\sum_{i=3}^\infty
[\lambda_{k,j}M^{k,j}_{l,i}(x)]^{\widetilde{a}}>2^{\widetilde{k}\widetilde{a}}\}}\r\|_{L^\varphi(\rn)}\nonumber\\
&\leq2^{\widetilde{k}(1-\widetilde{a})}\lf\|\sum_{k=\widetilde{k}}^\infty\sum_{j\in\nn}\sum_{l\in\nn}\sum_{i=3}^\infty
\lf[\lambda_{k,j}M^{k,j}_{l,i}(x)\r]^{\widetilde{a}}\r\|_{L^\varphi(\rn)}\nonumber\\
&\lesssim2^{\widetilde{k}(1-\widetilde{a})}\lf\|\sum_{k=\widetilde{k}}^\infty2^{k\widetilde{a}}
\sum_{l\in\nn}b^{-l\widetilde{a}\varepsilon}\sum_{i=3}^\infty b^{-i\widetilde{a}(s+2)}\sum_{j\in\nn}
\mathbf{1}_{U_i(x_{k,j}+B_{l_{k,j}+l})}\r\|_{L^\varphi(\rn)}.
\end{align}
Furthermore, notice that $s\geq\lfloor\frac{\ln b}{\ln \lambda_-}[\frac{q(\varphi)}{i(\varphi)}-1]\rfloor$,
we can choose a $\widetilde{b}\in(\frac{\ln b/\ln \lambda_-}{\ln b/\ln \lambda_-+s+1},\min\{1,i(\varphi)\})$
and an $\widetilde{a}\in(0,1)$ such that $\widetilde{a}\widetilde{b}>\frac{\ln b/\ln \lambda_-}{\ln b/\ln \lambda_-+s+1}$,
which, together with \eqref{e2.22}, Lemmas \ref{l2.9} and \ref{l2.10},
and $\varepsilon\in(\frac{\ln b}{\ln \lambda_-}+s+1,\infty)$,
further implies that $\widetilde{a}\widetilde{b}>\frac1{\varepsilon}$, $\widetilde{a}\widetilde{b}>\frac1{s+2}$, and
\begin{align*}
{\rm{II_4}}
&\lesssim2^{\widetilde{k}(1-\widetilde{a})}\lf\{\sum_{k=\widetilde{k}}^\infty2^{k\widetilde{a}\widetilde{b}}
\sum_{l\in\nn}b^{-l\widetilde{a}\widetilde{b}\varepsilon}\sum_{i=3}^\infty b^{-i\widetilde{a}\widetilde{b}(s+2)}
\lf\|\lf[\sum_{j\in\nn}\mathbf{1}_{U_i(x_{k,j}+B_{l_{k,j}+l})}\r]^{\widetilde{b}}
\r\|_{L^{\varphi_{1/\widetilde{b}}}(\rn)}\r\}^{\frac1{\widetilde{b}}}\\
&\lesssim2^{\widetilde{k}(1-\widetilde{a})}\lf\{\sum_{k=\widetilde{k}}^\infty2^{k\widetilde{a}\widetilde{b}}
\sum_{l\in\nn}b^{-l(\widetilde{a}\widetilde{b}\varepsilon-1)}\sum_{i=3}^\infty b^{-i[\widetilde{a}\widetilde{b}(s+2)-1]}
\lf\|\sum_{j\in\nn}\mathbf{1}_{U_i(x_{k,j}+B_{l_{k,j}})}
\r\|^{\widetilde{b}}_{L^{\varphi}(\rn)}\r\}^{\frac1{\widetilde{b}}}\\
&\lesssim2^{\widetilde{k}(1-\widetilde{a})}\lf\{\sum_{k=\widetilde{k}}^\infty2^{k\widetilde{a}\widetilde{b}}
\lf\|\sum_{j\in\nn}\mathbf{1}_{x_{k,j}+B_{l_{k,j}}}\r\|^{\widetilde{b}}_{L^{\varphi}(\rn)}\r\}^{\frac1{\widetilde{b}}}.
\end{align*}
From this, $\widetilde{a}\in(0,1)$ and Lemma \ref{l2.7} with $\widetilde{\alpha}:=1-\widetilde{a}$, $\widetilde{\beta}:=\widetilde{b}$, and
$\mu_k:=2^k\|\mathbf{1}_{x_{k,j}+B_{l_{k,j}}}\|_{L^{\varphi}(\rn)}$, we deduce that
\begin{align}\label{e2.23}
\lf[\sum_{\widetilde{k}\in\zz}\lf({\rm{II_{4}}}\r)^q\r]^{\frac1{q}}\ls
\lf[\sum_{k\in\zz}2^{kq}\lf\|\sum_{j\in\nn}\mathbf{1}_{B_{k,j}}\r\|^q_{L^\varphi(\rn)}\r]^{\frac1{q}},
\end{align}
which implies the desired result. Combining the above estimates of
\eqref{e2.4}, \eqref{e2.13}, \eqref{e2.15}, \eqref{e2.21}, and \eqref{e2.23},
we conclude that \eqref{e2.1} holds true.
This completes the proof of Theorem \ref{t1.1}.
\end{proof}




\section{Proof of Theorems \ref{t1.3} and \ref{t1.4}}\label{s3}
\hskip\parindent
In this section, we prove Theorems \ref{t1.3} and \ref{t1.4}.
To show Theorem \ref{t1.3}, we need the following several technical lemmas.
\begin{lemma}\label{l3.1}
Let $q\in(0,\infty]$, $m\in\zz_+$, and $\vz\in\mathbb{A}_\infty(A)$ be a Musielak-Orlicz function.
Then, for any $i\in\zz$ and $f\in H^{\varphi,q}_{A,m}(\mathbb{R}^{n})$,
$\Omega_i:=\{x\in\rn:f^*_m(x)>2^i\}$ is a proper open subset of $\rn$.
\end{lemma}
\begin{proof}
Let $i\in\zz$ and $f\in H^{\varphi,q}_{A,m}(\mathbb{R}^{n})$.
When $\Omega_i=\rn$, then, from Lemma \ref{l2.13} and $f\in H^{\varphi,q}_{A,m}(\mathbb{R}^{n})$,
it follows that
\begin{align}\label{e3.1}
\lf\|\mathbf{1}_{\rn}\r\|_{L^\varphi(\rn)}=\lf\|\mathbf{1}_{\Omega_i}\r\|_{L^\varphi(\rn)}
\lesssim\|f\|_{H^{\varphi,q}_{A,m}(\mathbb{R}^{n})}<\infty.
\end{align}
However, by $\vz\in\mathbb{A}_\infty(A)$ and \cite[(v) and (vi) of Lemma 1.1.3]{ylk17},
we know that there exists a $\mu\in(1,\infty]$ such that, for any dilated ball
$B\in\mathfrak{B}(\rn)$, any measurable subset $E\subset(x+B)$, and $t\in(0,\infty)$,
$$\frac{\varphi(x+B,t)}{\varphi(E,t)}\gtrsim\lf(\frac{|x+B|}{|E|}\r)^{\frac{\mu-1}{\mu}}.$$
From this, we deduce that, for any $\lambda\in(0,\infty)$ and $k\in\zz$,
\begin{align*}
\int_{\rn}\varphi\lf(x,\frac1{\lambda}\r)dx
&\geq\lim_{k\rightarrow\infty}\int_{0+B_k}\varphi\lf(x,\frac1{\lambda}\r)dx
=\lim_{k\rightarrow\infty}\varphi\lf(0+B_k,\frac1{\lambda}\r)\\
&\gtrsim\lim_{k\rightarrow\infty}b^{\frac{(k-1)(\mu-1)}{\mu}}\varphi\lf(0+B_1,\frac1{\lambda}\r)=\infty,
\end{align*}
which further implies that
$$\lf\|\mathbf{1}_{\rn}\r\|_{L^\varphi(\rn)}=
\inf\lf\{ \lz\in(0,\fz): \int_\rn \vz\lf(x, \frac{1}{\lz}\r)dx\le 1\r\}=\infty.$$
This contradicts \eqref{e3.1}. Therefore, $\Omega_i\subsetneqq\rn$.
Moreover, by the fact that $f^*_m$ is lower semi-continuous (see, for instance, \cite[Proposition 3.5]{b03}),
we know that, for any $i\in\zz$, $\Omega_i$ is an open subset of $\rn$.
This completes the proof of Lemma \ref{l3.1}.
\end{proof}

The following lemma comes from \cite[Lemma 6.6]{b03}.
\begin{lemma}\label{l3.2}
Let $\Phi\in\cs(\rn)$ and $m\in\zz_+$. Then there exists a positive constant $C$ such that,
for any $f\in\cs'(\rn)$ and $x\in\rn$,
$$\sup_{t\in(0,\infty)}\lf(f\ast\Phi_t\r)^*_{m+1}(x)\leq Cf^*_m(x).$$
\end{lemma}

The following Whitney decomposition theorem is just from \cite[Lemma 2.7]{b03}.
\begin{lemma}\label{l3.3}
Let $\Omega\subset\rn$ be an open set with $|\Omega|<\infty$.
Then, for any $m\in\zz_+$, there exist a sequence of point $\{x_j\}_{j\in\nn}\subset\Omega$
and a sequence of integers $\{\ell_j\}_{j\in\nn}\subset\zz$ such that
\begin{enumerate}
\item[\rm{(i)}] $\Omega=\bigcup_{j\in\nn}(x_j+B_{\ell_j})$;
\item[\rm{(ii)}] $\{x_j+B_{\ell_j-\tau}\}_{j\in\nn}$ are pairwise disjoint,
where $\tau$ is as in \eqref{e1.4} and \eqref{e1.5};
\item[\rm{(iii)}] for any $j\in\nn$, $(x_j+B_{\ell_j+m})\cap\Omega^\complement=\emptyset$,
but $(x_j+B_{\ell_j+m+1})\cap\Omega^\complement\neq\emptyset$;
\item[\rm{(iv)}] if $(x_i+B_{\ell_i+m-2\tau})\cap(x_j+B_{\ell_j+m-2\tau})\neq\emptyset$, then $|\ell_i-\ell_j|\leq\tau$;
\item[\rm{(v)}] for any $i\in\nn$, $\sharp\{j\in\nn:\ (x_i+B_{\ell_i+m-2\tau})
\cap(x_j+B_{\ell_j+m-2\tau})\neq\emptyset\}\leq L$, where $L$ is a positive constant independent of $\Omega, f$ and $i$.
\end{enumerate}
\end{lemma}

The following lemma comes from \cite[Lemma 3.21]{jwyyz23}.
\begin{lemma}\label{l3.4}
Let $q\in(0,\infty]$ and $\vz\in\mathbb{A}_\infty(A)$ be a Musielak-Orlicz function with $0<i(\varphi)\leq I(\varphi)<\infty$,
where $i(\varphi)$ and $I(\varphi)$ are, respectively, as in \eqref{e1.1} and \eqref{e1.2}.
Then there exists a positive constant $\theta\in(0,1]$ such that, for any sequence
$\{f_k\}_{k\in\zz}\subset L^{\varphi,q}(\rn)$,
$$\lf\|\sum_{k\in\zz}|f_k|\r\|^\theta_{L^{\varphi,q}(\rn)}\leq 4\sum_{k\in\zz}\lf\|f_k\r\|^\theta_{L^{\varphi,q}(\rn)}.$$
\end{lemma}

The following lemma shows that anisotropic $(\varphi,r,s,\ez)$-molecule $M$ are in $H^\varphi_{A,m}(\rn)$,
which generalizes \cite[Lemma 32]{lyy14} by improving the assumption that $\varphi$ is a growth function
to the full range $0<i(\varphi)\leq I(\varphi)<\infty$.
\begin{lemma}\label{l3.5}
Let $\vz\in\mathbb{A}_\infty(A)$ be a Musielak-Orlicz function with $0<i(\varphi)\leq I(\varphi)<\infty$,
$r\in(\max\{q(\varphi),I(\varphi)\},\infty]$, $s\in[m(\varphi),\infty)\cap\zz_+$, $m\in[s,\infty)\cap\zz_+$,
and $\varepsilon\in(\frac{\ln b}{\ln \lambda_-}+s+1,\infty)$,
where $q(\varphi)$, $i(\varphi)$, $I(\varphi)$, and $m(\varphi)$ are,
respectively, as in \eqref{e1.6}, \eqref{e1.1}, \eqref{e1.2}, and \eqref{e1.7}.
Then there exists a positive constant $C$ such that, for any anisotropic $(\varphi,r,s,\ez)$-molecule $M$,
\begin{align}\label{e3.2}
\|M\|_{H^\varphi_{A,m}(\rn)}\leq C.
\end{align}
\end{lemma}
\begin{proof}
We first show that, for any anisotropic $(\varphi,r,s)$-atom $a$ associated with some $x_0+B_j$,
\begin{align}\label{e3.3}
\int_\rn\varphi\lf(x,a^*_m(x)\r)dx\lesssim\varphi\lf(x_0+B_j,\|a\|_{L^r_\varphi(x_0+B_j)}\r)
\end{align}
and hence $\|a\|_{H^\varphi_{A,m}(\rn)}\lesssim1$. Since any anisotropic $(\varphi,\infty,s)$-atom is also
an anisotropic $(\varphi,r,s)$-atom, we just consider the case $r\in(\max\{q(\varphi),I(\varphi)\},\infty)$.
Now we can write
\begin{align}\label{e3.4}
\int_\rn\varphi\lf(x,a^*_m(x)\r)dx&=\int_{x_0+B_{j+\tau}}\varphi\lf(x,a^*_m(x)\r)dx
+\int_{(x_0+B_{j+\tau})^\complement}\varphi\lf(x,a^*_m(x)\r)dx\nonumber\\
&=:{\rm{I_{1}+I_{2}}},
\end{align}
where $\tau$ is as in \eqref{e1.4} and \eqref{e1.5}. We first deal with ${\rm{I_{1}}}$.
By $r\in(q(\varphi),\infty)$ and Lemma \ref{l2.1}(ii), we know that $\vz\in\mathbb{A}_r(A)$.
Meanwhile, let $p^+_\varphi\in(I(\varphi),r)$. Then, by Lemma \ref{l2.1}(i), we conclude that
$\varphi$ is of uniformly upper type $p^+_\varphi$, which together with the H\"{o}lder inequality for
$\frac{p^+_\varphi}{r}+\frac{r-p^+_\varphi}{r}=1$, the fact that $a^*_m\lesssim M_{\mathrm{HL}}(a)$,
$\vz\in\mathbb{A}_r(A)$, Lemma \ref{l2.1}(ii), and Lemma \ref{l2.2}(i), further implies that
\begin{align}\label{e3.5}
{\rm{I_{1}}}&=\int_{x_0+B_{j+\tau}}\varphi\lf(x,\frac{a^*_m(x)
\|a\|_{L^r_\varphi(x_0+B_{j})}}{\|a\|_{L^r_\varphi(x_0+B_{j})}}\r)dx\nonumber\\
&\lesssim\int_{x_0+B_{j+\tau}}\lf\{1+\lf[\frac{a^*_m(x)}{\|a\|_{L^r_\varphi(x_0+B_{j})}}\r]^{p^+_\varphi}\r\}
\varphi\lf(x,\|a\|_{L^r_\varphi(x_0+B_{j})}\r)dx \nonumber\\
&\lesssim\|a\|^{-p^+_\varphi}_{L^r_\varphi(x_0+B_{j})}
\lf\{\int_{x_0+B_{j+\tau}}\lf[a^*_m(x)\r]^r\varphi\lf(x,\|a\|_{L^r_\varphi(x_0+B_{j})}\r)dx\r\}^{\frac{p^+_\varphi}{r}}\nonumber\\
&\hs\times\lf[\varphi\lf(x_0+B_{j+\tau},\|a\|_{L^r_\varphi(x_0+B_{j})}\r)\r]^{\frac{r-p^+_\varphi}{r}}
+\varphi\lf(x_0+B_{j+\tau},\|a\|_{L^r_\varphi(x_0+B_{j})}\r)\nonumber\\
&\lesssim\|a\|^{-p^+_\varphi}_{L^r_\varphi(x_0+B_{j})}
\lf\{\int_{x_0+B_{j+\tau}}\lf[M_{\mathrm{HL}}(a)(x)\r]^r\varphi\lf(x,\|a\|_{L^r_\varphi(x_0+B_{j})}\r)dx\r\}^{\frac{p^+_\varphi}{r}}\nonumber\\
&\hs\times\lf[\varphi\lf(x_0+B_{j+\tau},\|a\|_{L^r_\varphi(x_0+B_{j})}\r)\r]^{\frac{r-p^+_\varphi}{r}}
+\varphi\lf(x_0+B_{j+\tau},\|a\|_{L^r_\varphi(x_0+B_{j})}\r)\nonumber\\
&\lesssim\|a\|^{-p^+_\varphi}_{L^r_\varphi(x_0+B_{j})}
\lf\{\int_{x_0+B_{j}}\lf[a(x)\r]^r\varphi\lf(x,\|a\|_{L^r_\varphi(x_0+B_{j})}\r)dx\r\}^{\frac{p^+_\varphi}{r}}\nonumber\\
&\hs\times\lf[\varphi\lf(x_0+B_{j},\|a\|_{L^r_\varphi(x_0+B_{j})}\r)\r]^{\frac{r-p^+_\varphi}{r}}
+\varphi\lf(x_0+B_{j},\|a\|_{L^r_\varphi(x_0+B_{j})}\r)\nonumber\\
&\lesssim\varphi\lf(x_0+B_{j},\|a\|_{L^r_\varphi(x_0+B_{j})}\r),
\end{align}
which implies the desired result.

Next, we deal with ${\rm{I_{2}}}$. By an argument similar to that used in the estimation of
\cite[(101), Lemma 32]{lyy14} (which is independent of the assumption that $\varphi$ is an anisotropic growth function),
we obtain that
$${\rm{I_{2}}}\lesssim\varphi\lf(x_0+B_{j},\|a\|_{L^r_\varphi(x_0+B_{j})}\r),$$
which, combined with \eqref{e3.4} and \eqref{e3.5}, further implies that \eqref{e3.3} holds true.
Furthermore, from \eqref{e3.3}, $\|a\|_{L^r_\varphi(x_0+B_{j})}\leq\|\mathbf{1}_{x_0+B_{j}}\|^{-1}_{L^\varphi(\rn)}$,
and Lemma \ref{l2.6}, we deduce that
$$\int_\rn\varphi\lf(x,a^*_m(x)\r)dx\lesssim\varphi\lf(x_0+B_{j},\|a\|_{L^r_\varphi(x_0+B_{j})}\r)
\lesssim\varphi\lf(x_0+B_{j},\lf\|\mathbf{1}_{x_0+B_{j}}\r\|^{-1}_{L^\varphi(\rn)}\r)\sim1,$$
which, together with \cite[Lemma 1.1.11(i)]{ylk17}, further implies that $\|a\|_{H^\varphi_{A,m}(\rn)}\lesssim1$.

Finally, we show \eqref{e3.2}. Let $M$ be an anisotropic $(\varphi,r,s,\ez)$-molecule associated with dilated balls
$x_0+B_{j}\in\mathfrak{B}(\rn)$. Then, by checking the proof of \cite[Lemma 3.7]{lffy16} (see also \cite[(12)]{sll19}), we know
that, for any $l\in\zz_+$, there exists a sequence of multiples of anisotropic $(\varphi,r,s)$-atoms $\{a_{l}\}_{l\in\zz_+}$
associated with dilated balls $\{x_{0}+B_{j+l}\}_{l\in\zz_+}$, such that $M=\sum_{l\in\zz_+}a_l$ both pointwisely on $\rn$ and in $\cs'(\rn)$, and
$$\lf\|a_l\r\|_{L^r_\varphi(x_{0}+B_{j+l})}\lesssim b^{-l\varepsilon}\lf\|\mathbf{1}_{x_{0}+B_{j+l}}\r\|^{-1}_{L^\varphi(\rn)},$$
which together with Lemma \ref{l3.4} and $\|a\|_{H^\varphi_{A,m}(\rn)}\lesssim1$, further implies that there exists a
$\theta\in(0,1]$ such that
$$\|M\|^\theta_{H^\varphi_{A,m}(\rn)}=\bigg\|\sum_{l\in\zz_+}a_l\bigg\|^\theta_{H^\varphi_{A,m}(\rn)}
\lesssim\sum_{l\in\zz_+}\lf\|a_l\r\|^\theta_{H^\varphi_{A,m}(\rn)}\lesssim \sum_{l\in\zz_+}b^{-l\varepsilon\theta}\sim1.$$
This finishes the proof of \eqref{e3.2} and hence of Lemma \ref{l3.5}.
\end{proof}

Now, we show Theorem \ref{t1.3} via using Theorem \ref{t1.2} and Lemmas \ref{l3.3}, \ref{l3.4} and \ref{l3.5}.
\begin{proof}[Proof of Theorem \ref{t1.3}]
By the facts that, for any $m\geq s\geq\lfloor\frac{\ln b}{\ln\lambda_-}\frac{q(\varphi)}{i(\varphi)}\rfloor$,
$H^{\varphi,q}_{A,m}(\mathbb{R}^{n})\subset H^{\varphi,q}_{A,m+1}(\mathbb{R}^{n})$,
$\mathcal{H}^{\varphi,\infty}_{A,m}(\mathbb{R}^{n})\subset \mathcal{H}^{\varphi,\infty}_{A,m+1}(\mathbb{R}^{n})$,
and any anisotropic $(\varphi,\infty,s+1)$-atom is also an anisotropic $(\varphi,\infty,s)$-atom.
Therefore, to show \eqref{e1.12} and \eqref{e1.13}, without loss of generality, we may assume that
\begin{align}\label{eq3.6}
m\geq s\geq\max\lf\{\lf\lfloor\frac{\ln b}{\ln\lambda_-}\frac{q(\varphi)}{i(\varphi)}\r\rfloor,
\lf\lfloor\frac{\ln b}{\ln\lambda_-}\frac{q(\varphi)I(\varphi)}{i(\varphi)}\r\rfloor\r\}.
\end{align}
Let $f\in H^{\varphi,q}_{A,m}(\mathbb{R}^{n})$ and, for any $k\in\zz$, let $f^k:=f\ast\Phi_{-k}$, where $\Phi\in\cs(\rn)$
satisfies $\int_\rn\Phi(x)dx=1$. Then, by \cite[Theorem 3.13]{sw71}, we know that, for any $k\in\zz$, $f^k\in[C^\infty(\rn)\cap\cs'(\rn)]$.
Moreover, from  \cite[Lemma 3.8]{b03}, it follows that $f^k\rightarrow f$ in $\cs'(\rn)$ as $k\rightarrow\infty$.
By this and Lemma \ref{l3.2}, we can find that, for any $k\in\nn$ and $x\in\rn$,
\begin{align}\label{eq3.7}
\lf(f^k\r)^*_{m+1}(x)\lesssim f^*_m(x),
\end{align}
which, together with $f\in H^{\varphi,q}_{A,m}(\mathbb{R}^{n})$, further implies that, for any $k\in\nn$,
\begin{align}\label{eq3.8}
f^k\in H^{\varphi,q}_{A,m+1}(\rn)
\quad \mathrm{and} \quad
\lf\|f^k\r\|_{H^{\varphi,q}_{A,m+1}(\rn)}\lesssim \|f\|_{H^{\varphi,q}_{A,m}(\rn)}.
\end{align}
Then we show the present theorem through the following two steps.

{\bf Step I}: In this step, we show that, for any $k\in\nn$,
\begin{align}\label{eq3.9}
f^k=\sum_{i\in\zz}\sum_j\hbar^k_{i,j}
\quad \mathrm{in} \ \cs'(\rn),
\end{align}
where, for any $i\in\zz$ and $j$, $\hbar^k_{i,j}$ is an anisotropic $(\varphi,\infty,s)$-atom multiplied by a positive
constant depending on both $i$ and $j$, but independent of both $f$ and $k$.
To establish \eqref{eq3.9}, we develop three key claims by adapting techniques from \cite[Theorem 6.4]{b03}.
We now present the first claim. For any $i\in\zz$, let
\begin{align}\label{eq3.10}
\Omega_i:=\lf\{x\in\rn:f^*_m(x)>2^i\r\}.
\end{align}
Using this and Lemma \ref{l3.1}, we conclude that, for any $i\in\zz$, $\Omega_i$ is a proper open subset of $\rn$,
which, combined with Lemma \ref{l3.3}, further implies that, for any $i\in\zz$,
there exist a sequence of point $\{x_{i,j}\}_{j\in\nn}\subset\Omega$
and a sequence of integers $\{\ell_{i,j}\}_{j\in\nn}\subset\zz$ such that
\begin{enumerate}
\item[\rm{(i)}] $\Omega_i=\bigcup_{j\in\nn}(x_{i,j}+B_{\ell_{i,j}})$;
\item[\rm{(ii)}] $(x_{i,j}+B_{\ell_{i,j}-\tau})\cap (x_{i,k}+B_{\ell_{i,k}-\tau}=\emptyset)$
for any $j,k\in\nn$ with $j\neq k$,
where $\tau$ is as in \eqref{e1.4} and \eqref{e1.5};
\item[\rm{(iii)}] for any $j\in\nn$, $(x_{i,j}+B_{\ell_{i,j}+6\tau})\cap\Omega_i^\complement=\emptyset$,
but $(x_{i,j}+B_{\ell_{i,j}+6\tau+1})\cap\Omega_i^\complement\neq\emptyset$;
\item[\rm{(iv)}] if $(x_{i,j}+B_{\ell_{i,j}+4\tau})\cap(x_{i,k}+B_{\ell_{i,k}+4\tau})\neq\emptyset$, then $|\ell_{i,j}-\ell_{i,k}|\leq\tau$;
\item[\rm{(v)}] for any $j\in\nn$, $\sharp\{k\in\nn:\ (x_{i,j}+B_{\ell_{i,j}+4\tau})
\cap(x_{i,k}+B_{\ell_{i,k}+4\tau})\neq\emptyset\}\leq R$, where $R$ and $\tau$ are same as in Lemma \ref{l3.3}.
\end{enumerate}

Next, we present the second claim. Fix a function $\eta\in C_\mathrm{c}^\infty(\rn)$ such that $\supp(\eta)\subset B_\tau$,
$\eta\in[0,1]$, and $\eta\equiv1$ on $B_0$. For any $i\in\zz$, $j\in\nn$, and $x\in\rn$, let
\begin{align}\label{eq3.11}
\eta_{i,j}(x):=\eta\lf(A^{-\ell_{i,j}}\lf(x-x_{i,j}\r)\r)
\quad \mathrm{and} \quad
\xi_{i,j}(x):=
\begin{cases}
\frac{\eta_{i,j}(x)}{\sum_{j\in\nn}\eta_{i,j}(x)}\  & \text{if}\  x\in\Omega_i,\\
0 \  & \text{if}\  x\in\Omega_i^\complement.
\end{cases}
\end{align}
Then $\xi_{i,j}\in\cs(\rn)$, $\supp(\xi_{i,j})\subset(x_{i,j}+B_{\ell_{i,j}+\tau})$,
$\xi_{i,j}\in[0,1]$, $\xi_{i,j}\equiv1$ on $x_{i,j}+B_{\ell_{i,j}-\tau}$ by the above item (ii),
and $\sum_{j\in\nn}\xi_{i,j}=\mathbf{1}_{\Omega_i}$, and hence the family $\{\xi_{i,j}\}_{j\in\nn}$
forms a smooth partition of unity of $\Omega_i$.

Finally, we present the third claim. Let $\mathcal{P}_s(\rn)$ denote the linear space of all the polynomials on $\rn$ of
total degree not greater than $s$ and, for any $i\in\zz$ and $j$, define the inner product $(\cdot,\cdot)_{i,j}$
by setting, for any $P,Q\in\mathcal{P}_s(\rn)$,
\begin{align}\label{eq3.12}
(P,Q)_{i,j}:=\frac1{\int_\rn\xi_{i,j}(x)dx}\int_\rn P(x)\overline{Q(x)}\xi_{i,j}(x)dx,
\end{align}
where, for any $z\in\mathbb{C}$, we denote its complex conjugate by $\overline{z}$. Then it is easy to know that
$(\mathcal{P}_s(\rn),(\cdot,\cdot)_{i,j})$ is a finite dimensional Hilbert space.
For any $k\in\nn$, since $f^k$ induces a linear function on $\mathcal{P}_s(\rn)$, that is, for any $Q\in\mathcal{P}_s(\rn)$,
$$Q\mapsto\frac1{\int_\rn\xi_{i,j}(x)dx}\int_\rn f^k(x)Q(x)\xi_{i,j}(x)dx,$$
then, it follows from the Riesz theorem on Hilbert spaces that there exists a unique polynomial $P^k_{i,j}\in\mathcal{P}_s(\rn)$ such that, for any $Q\in\mathcal{P}_s(\rn)$,
\begin{align}\label{eq3.13}
\lf(Q,\overline{P^k_{i,j}}\r)_{i,j}=\frac1{\int_\rn\xi_{i,j}(x)dx}\int_\rn f^k(x)\overline{Q(x)}\xi_{i,j}(x)dx.
\end{align}
Moreover, for any $k\in\nn$, $i\in\zz$, $j$, and $x\in\rn$, let
\begin{align}\label{eq3.14}
a^k_{i,j}(x):=\lf[f^k(x)-P^k_{i,j}(x)\r]\xi_{i,j}(x).
\end{align}
Then, by \eqref{eq3.12} and \eqref{eq3.13}, we know that, for any $k\in\nn$, $i\in\zz$, $j$, and $Q\in\mathcal{P}_s(\rn)$,
\begin{align}\label{eq3.15}
\int_\rn a^k_{i,j}(x)Q(x)dx=0.
\end{align}

Next, to show \eqref{eq3.9}, we establish a Calder\'{o}n-Zygmund decomposition of $f^k$ of degree $s$.
Indeed, by \cite[Lemma 6.2]{b03} (see also \cite[Lemma 35]{lyy14}), we conclude that, for any $k\in\nn$, $i\in\zz$, $j$, and $x\in\rn$,
\begin{align}\label{eq3.16}
\sup_{x\in\rn}\lf|P^k_{i,j}(x)\xi_{i,j}(x)\r|\lesssim2^i.
\end{align}
Moreover, from \eqref{eq3.11} and the above item (v), we deduce that, for any $k\in\nn$ and $i\in\zz$,
$\supp(a^k_{i,j})\subset(x_{i,j}+B_{\ell_{i,j}+\tau})$ and $\sum_j\mathbf{1}_{x_{i,j}+B_{\ell_{i,j}+\tau}}(x)\leq R$,
which, combined with \eqref{eq3.14}, \eqref{eq3.16}, $\xi_{i,j}\in[0,1]$, and $f^k\in[C^\infty(\rn)\cap\cs'(\rn)]$,
implies that, for any $k\in\nn$, $i\in\zz$, and $x\in\rn$,
$$\sum_j\lf|a^k_{i,j}(x)\r|=\sum_j\lf|f^k(x)-P^k_{i,j}(x)\r|\lf|\xi_{i,j}(x)\r|<\infty,$$
this further implies that $\sum_ja^k_{i,j}$ exists everywhere on $\rn$. By the facts that
$\supp(a^k_{i,j})\subset(x_{i,j}+B_{\ell_{i,j}+\tau})$, $\sum_j\mathbf{1}_{x_{i,j}+B_{\ell_{i,j}+\tau}}(x)\leq R$,
$f^k\in[C^\infty(\rn)\cap\cs'(\rn)]$, \eqref{eq3.14}, $\xi_{i,j}\in[0,1]$, \eqref{eq3.11}, \eqref{eq3.16},
and the above item (i), we find that, for any $k\in\nn$, $i\in\zz$, and $\Phi\in\cs(\rn)$,
\begin{align*}
\sum_j\lf|\lf\langle a^k_{i,j},\Phi\r\rangle\r|
&\leq\sum_j\int_{x_{i,j}+B_{\ell_{i,j}+\tau}}\lf|\lf[f^k(x)-P^k_{i,j}(x)\r]\xi_{i,j}(x)\Phi(x)\r|dx\\
&\leq\int_\rn\lf|f^k(x)\Phi(x)\r|dx+\int_\rn\sum_j\lf|P^k_{i,j}(x)\xi_{i,j}(x)\Phi(x)\r|dx\\
&\lesssim\int_\rn\lf|f^k(x)\Phi(x)\r|dx+2^i\int_{\Omega_i}\lf|\Phi(x)\r|dx<\infty,
\end{align*}
which implies that $\sum_ja^k_{i,j}$ converges in $\cs'(\rn)$.
Thus, by the fact that $\sum_{j}\xi_{i,j}=\mathbf{1}_{\Omega_i}$, we can define, for any $k\in\nn$ and $i\in\zz$,
\begin{align}\label{eq3.17}
h^k_i:=f^k-\sum_ja^k_{i,j}=f^k-\sum_j\lf[f^k-P^k_{i,j}\r]\xi_{i,j}
=f^k\mathbf{1}_{\Omega_i^\complement}+\sum_jP^k_{i,j}\xi_{i,j}
\end{align}
both in $\cs'(\rn)$ and almost everywhere on $\rn$, which implies that, for any $k\in\nn$ and $i\in\zz$,
\begin{align}\label{eq3.18}
f^k=h^k_i+\sum_ja^k_{i,j}
\end{align}
both in $\cs'(\rn)$ and almost everywhere on $\rn$.

Now, we show that, for any $k\in\nn$,
\begin{align}\label{eq3.19}
f^k=\sum_{i\in\zz}\lf(h^k_{i+1}-h^k_i\r)
\quad \mathrm{in} \ \cs'(\rn).
\end{align}
Indeed, by \eqref{eq3.18} with $i:=M_2+1$ and $M_2\in\nn$, we conclude that, for any $k,M_1,M_2\in\nn$,
\begin{align}\label{eq3.20}
\lf\|f^k-\sum^{M_2}_{i=-M_1}\lf(h^k_{i+1}-h^k_i\r)\r\|_{H^{\varphi_{\widetilde{p}}}
_{A,m+1}(\rn)+L^\infty(\rn)}
&=\lf\|h^k_{M_2+1}+\sum_ja^k_{M_2+1,j}
-\sum^{M_2}_{i=-M_1}\lf(h^k_{i+1}-h^k_i\r)\r\|_{H^{\varphi_{\widetilde{p}}}_{A,m+1}(\rn)+L^\infty(\rn)}\nonumber\\
&\leq\lf\|\sum_ja^k_{M_2+1,j}\r\|_{H^{\varphi_{\widetilde{p}}}_{A,m+1}(\rn)}
+\lf\|h^k_{-M_1}\r\|_{L^\infty(\rn)}
=:{\rm{II_1}+\rm{II_2}},
\end{align}
where $\varphi_{\widetilde{p}}$ is the same as in Lemma \ref{l2.9} with
$\theta$ therein replaced by a $\widetilde{p}\in(0,1)$ specified in \eqref{eq3.21}
below and $H^{\varphi_{\widetilde{p}}}_{A,m+1}(\rn)+L^\infty(\rn)$ denotes the set of all the $F\in\cs'(\rn)$ such that
\begin{align*}
&\|F\|_{H^{\varphi_{\widetilde{p}}}_{A,m+1}(\rn)+L^\infty(\rn)}\\
&:=\inf\lf\{\|F_1\|_{H^{\varphi_{\widetilde{p}}}_{A,m+1}(\rn)}+\|F_2\|_{L^\infty(\rn)}:F=F_1+F_2,
F_1\in H^{\varphi_{\widetilde{p}}}_{A,m+1}(\rn), F_2\in L^\infty(\rn)\r\}
\end{align*}
with the infimum being taken over all decompositions of $f$ as above.

We first deal with the term $\rm{II_1}$. To this end, for any $k\in\nn$ and $i\in\zz$, let
$$\widetilde{\Omega}_i:=\lf\{x\in\rn:\lf(f^k\r)^*_{m+1}(x)>2^i\r\}.$$
Moreover, from \eqref{eq3.6}, it follows that
$$m_s:=\min\{s+1,m+1\}>\max\lf\{\frac{\ln b}{\ln\lambda_-}\frac{q(\varphi)}{i(\varphi)},
\frac{\ln b}{\ln\lambda_-}\frac{q(\varphi)I(\varphi)}{i(\varphi)}\r\},$$
which yields that there exist $q_0\in(q(\varphi),\infty)$ and $p_0\in(0,i(\varphi))$ such that
$$m_s>\max\lf\{\frac{\ln b}{\ln\lambda_-}\frac{q_0}{p_0},
\frac{\ln b}{\ln\lambda_-}\frac{q_0I(\varphi)}{p_0}\r\},$$
this further implies that $m_sp_0>\frac{q_0\ln b}{\ln\lambda_-}$. Moreover, we can choose
\begin{align}\label{eq3.21}
\widetilde{p}\in(0,1)\cap\lf(\frac{q_0\ln b}{m_sp_0\ln\lambda_-},\frac{1}{I(\varphi)}\r).
\end{align}
Then we have $m_sp_0\widetilde{p}>\frac{q_0\ln b}{\ln\lambda_-}$, which, together with \eqref{eq3.21} and Lemma \ref{l2.9},
further implies that $\varphi_{\widetilde{p}}\in\mathbb{A}_{m_sp_0/(\frac{\ln b}{\ln\lambda_-})}(A)$,
$\varphi_{\widetilde{p}}$ is an anisotropic growth function uniformly lower type $p_0\widetilde{p}\in(0,1)$
(see \cite[Definition 3]{lyy14} for its definition), and $m_s>\frac{\ln b}{\ln\lambda_-}\frac{q(\varphi_{\widetilde{p}})}{i(\varphi_{\widetilde{p}})}$.
Thus, repeating the proof process of \cite[Lemma 24]{lyy14} with $m,\varphi,s,f,b_j,\lambda$, and $\Omega$ there in replaced,
respectively, by $m+1,\varphi_{\widetilde{p}},m_s,\frac{f^k}{\alpha},\frac{a^k_{M_2,j}}{\alpha},2^{M_2}$, and $\widetilde{\Omega}_{M_2}$ here, we obtain that, for any $k,M_2\in\nn$ and $\alpha\in(0,\infty)$,
\begin{align}\label{eq3.22}
\int_\rn\varphi_{\widetilde{p}}\lf(x,\frac{\lf(\sum_ja^k_{M_2,j}\r)^*_{m+1}(x)}{\alpha}\r)dx
\lesssim\int_{\widetilde{\Omega}_{M_2}}\varphi_{\widetilde{p}}\lf(x,\frac{\lf(f^k\r)^*_{m+1}(x)}{\alpha}\r)dx,
\end{align}
and then using Lemma \ref{l2.6} and \cite[Lemma 1.1.11]{ylk17} and letting
$$\alpha:=\lf\|\lf(f^k\r)^*_{m+1}\mathbf{1}_{\widetilde{\Omega}_{M_2}}\r\|_{L^{\varphi_{\widetilde{p}}}(\rn)}$$
in \eqref{eq3.22}, we obtain that, for any $k,M_2\in\nn$,
\begin{align*}
\lf\|\lf(\sum_ja^k_{M_2,j}\r)^*_{m+1}\r\|_{L^{\varphi_{\widetilde{p}}}(\rn)}
\lesssim\lf\|\lf(f^k\r)^*_{m+1}\mathbf{1}_{\widetilde{\Omega}_{M_2}}\r\|_{L^{\varphi_{\widetilde{p}}}(\rn)},
\end{align*}
which, combined with Lemma \ref{l3.4}, the fact that $L^{\varphi,q_1}(\rn)$ is continuously embedded into $L^{\varphi,q_2}(\rn)$
with $q_1,q_2\in(0,\infty]$ satisfy $q_1<q_2$, \eqref{eq3.8}, and $\tilde{p}\in(0,1)$, further implies that there exists
a $\theta\in(0,1]$ such that, for any $k,M_2\in\nn$,
\begin{align}\label{eq3.23}
\lf[\rm{II_1}\r]^\theta&=\lf\|\sum_ja^k_{M_2+1,j}\r\|^\theta_{H^{\varphi_{\widetilde{p}}}_{A,m+1}(\rn)}
=\lf\|\lf(\sum_ja^k_{M_2,j}\r)^*_{m+1}\r\|^\theta_{L^{\varphi_{\widetilde{p}}}(\rn)}\nonumber\\
&\lesssim\lf\|\lf(f^k\r)^*_{m+1}\mathbf{1}_{\widetilde{\Omega}_{M_2}}\r\|^\theta_{L^{\varphi_{\widetilde{p}}}(\rn)}
\lesssim\lf\|\sum^\infty_{l=M_2}\lf(f^k\r)^*_{m+1}
\mathbf{1}_{\widetilde{\Omega}_{l}\backslash\widetilde{\Omega}_{l+1}}\r\|^\theta_{L^{\varphi_{\widetilde{p}}}(\rn)}\nonumber\\
&\lesssim\sum^\infty_{l=M_2}\lf\|2^l\mathbf{1}_{\widetilde{\Omega}_{l}\backslash\widetilde{\Omega}_{l+1}}
\r\|^\theta_{L^{\varphi_{\widetilde{p}}}(\rn)}
\lesssim\sum^\infty_{l=M_2}\lf\|2^{l\tilde{p}}\mathbf{1}_{\widetilde{\Omega}_{l}}
\r\|^{\frac{\theta}{\tilde{p}}}_{L^\varphi(\rn)}\nonumber\\
&=\sum^\infty_{l=M_2}2^{\frac{l(\tilde{p}-1)\theta}{\tilde{p}}}\lf\|2^l\mathbf{1}_{\widetilde{\Omega}_{l}}
\r\|^{\frac{\theta}{\tilde{p}}}_{L^\varphi(\rn)}
\lesssim\sum^\infty_{l=M_2}2^{\frac{l(\tilde{p}-1)\theta}{\tilde{p}}}\lf\|\lf(f^k\r)^*_{m+1}
\r\|^{\frac{\theta}{\tilde{p}}}_{L^{\varphi,\infty}(\rn)}\nonumber\\
&\lesssim2^{\frac{M_2(\tilde{p}-1)\theta}{\tilde{p}}}\lf\|f^k
\r\|^{\frac{\theta}{\tilde{p}}}_{H^{\varphi,q}_{A,m+1}(\rn)}
\lesssim2^{\frac{M_2(\tilde{p}-1)\theta}{\tilde{p}}}\lf\|f
\r\|^{\frac{\theta}{\tilde{p}}}_{H^{\varphi,q}_{A,m}(\rn)}\rightarrow0
\end{align}
as $M_2\rightarrow\infty$, which is the desired estimate of $\rm{II_1}$.

Next, we deal with the term $\rm{II_2}$. From \eqref{eq3.17}, we deduce that, for any $k,M_1\in\nn$
and almost every $x\in\rn$,
\begin{align}\label{eq3.24}
\lf|h^k_{-M_1}(x)\r|\leq\lf|f^k(x)\mathbf{1}_{\Omega_{-M_1}^\complement}(x)\r|+\lf|\sum_jP^k_{-M_1,j}(x)\xi_{-M_1,j}(x)\r|.
\end{align}
Moreover, by the fact that $f^k\lesssim(f^k)^*_{m+1}$ almost everywhere on $\rn$, \eqref{eq3.7}, and \eqref{eq3.10},
we know that, for any $k,M_1\in\nn$ and $x\in\rn$,
\begin{align}\label{eq3.25}
\lf|f^k(x)\mathbf{1}_{\Omega_{-M_1}^\complement}(x)\r|
\lesssim\lf(f^k\r)^*_{m+1}(x)\mathbf{1}_{\Omega_{-M_1}^\complement}(x)
\lesssim f^*_m(x)\mathbf{1}_{\Omega_{-M_1}^\complement}(x)\leq2^{-M_1}.
\end{align}
Furthermore, by \eqref{eq3.16} with $i:=-M_1$, we find that, for any $k,M_1\in\nn$ and $j$,
\begin{align}\label{eq3.26}
\sup_{x\in\rn}\lf|P^k_{-M_1,j}(x)\xi_{-M_1,j}(x)\r|\lesssim2^{-M_1},
\end{align}
which, together with the fact that, for any $M_1\in\nn$ and $x\in\rn$, the number of $j$ satisfying $\xi_{-M_1,j}(x)\neq0$
is not more than $C_{(n)}$, further implies that, for any $k\in\nn$ and $x\in\rn$,
$$\sum_j\lf|P^k_{-M_1,j}(x)\xi_{-M_1,j}(x)\r|\lesssim2^{-M_1}.$$
From this, \eqref{eq3.24}, and \eqref{eq3.25}, it follows that, for any $k,M_1\in\nn$,
\begin{align}\label{eq3.27}
{\rm{II_2}}=\lf\|h^k_{-M_1}\r\|_{L^\infty(\rn)}\lesssim2^{-M_1}\rightarrow0
\end{align}
as $M_1\rightarrow\infty$, which is the desired estimate of $\rm{II_2}$.
Thus, combining the above estimates \eqref{eq3.20}, \eqref{eq3.23}, and \eqref{eq3.27}, we obtain that
$$\lf\|f^k-\sum^{M_2}_{i=-M_1}\lf(h^k_{i+1}-h^k_i\r)\r\|_{H^{\varphi_{\widetilde{p}}}_{A,m+1}(\rn)+L^\infty(\rn)}\rightarrow0$$
as $M_1,M_2\rightarrow\infty$, which, together with the fact that both $H^{\varphi}_{A,m}(\rn)$ and $L^\infty(\rn)$ are continuously embedded into $\cs'(\rn)$
(see, for instance, \cite[Proposition 6]{lyy14}), further implies that the claim \eqref{eq3.19} holds true.

To complete the proof of the claim \eqref{eq3.9}, for any $k\in\nn$, $i\in\zz$, $j$, and $l$, let $P^k_{i+1,j,l}$
be the the orthogonal projection of $[f^k-P^k_{i+1,l}]\xi_{i,j}$ onto $\mathcal{P}_s(\rn)$ with respect to the inner
product $(\cdot,\cdot)_{i+1,l}$ associated to $\xi_{i+1,j}$ the same as in \eqref{eq3.12}, that is, the unique element
of $\mathcal{P}_s(\rn)$ such that, for any $Q\in\mathcal{P}_s(\rn)$,
\begin{align}\label{eq3.28}
\int_\rn\lf[f^k(x)-P^k_{i+1,l}(x)\r]\xi_{i,j}(x)\overline{Q(x)}\xi_{i+1,l}(x)dx
=\int_\rn P^k_{i+1,j,l}(x)\overline{Q(x)}\xi_{i+1,l}(x)dx.
\end{align}
By this and an an argument similar to the proof of \cite[Lemma 36]{lyy14}, we conclude that, for any $k\in\nn$ and $i\in\zz$,
$$\sum_j\sum_lP^k_{i+1,j,l}\xi_{i+1,l}=0$$
both in $\cs'(\rn)$ and pointwisely on $\rn$, which, combined with \eqref{eq3.17}, $\sum_{j}\xi_{i,j}=\mathbf{1}_{\Omega_i}$,
and $\supp(\sum_la^k_{i+1,l})\subset\Omega_{i+1}\subset\Omega_{i}$, further implies that, for any $k\in\nn$ and $i\in\zz$,
\begin{align}\label{eq3.29}
h^k_{i+1}-h^k_i&=\lf(f^k-\sum_la^k_{i+1,l}\r)-\lf(f^k-\sum_ja^k_{i,j}\r)\nonumber\\
&=\sum_ja^k_{i,j}-\sum_la^k_{i+1,l}\nonumber\\
&=\sum_ja^k_{i,j}-\sum_j\sum_la^k_{i+1,l}\xi_{i,j}+\sum_j\sum_lP^k_{i+1,j,l}\xi_{i+1,l}\nonumber\\
&=\sum_j\lf[a^k_{i,j}-\sum_l\lf(a^k_{i+1,l}\xi_{i,j}-P^k_{i+1,j,l}\xi_{i+1,l}\r)\r]\nonumber\\
&=:\sum_j\hbar^k_{i,j}
\end{align}
both in $\cs'(\rn)$ and almost everywhere on $\rn$. Meanwhile, it follows from \eqref{eq3.14} that, for any $k\in\nn$, $i\in\zz$, and $j$,
\begin{align}\label{eq3.30}
\hbar^k_{i,j}=\lf(f^k-P^k_{i,j}\r)\xi_{i,j}-\sum_l\lf[\lf(f^k-P^k_{i+1,l}\r)\xi_{i,j}-P^k_{i+1,j,l}\r]\xi_{i+1,l}
\end{align}
both in $\cs'(\rn)$ and almost everywhere on $\rn$.

Next we claim that, for any $i\in\zz$ and $j$, $\hbar^k_{i,j}$ is a multiple of an anisotropic $(\varphi,\infty,s)$-atom.
Indeed, by the definition of $P^k_{i+1,j,l}$ with $i\in\zz$, $j$, and $l$, we find that $P^k_{i+1,j,l}=0$ unless
$[x_{i,j}+B_{\ell_{i,j}}]\cap[x_{i+1,l}+B_{\ell_{i+1,l}}]\neq\emptyset$, which, combined with \cite[Lemma 34]{lyy14},
further implies that, for any $i\in\zz$, $j$, and $l$,
$\supp(\xi_{i+1,l})\subset[x_{i+1,l}+B_{\ell_{i+1,l}}]\subset[x_{i,j}+B_{\ell_{i,j}+4\tau}]$.
Using this, \eqref{eq3.30}, and $\supp(\xi_{i,j})\subset(x_{i,j}+B_{\ell_{i,j}+\tau})$,
we obtain that, for any $k\in\nn$,
\begin{align}\label{eq3.31}
\supp(\hbar^k_{i,j})\subset(x_{i,j}+B_{\ell_{i,j}+4\tau}).
\end{align}
Moreover, by the fact that $\sum_{l}\xi_{i+1,l}=\mathbf{1}_{\Omega_{i+1}}$ for any $i\in\zz$ and \eqref{eq3.30}, we conclude that,
for any $k\in\nn$, $i\in\zz$, and $j$,
\begin{align}\label{eq3.32}
\hbar^k_{i,j}=f^k\xi_{i,j}\mathbf{1}_{\Omega_{i+1}^\complement}-P^k_{i,j}\xi_{i,j}+\xi_{i,j}\sum_l
P^k_{i+1,l}\xi_{i+1,l}+\sum_lP^k_{i+1,l}\xi_{i+1,l}
\end{align}
both in $\cs'(\rn)$ and almost everywhere on $\rn$. Furthermore, by a similar proof of
\cite[Lemma 5.1]{b03} (see also \cite[Lemma 35]{lyy14}), we obtain that, for any $k\in\nn$, $i\in\zz$, $j,l$, and $x\in\rn$,
$$\sup_{x\in\rn}\lf|P^k_{i+1,j,l}(x)\xi_{i+1,l}(x)\r|\lesssim2^{i+1},$$
which, combined with \eqref{eq3.32}, \eqref{eq3.25} with $M_1$ replaced by $i+1$, and \eqref{eq3.26} with $M_1$ replaced by $i$,
further implies that, for any $k\in\nn$, $i\in\zz$, and $j$,
\begin{align}\label{eq3.33}
\lf\|\hbar^k_{i,j}\r\|_{L^\infty(\rn)}\lesssim2^i.
\end{align}
Meanwhile, by \eqref{eq3.30}, $\sum_l\mathbf{1}_{\supp(\xi_{i+1,l})}\lesssim1$, \eqref{eq3.15}, and \eqref{eq3.28},
we know that, for any $k\in\nn$ and $Q\in\mathcal{P}_s(\rn)$,
\begin{align}\label{eq3.34}
\int_\rn\hbar^k_{i,j}(x)Q(x)dx=0,
\end{align}
which, together with \eqref{eq3.31} and \eqref{eq3.33}, further implies that,
for any $i\in\zz$ and $j$, $\hbar^k_{i,j}$ is a multiple of an anisotropic $(\varphi,\infty,s)$-atom.
By this, \eqref{eq3.19}, and \eqref{eq3.29}, we conclude that the claim \eqref{eq3.9} holds true.

{\bf Step II}: In this step, we will complete the proof of the present theorem. By \eqref{eq3.33}, we find that
$\{\hbar^k_{i,j}\}_{k\in\nn}$ is bounded in $L^\infty(\rn)$ uniformly in $k\in\nn$, which, combined with the
Alaoglu theorem and the well-known diagonal
rule for series two times, further implies that there exists a sequence
$\{\hbar_{i,j}\}_{i\in\zz,j}\subset L^\infty(\rn)$ and a sequence $\{k_l\}_{l\in\nn}\subset\nn$ such that $k_l\rightarrow\infty$ as
$l\rightarrow\infty$ and, for any $i\in\zz$ and $j$, $\hbar^{k_l}_{i,j}\rightarrow\hbar_{i,j}$ in the weak-$*$ topology of
$L^\infty(\rn)$ as $l\rightarrow\infty$, that is, for any $g\in L^1(\rn)$,
$$\lim_{l\rightarrow\infty}\lf\langle\hbar^{k_l}_{i,j},g\r\rangle=\langle\hbar_{i,j},g\rangle.$$
From this, \eqref{eq3.31}, \eqref{eq3.33}, and \eqref{eq3.34}, it follows that there exists a positive constant
$C_1$, independent of $f$, such that, for any $i\in\zz$ and $j$,
\begin{align}\label{eq3.35}
\supp(\hbar_{i,j})\subset(x_{i,j}+B_{\ell_{i,j}+4\tau}), \quad
\lf\|\hbar_{i,j}\r\|_{L^\infty(\rn)}\leq C_12^i,
\end{align}
and, for any $\gamma\in\zz^n_+$ with $|\gamma|\leq s$,
$$\int_\rn\hbar_{i,j}(x)x^\gamma dx=\lf\langle\hbar_{i,j},x^\gamma\mathbf{1}_{x_{i,j}+B_{\ell_{i,j}+4\tau}}\r\rangle
=\lim_{l\rightarrow\infty}\int_\rn\hbar^{k_l}_{i,j}(x)x^\gamma dx=0.$$
Moreover, by $\cs(\rn)\subset L^1(\rn)$, we deduce that, for any $i\in\zz$, $j$, and $\Phi\in\cs(\rn)$,
\begin{align}\label{eq3.36}
\lim_{l\rightarrow\infty}\lf\langle\hbar^{k_l}_{i,j},\Phi\r\rangle=\lf\langle\hbar_{i,j},\Phi\r\rangle,
\end{align}
which implies that $\lim_{l\rightarrow\infty}\hbar^{k_l}_{i,j}=\hbar_{i,j}$ in $\cs'(\rn)$. Furthermore,
for any $i\in\zz$ and $j$, let
$$\lambda_{i,j}:=C_12^i\lf\|\mathbf{1}_{x_{i,j}+B_{\ell_{i,j}+4\tau}}\r\|_{L^\varphi(\rn)}
\quad \mathrm{and} \quad
a_{i,j}:=\frac{\hbar_{i,j}}{\lambda_{i,j}},$$
where $C_1$ is as in \eqref{eq3.35}. Then $\hbar_{i,j}=\lambda_{i,j}a_{i,j}$ and $a_{i,j}$ is an anisotropic $(\varphi,\infty,s)$-atom
and hence $\hbar_{i,j}$ is a multiple of an anisotropic $(\varphi,\infty,s)$-atom.

Meanwhile, by Lemma \ref{l2.12} and the above items (i) and (v), we know that
\begin{align}\label{eq3.37}
\lf[\sum_{i\in\zz}2^{iq}\lf\|\sum_{j\in\nn}\mathbf{1}_{x_{i,j}+B_{\ell_{i,j}+4\tau}}\r\|^q_{L^{\varphi}(\rn)}\r]^{\frac1{q}}
&\lesssim\lf[\sum_{i\in\zz}2^{iq}\lf\|\sum_{j\in\nn}\mathbf{1}_{x_{i,j}+B_{\ell_{i,j}}}\r\|^q_{L^{\varphi}(\rn)}\r]^{\frac1{q}}\nonumber\\
&\lesssim\lf[\sum_{i\in\zz}2^{iq}\lf\|\mathbf{1}_{\Omega_i}\r\|^q_{L^{\varphi}(\rn)}\r]^{\frac1{q}}
\lesssim \|f\|_{H^{\varphi,q}_{A,m}(\mathbb{R}^{n})}
\end{align}
with the usual interpretation for $q=\infty$. Similarly, when
$f\in\mathcal{H}^{\varphi,\infty}_{A,m}(\mathbb{R}^{n})$, we obtain
\begin{align*}
\lim_{|i|\rightarrow\infty}2^i\lf\|\sum_{j\in\nn}\mathbf{1}_{x_{i,j}+B_{\ell_{i,j}+4\tau}}
\r\|_{L^{\varphi}(\rn)}
\lesssim\lim_{|i|\rightarrow\infty}2^i\lf\|\mathbf{1}_{\Omega_i}\r\|_{L^{\varphi}(\rn)}
\lesssim\lim_{|i|\rightarrow\infty}\lf\|f^*_m\mathbf{1}_{\Omega_i}\r\|_{L^{\varphi,\infty}(\rn)}=0,
\end{align*}
which, together with \eqref{eq3.37}, further implies that \eqref{e1.12} and \eqref{e1.13} hold true.

Thus, to complete the proof of Theorem \ref{t1.3}, we still need to show that
\begin{align}\label{eq3.38}
f=\sum_{i\in\zz}\sum_{j\in\nn}\hbar_{i,j}
\quad \mathrm{in} \ \cs'(\rn),
\end{align}
that is, \eqref{e1.11} holds true. Indeed, applying \eqref{eq3.35} and the above item (v), we obtain that, for any $i\in\zz$,
\begin{align}\label{eq3.39}
F_i:=\sum_{j\in\nn}\hbar_{i,j}
\end{align}
converges both in $\cs'(\rn)$ and almost everywhere on $\rn$. Moreover, since \eqref{eq3.29}, for any $k\in\nn$ and $i\in\zz$, we can define
\begin{align}\label{eq3.40}
F^k_i:=h^k_{i+1}-h^k_i=\sum_{j\in\nn}\hbar^k_{i,j}
\end{align}
both in $\cs'(\rn)$ and almost everywhere on $\rn$. Next, we show that, for any $i\in\zz$,
\begin{align}\label{eq3.41}
\lim_{l\rightarrow\infty}F^{k_l}_i=F_i
\quad \mathrm{in} \ \cs'(\rn).
\end{align}
Indeed, from \eqref{eq3.29}, we deduce that, for any $l\in\nn$, $i\in\zz$, and $\Phi\in\cs(\rn)$,
\begin{align}\label{eq3.42}
\lf\langle F^{k_l}_i,\Phi\r\rangle=\sum_{j\in\nn}\lf\langle \hbar^{k_l}_{i,j},\Phi\r\rangle
=\sum_{j\in\nn}\int_\rn\hbar^{k_l}_{i,j}(x)\Phi(x)dx.
\end{align}
Furthermore, from \eqref{eq3.33}, \eqref{eq3.31}, and the above item (v), it follows that, for any $l\in\nn$, $i\in\zz$, and $\Phi\in\cs(\rn)$,
$$\sum_{j\in\nn}\lf|\int_\rn\hbar^{k_l}_{i,j}(x)\Phi(x)dx\r|
\lesssim\sum_{j\in\nn}\int_{x_{i,j}+B_{\ell_{i,j}+4\tau}}\lf|2^i\Phi(x)\r|dx<\infty,$$
which, together with \eqref{eq3.42}, the Lebesgue dominated convergence theorem, and
\eqref{eq3.36}, implies that, for any $i\in\zz$,
\begin{align*}
\lim_{l\rightarrow\infty}\lf\langle F^{k_l}_i,\Phi\r\rangle
&=\lim_{l\rightarrow\infty}\sum_{j\in\nn}\int_\rn\hbar^{k_l}_{i,j}(x)\Phi(x)dx\\
&=\sum_{j\in\nn}\lim_{l\rightarrow\infty}\int_\rn\hbar^{k_l}_{i,j}(x)\Phi(x)dx\\
&=\sum_{j\in\nn}\lim_{l\rightarrow\infty}\lf\langle\hbar^{k_l}_{i,j},\Phi\r\rangle=\sum_{j\in\nn}\lf\langle\hbar_{i,j},\Phi\r\rangle,
\end{align*}
this further implies that the claim \eqref{eq3.41} holds true.

Next we claim that, for any $\Phi\in\cs(\rn)$ and $\mu\in(0,\infty)$,
there exists a $N_1\in\nn$ such that, for any $k\in\nn$,
\begin{align}\label{eq3.43}
\lf|\sum_{\{i\in\zz:|i|\geq N_1\}}\lf\langle F^{k}_i,\Phi\r\rangle\r|<\frac{\mu}{3}
\quad \mathrm{in} \ \cs'(\rn).
\end{align}
Indeed, it follows from \eqref{eq3.9} that, for any $k,N\in\nn$, $\sum_{\{i\in\zz:|i|\geq N\}}F^{k}_i$
converges in $\cs'(\rn)$. Moreover, by \eqref{eq3.40}, \eqref{eq3.29}, the Fubini theorem, \eqref{eq3.33},
and the above items (i) and (v), we know that, for any $k,N\in\nn$ and $\Phi\in\cs'(\rn)$,
\begin{align*}
\sum^{-N}_{i=-\infty}\lf|\lf\langle F^{k}_i,\Phi\r\rangle\r|
&\leq\sum^{-N}_{i=-\infty}\int_\rn\lf|F^{k}_i(x)\Phi(x)\r|dx\\
&\leq\sum^{-N}_{i=-\infty}\sum_{j\in\nn}\int_{x_{i,j}+B_{\ell_{i,j}+4\tau}}\lf|\hbar^{k}_{i,j}(x)\Phi(x)\r|dx\\
&\lesssim\sum^{-N}_{i=-\infty}2^i\|\Phi\|_{L^1(\rn)}\lesssim2^{-N}\|\Phi\|_{L^1(\rn)},
\end{align*}
which implies that, for any $k,N\in\nn$, $\sum^{-N}_{i=-\infty}F^{k}_i$ converges in $\cs'(\rn)$ and, for any given $\Phi$ and $\mu$
the same as in \eqref{eq3.43}, there exists a $N_{1,1}\in\nn$ such that, for any $k\in\nn$ and $N\in(N_{1,1},\infty)\cap\nn$,
\begin{align}\label{eq3.44}
\lf|\sum^{-N}_{i=-\infty}\lf\langle F^{k}_i,\Phi\r\rangle\r|<\frac{\mu}{6}
\quad \mathrm{in} \ \cs'(\rn).
\end{align}

Moreover, we also show that, for any $k,N\in\nn$, $\sum^{\infty}_{i=N}F^{k}_i$ converges in $\cs'(\rn)$ and, for any given $\Phi$ and $\mu$
the same as in \eqref{eq3.43}, there exists a $N_{1,2}\in\nn$ such that, for any $k\in\nn$ and $N\in(N_{1,2},\infty)\cap\nn$,
\begin{align}\label{eq3.45}
\lf|\sum^{\infty}_{i=N}\lf\langle F^{k}_i,\Phi\r\rangle\r|<\frac{\mu}{6}
\quad \mathrm{in} \ \cs'(\rn).
\end{align}
To prove this, let $p^-_\varphi\in(0,i(\varphi))$, $p^+_\varphi\in(I(\varphi),\infty)$,
$p_0\in(0,\min\{1,p^-_\varphi\})$, $\tilde{\varphi}(\cdot,t):=\varphi(\cdot,1)t^{p_0}$
for any $t\in[0,\infty)$, and $H^{p_0}_{A,\varphi(\cdot,1)}(\rn):=H^{\tilde{\varphi}}_{A}(\rn)$.
Then, by the fact that $\varphi$ is of uniformly
upper type $p^+_\varphi$ and of uniformly lower type $p^-_\varphi$ (see, for instance, \cite[Remark 2.2]{jwyyz23}) and the definition
of $\|\cdot\|_{H^{p_0}_{A,\varphi(\cdot,1)}(\rn)}$, we know that, for any $F\in\cs'(\rn)$,
\begin{align}\label{eq3.46}
\|F\|_{H^{p_0}_{A,\varphi(\cdot,1)}(\rn)}=\lf[\int_\rn|F^*(x)|^{p_0}\varphi(x,1)dx\r]^{\frac1{p_0}}.
\end{align}
In addition, by the above established result that, for any $k,N\in\nn$,
both $\sum_{\{i\in\zz:|i|\geq N\}}F^{k}_i$ and $\sum^{-N}_{i=-\infty}F^{k}_i$
converges in $\cs'(\rn)$, we know that $\sum^{\infty}_{i=N}F^{k}_i$
converges in $\cs'(\rn)$. By this and the fact that $H^{p_0}_{A,\varphi(\cdot,1)}(\rn)$ is continuously embedded
into $\cs'(\rn)$ (see, for instance, \cite[Proposition 6]{lyy14}),
to prove \eqref{eq3.45}, it suffices to show that
\begin{align}\label{eq3.47}
\lim_{N\rightarrow\infty}\sup_{k\in\nn}\lf\|\sum^{\infty}_{i=N}F^{k}_i\r\|_{H^{p_0}_{A,\varphi(\cdot,1)}(\rn)}\rightarrow0.
\end{align}
Indeed, by the fact that, for any $k,j\in\nn$ and $i\in\zz$, $[2^i\|\mathbf{1}_{x_{i,j}+B_{\ell_{i,j}+4\tau}}\|_{L^{p_0}_{\varphi(\cdot,1)}(\rn)}]^{-1}\hbar^{k}_{i,j}$
is an anisotropic $(\widetilde{\varphi},\infty,s)$-atom up to a harmless constant multiple,
\eqref{eq3.40}, \eqref{eq3.46}, $p_0\in(0,\min\{1,p^-_\varphi\})$,
the facts that $(\sum_{k\in\nn}|\lambda_k|)^\theta\leq\sum_{k\in\nn}|\lambda_k|^\theta$
for any $\{\lambda_k\}_{k\in\nn}\subset\mathbb{C}$ and $\theta\in(0,1]$ and
any anisotropic $(\varphi,r,s)$-atom supported in a ball $B\in\mathfrak{B}(\rn)$ is an anisotropic $(\varphi,r,s,\varepsilon)$-molecule related to the same ball
$B\in\mathfrak{B}(\rn)$, Lemma \ref{l3.5} with $\varphi$ therein replaced by $\tilde{\varphi}$ here, \eqref{eq3.31},
Lemma \ref{l2.2}(i), the above item (i), the fact that
$$\int_{\Omega_i}\varphi(x,1)dx\lesssim\max\lf\{\lf\|\mathbf{1}_{\Omega_i}\r\|^{p^+_\varphi}_{L^\varphi(\rn)},
\lf\|\mathbf{1}_{\Omega_i}\r\|^{p^-_\varphi}_{L^\varphi(\rn)}\r\}$$
(see, for instance, \cite[(3.90)]{jwyyz23}), the definition of $\|\cdot\|_{H^{\varphi,\infty}_{A,m}(\rn)}$,
$\mathbf{1}_{\Omega_i}\leq2^{-i}f^*_m$ for any $i\in\zz$, we know that, for any $k,N\in\nn$,
\begin{align*}
&\lf\|\sum^{\infty}_{i=N}F^{k}_i\r\|^{p_0}_{H^{p_0}_{A,\varphi(\cdot,1)}(\rn)}
=\lf\|\sum^{\infty}_{i=N}\sum_{j\in\nn}\hbar^{k}_{i,j}\r\|^{p_0}_{H^{p_0}_{A,\varphi(\cdot,1)}(\rn)}\\
\hs&=\int_\rn\lf|\lf(\sum^{\infty}_{i=N}\sum_{j\in\nn}\hbar^{k}_{i,j}\r)^*(x)\r|^{p_0}\varphi(x,1)dx
\leq\sum^{\infty}_{i=N}\sum_{j\in\nn}\lf\|\hbar^{k}_{i,j}\r\|^{p_0}_{H^{p_0}_{A,\varphi(\cdot,1)}(\rn)}\\
\hs&\lesssim\sum^{\infty}_{i=N}\sum_{j\in\nn}2^{ip_0}\lf\|\mathbf{1}_{x_{i,j}+B_{\ell_{i,j}+4\tau}}\r\|^{p_0}_{L^{p_0}_{\varphi(\cdot,1)}(\rn)}
=\sum^{\infty}_{i=N}\sum_{j\in\nn}2^{ip_0}\int_{x_{i,j}+B_{\ell_{i,j}+4\tau}}\varphi(x,1)dx\\
\hs&\lesssim\sum^{\infty}_{i=N}\sum_{j\in\nn}2^{ip_0}\int_{x_{i,j}+B_{\ell_{i,j}}}\varphi(x,1)dx
\lesssim\sum^{\infty}_{i=N}2^{ip_0}\int_{\Omega_i}\varphi(x,1)dx\\
\hs&\lesssim\sum^{\infty}_{i=N}2^{ip_0}\max\lf\{\lf\|\mathbf{1}_{\Omega_i}\r\|^{p^+_\varphi}_{L^\varphi(\rn)},
\lf\|\mathbf{1}_{\Omega_i}\r\|^{p^-_\varphi}_{L^\varphi(\rn)}\r\}\\
\hs&\leq\sum^{\infty}_{i=N}\lf[2^{i(p_0-p^+_\varphi)}\|f\|^{p^+_\varphi}_{H^{\varphi,\infty}_{A,m}(\rn)}
+2^{i(p_0-p^-_\varphi)}\|f\|^{p^-_\varphi}_{H^{\varphi,\infty}_{A,m}(\rn)}\r]\\
\hs&\lesssim2^{N(p_0-p^+_\varphi)}\|f\|^{p^+_\varphi}_{H^{\varphi,\infty}_{A,m}(\rn)}
+2^{N(p_0-p^-_\varphi)}\|f\|^{p^-_\varphi}_{H^{\varphi,\infty}_{A,m}(\rn)}
\rightarrow0
\end{align*}
as $N\rightarrow\infty$, which implies that \eqref{eq3.47} holds true, further implies that \eqref{eq3.45} holds true.
Thus, combining the above estimates \eqref{eq3.44} and \eqref{eq3.45}, we conclude that \eqref{eq3.43} holds true.

In addition, choose an integer $N_0>\max\{N_1,N_2\}$, by \eqref{eq3.43} and similar to the proof of \cite[p.53]{jwyyz23}, we conclude that
\begin{align*}
&\lf|\sum_{i\in\zz}\lf\langle F^{k_l}_i,\Phi\r\rangle-\sum_{i\in\zz}\lf\langle F_i,\Phi\r\rangle\r|\\
\hs&\leq\lf|\sum_{\{i\in\zz:|i|\geq N_0\}}\lf\langle F^{k_l}_i,\Phi\r\rangle\r|
+\lf|\sum_{\{i\in\zz:|i|\geq N_0\}}\lf\langle F_i,\Phi\r\rangle\r|
+\lf|\sum_{\{i\in\zz:|i|< N_0\}}\lf\langle F^{k_l}_i-F_i,\Phi\r\rangle\r|\\
\hs&<\frac{\mu}{3}+\frac{\mu}{3}+\frac{\mu}{3}=\mu,
\end{align*}
which implies that $\lim_{l\rightarrow\infty}\sum_{i\in\zz}F^{k_l}_i=\sum_{i\in\zz}F_i$ in $\cs'(\rn)$.
By this, the fact that $f^{k_l}\rightarrow f$ in $\cs'(\rn)$ as $l\rightarrow\infty$,
\eqref{eq3.9}, \eqref{eq3.29}, \eqref{eq3.40}, and \eqref{eq3.39}, we find that
$$f=\lim_{l\rightarrow\infty}f^{k_l}=\lim_{l\rightarrow\infty}\sum_{i\in\zz}F^{k_l}_i
=\sum_{i\in\zz}F_i=\sum_{i\in\zz}\sum_{j\in\nn}\hbar_{i,j}$$
in $\cs'(\rn)$. Thus, we finishes the proof of \eqref{eq3.38} and hence of \eqref{e1.12} and \eqref{e1.13}.

Furthermore, from Theorem \ref{t1.2}, \eqref{e1.12}, and \eqref{e1.13}, we deduce that,
if $f\in H^{\varphi,q}_{A}(\mathbb{R}^{n})$ with $q\in(0,\infty)$
[resp., $f\in\mathcal{H}^{\varphi,\infty}_{A}(\mathbb{R}^{n})$], then $f=\sum_{i\in\zz}\sum_{j\in\nn}\lambda_{i,j}a_{i,j}$
in $H^{\varphi,q}_{A}(\mathbb{R}^{n})$ with $q\in(0,\infty)$ [resp., $\mathcal{H}^{\varphi,\infty}_{A}(\mathbb{R}^{n})$].
This completes the proof of Theorem \ref{t1.3}.
\end{proof}

Now, we show Theorem \ref{t1.4} via using Theorems \ref{t1.2} and \ref{t1.3}.
\begin{proof}[Proof of Theorem \ref{t1.4}]
Let $m\geq s\geq\lfloor\frac{\ln b}{\ln\lambda_-}\frac{q(\varphi)}{i(\varphi)}\rfloor$.
By the definition of $H^{\varphi,q}_{A}(\mathbb{R}^{n})$ and $H^{\varphi,q}_{A,m}(\mathbb{R}^{n})$,
we know that $H^{\varphi,q}_{A}(\mathbb{R}^{n})\subset H^{\varphi,q}_{A,m}(\mathbb{R}^{n})$ and,
for any $f\in H^{\varphi,q}_{A}(\mathbb{R}^{n})$,
$$\|f\|_{H^{\varphi,q}_{A,m}(\mathbb{R}^{n})}\lesssim\|f\|_{H^{\varphi,q}_{A}(\mathbb{R}^{n})}.$$
Thus, to completes the proof of Theorem \ref{t1.4}, it suffices to prove that,
for any $f\in H^{\varphi,q}_{A,m}(\mathbb{R}^{n})$,
\begin{align}\label{e3.6}
\|f\|_{H^{\varphi,q}_{A}(\mathbb{R}^{n})}\lesssim\|f\|_{H^{\varphi,q}_{A,m}(\mathbb{R}^{n})},
\end{align}
which implies that $H^{\varphi,q}_{A,m}(\mathbb{R}^{n})\subset H^{\varphi,q}_{A}(\mathbb{R}^{n})$.
To this end, let $f\in H^{\varphi,q}_{A,m}(\mathbb{R}^{n})$. Then, from Theorem \ref{t1.3},
we deduce that there exists a positive constant $C$ and a sequence $\{a_{k,j}\}_{k\in\zz,j\in\nn}$ of anisotropic $(\varphi,\infty,m)$-atoms supported, respectively, in the dilated balls $\{B_{k,j}\}_{k\in\zz,j\in\nn}\subset\mathfrak{B}(\rn)$ such that
\begin{equation}\label{e3.7}
f=\sum_{k\in\zz}\sum_{j\in\nn}C2^k\lf\|\mathbf{1}_{B_{k,j}}\r\|_{L^\varphi(\rn)}a_{k,j}
\end{equation}
in $\cs'(\rn)$ and
\begin{equation}\label{e3.8}
\lf[\sum_{k\in\zz}2^{kq}\lf\|\sum_{j\in\nn}\mathbf{1}_{B_{k,j}}\r\|^q_{L^\varphi(\rn)}\r]^{\frac1{q}}
\lesssim\|f\|_{H^{\varphi,q}_{A,m}(\mathbb{R}^{n})}
\end{equation}
with the usual interpretation for $q=\infty$. Moreover, notice that any anisotropic $(\varphi,\infty,m)$-atom
is also an anisotropic $(\varphi,\infty,s)$-atom, which together with \eqref{e3.7}, Theorem \ref{t1.2}, and \eqref{e3.8},
implies that
$$\|f\|_{H^{\varphi,q}_{A}(\mathbb{R}^{n})}
\lesssim\lf[\sum_{k\in\zz}2^{kq}\lf\|\sum_{j\in\nn}\mathbf{1}_{B_{k,j}}\r\|^q_{L^\varphi(\rn)}\r]^{\frac1{q}}
\lesssim\|f\|_{H^{\varphi,q}_{A,m}(\mathbb{R}^{n})},$$
further implies that \eqref{e3.6} holds true.

Finally, we show $\mathcal{H}^{\varphi,\infty}_{A}(\mathbb{R}^{n})=\mathcal{H}^{\varphi,\infty}_{A,m}(\mathbb{R}^{n})$.
By the definition of $\mathcal{H}^{\varphi,\infty}_{A}(\mathbb{R}^{n})$ and $\mathcal{H}^{\varphi,\infty}_{A,m}(\mathbb{R}^{n})$,
we conclude that $\mathcal{H}^{\varphi,\infty}_{A}(\mathbb{R}^{n})\subset\mathcal{H}^{\varphi,\infty}_{A,m}(\mathbb{R}^{n})$.
Next, we show $\mathcal{H}^{\varphi,\infty}_{A,m}(\mathbb{R}^{n})\subset\mathcal{H}^{\varphi,\infty}_{A}(\mathbb{R}^{n})$.
Indeed, let $f\in\mathcal{H}^{\varphi,\infty}_{A,m}(\mathbb{R}^{n})$. By Theorem \ref{t1.3}, we find that
\begin{equation}\label{e3.9}
\lim_{|k|\rightarrow\infty}2^k\lf\|\sum_{j\in\nn}\mathbf{1}_{B_{k,j}}\r\|_{L^\varphi(\rn)}=0.
\end{equation}
Moreover, notice that every anisotropic $(\varphi,\infty,m)$-atom
is also an anisotropic $(\varphi,\infty,s)$-atom, which, together with Theorem \ref{t1.2}, \eqref{e3.8}, and \eqref{e3.9},
implies that $f\in\mathcal{H}^{\varphi,\infty}_{A}(\mathbb{R}^{n})$, further implies that
$\mathcal{H}^{\varphi,\infty}_{A,m}(\mathbb{R}^{n})\subset\mathcal{H}^{\varphi,\infty}_{A}(\mathbb{R}^{n})$.
This completes the proof of Theorem \ref{t1.4}.
\end{proof}




\section{Proof of Theorem \ref{t1.5}}\label{s4}
\hskip\parindent
In this section, we will give the proof of Theorem \ref{t1.5}. To begin with, we recall the definition
of the anisotropic Calder\'{o}n-Zygmund operators associated with dilation $A$ (see, for instance, \cite[Definition 9.1]{b03} and \cite[Definition 4.1]{lffy16}), and establish some auxiliary lemmas.
\begin{definition}\label{d4.1}
A locally integrable function $K$ on $\Omega:=\{(x,\,y)\in\rn\times\rn:\,x\ne y\}$ is called
an {\it anisotropic Calder\'on-Zygmund kernel}
(with respect to a dilation A and a quasi-norm $\rho$) if there exists positive constants $C$ and $\delta$ such that
\begin{enumerate}
\item[\rm{(i)}]  $|K(x,y)|\le \frac C{\rho(x-y)}$ for any $x, y\in\rn$ with $x\neq y$;
\item[\rm{(ii)}] if $(x,y)\in \Omega,\, z\in\rn$ and $\rho(z-x)\le b^{-2\tau}\rho(x-y)$, then
$$\lf|K(z,y)-K(x,y)\r|\le C\frac{[\rho(z-x)]^\delta}{[\rho(x-y)]^{1+\delta}};$$
\item[\rm{(iii)}] if $(x,y)\in \Omega,\,z\in\rn$ and $\rho(z-y)\le b^{-2\tau}\rho(x-y)$, then
$$\lf|K(x,z)-K(x,y)\r|\le C\frac{[\rho(z-y)]^\delta}{[\rho(x-y)]^{1+\delta}}.$$
\end{enumerate}
We call that
$T$ is an {\it anisotropic Calder\'on-Zygmund operator} with regularity $\delta$ if $T$ is a continuous linear operator
mapping $\mathcal{S}(\rn)$ into $\mathcal{S}'(\rn)$ that extends to a bounded linear
operator on $L^2(\rn)$ and there exists an anisotropic Calder\'on-Zygmund kernel $K$ such that,
for all $f\in C_{\rm c}^\infty(\rn)$ and $x\not\in \overline{\supp (f)}$, $$T(f)(x):=\int_{\supp (f)} K(x,y)f(y)dy,$$
where $\overline{\supp (f)}$ denotes the closure of $\supp (f)$ in $\rn$.
\end{definition}

To prove Theorem \ref{t1.5}, we first show that Theorem \ref{t1.3} also gives a decomposition theorem for any element in $H^{\varphi}_{A}(\mathbb{R}^{n})$ via atoms. In what follows, for any $s\in\zz$, the space $L^\infty_{\mathrm{c},s}(\rn)$ is
defined to be the set of all the $f\in L^\infty(\rn)$ having compact support and satisfying that, for any $\gamma\in\zz^n_+$
with $|\gamma|\leq s$, $\int_\rn f(x)x^\gamma dx=0$.
\begin{theorem}\label{t4.1}
Let all the symbols be the same as in Theorem \ref{t1.3}. Then
\begin{enumerate}
\item[\rm{(i)}] for any $f\in H^{\varphi}_{A}(\mathbb{R}^{n})$, there exists a sequence $\{a_{k,j}\}_{k\in\zz,j\in\nn}$ of anisotropic $(\varphi,\infty,s)$-atoms supported, respectively,
in the dilated balls $\{B_{k,j}\}_{k\in\zz,j\in\nn}\subset\mathfrak{B}(\rn)$, where $B_{k,j}:=x_{k,j}+B_{\ell_{k,j}}$
with $x_{k,j}\in\rn$ and $\ell_{k,j}\in\zz$, such that $\sum_{j\in\nn}\mathbf{1}_{B_{k,j}}\leq C_{(n)}$,
where $C_{(n)}$ is a positive constant depending only on $n$,
\begin{equation}\label{e4.1}
f=\sum_{k\in\zz}\sum_{j\in\nn}\lambda_{k,j}a_{k,j}
\end{equation}
converges in $H^{\varphi}_{A}(\mathbb{R}^{n})$ and, for any given $d\in(0,\infty)$,
\begin{equation}\label{e4.2}
\lf\|\lf\{\sum_{k\in\zz}\sum_{j\in\nn}\lf[\frac{\lambda_{k,j}\mathbf{1}_{B_{k,j}}}
{\|\mathbf{1}_{B_{k,j}}\|_{L^\varphi(\rn)}}
\r]^d\r\}^{\frac1{d}}\r\|_{L^\varphi(\rn)}\leq C\|f\|_{H^{\varphi}_{A}(\mathbb{R}^{n})},
\end{equation}
where $C$ is a positive constant independing of $f$ and, for any $k\in\zz$ and $j\in\nn$,
$\lambda_{k,j}:=C_12^k\|\mathbf{1}_{B_{k,j}}\|_{L^\varphi(\rn)}$;
\item[\rm{(ii)}] $L^\infty_{\mathrm{c},s}(\rn)$ is a dense subspace of $H^{\varphi}_{A}(\mathbb{R}^{n})$.
\end{enumerate}
\end{theorem}

To prove Theorem \ref{t1.5}, we also need show that, for any $r\in(1,\infty)$ and any
$f\in H^{\varphi,q}_{A}(\mathbb{R}^{n})\cap L^r(\rn)$, the decomposition in Theorem \ref{t1.3}
also converges almost everywhere on $\rn$ and in $L^r(\rn)$,
whose proof is similar to that of \cite[Theorem 3.40]{jwyyz23}. We omit the details here.
\begin{theorem}\label{t4.2}
Let all the symbols be the same as in Theorem \ref{t1.3} and $r\in(1,\infty)$. Then,
for any $f\in H^{\varphi,q}_{A}(\mathbb{R}^{n})\cap L^r(\rn)$, the same decomposition of $f$ as in Theorem \ref{t1.3}
also converges almost everywhere on $\rn$ and in $L^r(\rn)$.
\end{theorem}
Before proving Theorem \ref{t4.1}, we first need the following lemma.

\begin{lemma}\label{l4.1}
Let $\vz\in\mathbb{A}_\infty(A)$ be a Musielak-Orlicz function with $0<i(\varphi)\leq I(\varphi)<\infty$,
$d\in(0,\min\{1,i(\varphi)\})$, and $s\in[m(\varphi),\infty)\cap\zz_+$,
where $i(\varphi)$, $I(\varphi)$, and $m(\varphi)$ are,
respectively, as in \eqref{e1.1}, \eqref{e1.2}, and \eqref{e1.7}.
Let $\{\lambda_j\}_{j\in\nn}\subset[0,\infty)$ and $\{a_j\}_{j\in\nn}$ be a sequence of
anisotropic $(\varphi,\infty,s)$-atoms supported, respectively,
in the dilated balls $\{B_{j}\}_{j\in\nn}\subset\mathfrak{B}(\rn)$, where $B_{j}:=x_{j}+B_{\ell_{j}}$
with $x_{j}\in\rn$ and $\ell_{j}\in\zz$, satisfying
$$\lf\|\lf\{\sum_{j\in\nn}\lf[\frac{\lambda_{j}\mathbf{1}_{B_{j}}}
{\|\mathbf{1}_{B_{j}}\|_{L^\varphi(\rn)}}\r]^d\r\}^{\frac1{d}}\r\|_{L^\varphi(\rn)}<\infty.$$
Then $f:=\sum_{j\in\nn}\lambda_ja_j$ converges in $\cs'(\rn)$ and
$$\|f\|_{H^{\varphi}_{A}(\mathbb{R}^{n})}\lesssim\lf\|\lf\{\sum_{j\in\nn}\lf[\frac{\lambda_{j}\mathbf{1}_{B_{j}}}
{\|\mathbf{1}_{B_{j}}\|_{L^\varphi(\rn)}}\r]^d\r\}^{\frac1{d}}\r\|_{L^\varphi(\rn)},$$
where the implicit positive constant is independent of $\{B_{j}\}_{j\in\nn}$ and $\{\lambda_j\}_{j\in\nn}$.
\end{lemma}

To prove Lemma \ref{l4.1}, we begin with recalling the definition of the associate spaces.
The associate space $(L^\varphi(\rn))'$ of a orlicz space $L^\varphi(\rn)$ is defined by setting
$$\lf(L^\varphi(\rn)\r)':=\lf\{f \ \mathrm{is\ a\ measurable\ function}:\ \|f\|_{(L^\varphi(\rn))'}<\infty\r\},$$
where, for any measurable function $f$,
$$\|f\|_{(L^\varphi(\rn))'}:=\sup\lf\{\|fg\|_{L^1(\rn)}: g\in L^\varphi(\rn),\ \|g\|_{L^\varphi(\rn)}=1\r\}$$
(see, for instance, \cite[p.8, Definition 2.1]{bs88})).
The following lemma is a special case of \cite[ Lemmas 2.18-2.20]{yhyy23}).

\begin{lemma}\label{l4.2}
Let $(L^\varphi(\rn))'$ be the associate space of $L^\varphi(\rn)$. Then the following conclusions hold true.
\begin{enumerate}
\item[\rm{(i)}]  $(L^\varphi(\rn))'$ is a ball Banach space;
\item[\rm{(ii)}] for any $f\in L^\varphi(\rn)$ and $g\in (L^\varphi(\rn))'$,
$$\|fg\|_{L^1(\rn)}\leq\|f\|_{L^\varphi(\rn)}\|g\|_{(L^\varphi(\rn))'};$$
\item[\rm{(iii)}] a function $f\in L^\varphi(\rn)$ if and only if $f\in(L^\varphi(\rn))''$.
Moreover, for any $f\in L^\varphi(\rn)$ or $f\in(L^\varphi(\rn))''$,
$\|f\|_{L^\varphi(\rn)}=\|f\|_{(L^\varphi(\rn))''}$.
\end{enumerate}
\end{lemma}

To show Lemma \ref{l4.1}, we need the following technical proposition.
\begin{proposition}\label{p4.1}
Let $\vz\in\mathbb{A}_\infty(A)$ be a Musielak-Orlicz function with $0<i(\varphi)\leq I(\varphi)<\infty$,
$d\in(0,\min\{1,i(\varphi)\})$, and $q\in(1,\infty]\cap(I(\varphi),\infty]$,
where $i(\varphi)$ and $I(\varphi)$ are, respectively, as in \eqref{e1.1} and \eqref{e1.2}.
Assume that $\{\lambda_j\}_{j\in\nn}\subset\mathbb{C}$, $\{B_{j}\}_{j\in\nn}\subset\mathfrak{B}(\rn)$, where $B_{j}:=x_{j}+B_{\ell_{j}}$ with $x_{j}\in\rn$ and $\ell_{j}\in\zz$,
and $\{a_j\}_{j\in\nn}\subset L^q(\rn)$ satisfy, for any $j\in\nn$,
$\supp (a_j)\subset B_j$, $\|a_j\|_{L^q(\rn)}\le|B_j|^{\frac1{q}}\|\mathbf{1}_{B_j}\|^{-1}_{L^\varphi(\rn)}$,
and
$$\lf\|\lf\{\sum_{j\in\nn}\lf[\frac{\lambda_{j}\mathbf{1}_{B_{j}}}
{\|\mathbf{1}_{B_{j}}\|_{L^\varphi(\rn)}}\r]^d\r\}^{\frac1{d}}\r\|_{L^\varphi(\rn)}<\infty.$$
Then
$$\lf\|\lf[\sum_{j\in\nn}\lf|\lambda_{j}a_j\r|^d\r]^{\frac1{d}}\r\|_{L^\varphi(\rn)}
\leq C\lf\|\lf\{\sum_{j\in\nn}\lf[\frac{\lambda_{j}\mathbf{1}_{B_{j}}}
{\|\mathbf{1}_{B_{j}}\|_{L^\varphi(\rn)}}\r]^d\r\}^{\frac1{d}}\r\|_{L^\varphi(\rn)},$$
where $C$ is a positive constant independent of $B_{j}$, $a_{j}$, and $\lambda_j$.
\end{proposition}
\begin{proof}
By the definition of $(L^\varphi(\rn))'$, we find $g\in[L^{\varphi_{1/d}}(\rn)]'$ with
$\|g\|_{[L^{\varphi_{1/d}}(\rn)]'}=1$ such that
\begin{align*}
\lf\|\lf[\sum_{j\in\nn}\lf|\lambda_{j}a_j\r|^d\r]^{\frac1{d}}\r\|^d_{L^\varphi(\rn)}
=\lf\|\sum_{j\in\nn}\lf|\lambda_{j}a_j\r|^d\r\|_{L^{\varphi_{1/d}}(\rn)}
\lesssim\int_\rn\sum_{j\in\nn}\lf|\lambda_{j}a_j\r|^d|g(x)|dx.
\end{align*}
By this and the H\"{o}lder inequality for $\frac1{q/d}+\frac1{(q/d)'}=1$, we know that
\begin{align*}
\int_\rn\sum_{j\in\nn}\lf|\lambda_{j}a_j(x)\r|^d|g(x)|dx
&\leq\sum_{j\in\nn}\lf|\lambda_{j}\r|^d\lf\|a_j\r\|^d_{L^q(\rn)}\lf\|g\r\|_{L^{(q/d)'}(B_j)}\\
&\lesssim\sum_{j\in\nn}\frac{|\lambda_{j}|^d|B_{j}|^{\frac{d}{q}}}{\|\mathbf{1}_{B_j}\|^{d}_{L^\varphi(\rn)}}
\lf\|g\r\|_{L^{(\frac{q}{d})'}(B_j)}\\
&\lesssim\sum_{j\in\nn}\frac{|\lambda_{j}|^d|B_{j}|}{\|\mathbf{1}_{B_j}\|^{d}_{L^\varphi(\rn)}}
\inf_{z\in B_j}\lf[M_{\mathrm{HL}}\lf(|g|^{(\frac{q}{d})'}\r)\r]^{1/{(\frac{q}{d})'}}\\
&\lesssim\int_\rn\sum_{j\in\nn}\frac{|\lambda_{j}|^d\mathbf{1}_{B_j}(x)}{\|\mathbf{1}_{B_j}\|^{d}_{L^\varphi(\rn)}}
\lf[M_{\mathrm{HL}}\lf(|g|^{(\frac{q}{d})'}\r)\r]^{1/{(\frac{q}{d})'}}dx,
\end{align*}
which, combined with the H\"{o}lder inequality in Lemma \ref{l4.2}(ii), Lemma \ref{l2.9}(i),
and checking the proof of \cite[Theorem 7.14(ii)]{shyy17}, further implies that
\begin{align*}
\int_\rn\sum_{j\in\nn}\lf|\lambda_{j}a_j(x)\r|^d|g(x)|dx
&\lesssim\lf\|\sum_{j\in\nn}\frac{|\lambda_{j}|^d\mathbf{1}_{B_j}}
{\|\mathbf{1}_{B_j}\|^{d}_{L^\varphi(\rn)}}\r\|_{L^{\varphi_{1/d}}(\rn)}
\lf\|\lf[M_{\mathrm{HL}}\lf(|g|^{(\frac{q}{d})'}\r)\r]^{1/{(\frac{q}{d})'}}\r\|_{[L^{\varphi_{1/d}}(\rn)]'}\\
&\lesssim\lf\|\lf\{\sum_{j\in\nn}\lf[\frac{\lambda_{j}\mathbf{1}_{B_{j}}}
{\|\mathbf{1}_{B_{j}}\|_{L^\varphi(\rn)}}\r]^d\r\}^{\frac1{d}}\r\|^d_{L^\varphi(\rn)}
\|g\|_{[L^{\varphi_{1/d}}(\rn)]'}.
\end{align*}
This completes the proof of Proposition \ref{p4.1}.
\end{proof}

Now, we show Lemma \ref{l4.1} via using Proposition \ref{p4.1}.
\begin{proof}[Proof of Lemma \ref{l4.1}]
By \cite[(4.8)]{lyy17}, it is easy to know that, for any
$m\in[\lfloor\frac{\ln b}{\ln \lambda_-}[\frac{q(\varphi)}{i(\varphi)}-1]\rfloor,)\cap\nn^+$
and $x\in\rn$,
\begin{align}\label{e4.3}
f^*_m(x)&\lesssim\sum_{j\in\nn}\lf|\lambda_j\r|\lf(a_j\r)^*_m(x)\mathbf{1}_{A^\tau B_{j}}(x)
+\sum_{j\in\nn}\lf|\lambda_j\r|\lf(a_j\r)^*_m(x)\mathbf{1}_{(A^\tau B_{j})^\complement}(x)\nonumber\\
&\lesssim\lf\{\sum_{j\in\nn}\lf[\lf|\lambda_j\r|\lf(a_j\r)^*_m(x)\mathbf{1}_{A^\tau B_{j}}(x)\r]^d\r\}^{\frac1{d}}
+\sum_{j\in\nn}\frac{|\lambda_{j}|}{\|\mathbf{1}_{B_j}\|_{L^\varphi(\rn)}}
\lf[M_{\mathrm{HL}}\lf(\mathbf{1}_{B_j}\r)(x)\r]^\beta\nonumber\\
&=:{\rm{II_{1}+II_{2}}},
\end{align}
where $\tau$ is as in \eqref{e1.4} and \eqref{e1.5}, $\beta:=(\frac{\ln b}{\ln \lambda_-}+s+1)\frac{\ln b}{\ln \lambda_-}$,
and $M_{\mathrm{HL}}$ denotes the Hardy-Littlewood maximal operator.

We first deal with the term ${\rm{II_{1}}}$. From the boundedness of $f^*_m$ on $L^p(\rn)$ with $p\in(1,\infty]$
(see, for instance, \cite[Remark 2.10]{lyy16}), Proposition \ref{p4.1}, and Lemma \ref{l2.12}, it follows that
\begin{align}\label{e4.4}
\lf\|{\rm{II_{1}}}\r\|_{L^\varphi(\rn)}\lesssim
\lf\|\lf\{\sum_{j\in\nn}\lf[\frac{\lambda_{j}\mathbf{1}_{B_{j}}}
{\|\mathbf{1}_{B_{j}}\|_{L^\varphi(\rn)}}\r]^d\r\}^{\frac1{d}}\r\|_{L^\varphi(\rn)}.
\end{align}

Now we deal with the term ${\rm{II_{2}}}$. By Lemma \ref{l2.9}(i) and Lemma \ref{l2.11}, we know that
\begin{align*}
\lf\|{\rm{II_{2}}}\r\|_{L^\varphi(\rn)}
&\lesssim\lf\|\lf\{\sum_{j\in\nn}\frac{|\lambda_{j}|}{\|\mathbf{1}_{B_j}\|_{L^\varphi(\rn)}}
\lf[M_{\mathrm{HL}}\lf(\mathbf{1}_{B_j}\r)(x)\r]^\beta\r\}^{\frac1{\beta}}\r\|^\beta_{L^{\varphi_\beta}(\rn)}\\
&\lesssim\lf\|\sum_{j\in\nn}\frac{\lambda_{j}\mathbf{1}_{B_{j}}}
{\|\mathbf{1}_{B_{j}}\|_{L^\varphi(\rn)}}\r\|_{L^\varphi(\rn)}
\lesssim\lf\|\lf\{\sum_{j\in\nn}\lf[\frac{\lambda_{j}\mathbf{1}_{B_{j}}}
{\|\mathbf{1}_{B_{j}}\|_{L^\varphi(\rn)}}\r]^d\r\}^{\frac1{d}}\r\|_{L^\varphi(\rn)}
\end{align*}
which together with \eqref{e4.3} and \eqref{e4.4}, further implies that
$$\|f\|_{H^{\varphi}_{A}(\mathbb{R}^{n})}\sim\lf\|f^*_m\r\|_{L^{\varphi}(\mathbb{R}^{n})}
\lesssim\lf\|\lf\{\sum_{j\in\nn}\lf[\frac{\lambda_{j}\mathbf{1}_{B_{j}}}
{\|\mathbf{1}_{B_{j}}\|_{L^\varphi(\rn)}}\r]^d\r\}^{\frac1{d}}\r\|_{L^\varphi(\rn)}.$$
This completes the proof of Lemma \ref{l4.1}.
\end{proof}

Next, we show Theorem \ref{t4.1} via using Theorems \ref{t1.3} and \ref{t1.4}, and Lemma \ref{l4.1}.
\begin{proof}[Proof of Theorem \ref{t4.1}]
We first prove (i). Let $f\in H^{\varphi}_{A}(\mathbb{R}^{n})$ and, for any $k\in\zz$,
$\Omega_k:=\{x\in\Omega:f^*_m(x)>2^k\}$, where $m$ is as in \eqref{eq3.6} of Theorem \ref{t1.3}.
Then, by the fact that $H^{\varphi}_{A}(\mathbb{R}^{n})\subset H^{\varphi,\infty}_{A}(\mathbb{R}^{n})$
(see, for instance, \cite[Lemma 3.17]{jwyyz23}) and Theorems \ref{t1.3} and \ref{t1.4},
we find that there exists a sequence $\{a_{k,j}\}_{k\in\zz,j\in\nn}$ of anisotropic $(\varphi,\infty,s)$-atoms supported, respectively, on the dilated balls $\{B_{k,j}\}_{k\in\zz,j\in\nn}\subset\mathfrak{B}(\rn)$, where $B_{k,j}:=x_{k,j}+B_{\ell_{k,j}}$ with $x_{k,j}\in\rn$ and $\ell_{k,j}\in\zz$, such that
$\sum_{j\in\nn}\mathbf{1}_{B_{k,j}}\leq C_{(n)}$,
where $C_{(n)}$ is a positive constant depending only on $n$,
$$f=\sum_{k\in\zz}\sum_{j\in\nn}C_12^k\lf\|\mathbf{1}_{B_{k,j}}\r\|_{L^\varphi(\rn)}a_{k,j}$$
in $\cs'(\rn)$, and $\Omega_k=\bigcup_{j\in\nn}B_{k,j}$, where $C_1$ is as in Theorem \ref{t1.3}.
Meanwhile, from the fact that $\boz_{k+1}\subset \boz_{k}$ for any $k\in\zz$ and $|\cap_{k=1}^\infty\boz_{k}|=0$, we deduce that, for almost everywhere $x\in\rn$,
$$\sum_{k=-\infty}^\infty2^k\mathbf{1}_{\Omega_k}(x)=
\sum_{k=-\infty}^\infty2^k\sum_{j=k}^\infty\mathbf{1}_{\Omega_j\backslash\Omega_{j+1}}(x)=
2\sum_{j=-\infty}^\infty2^j\mathbf{1}_{\Omega_j\backslash \Omega_{j+1}}(x),$$
which, combined with $\lambda_{k,j}:=C_12^k\|\mathbf{1}_{B_{k,j}}\|_{L^{\varphi}(\rn)}$, $\sum_{j\in\nn}\mathbf{1}_{B_{k,j}}\lesssim1$,
the definition of $\Omega_k$, $\Omega_k=\bigcup_{j\in\nn}B_{k,j}$, and $f^*_m\leq f^*$, yields that
\begin{align}\label{e4.5}
&\lf\|\lf\{\sum_{k\in\zz,j\in\nn}\lf[\frac{|\lambda_{k,j}|\mathbf{1}_{B_{k,j}}}{\|\mathbf{1}_{B_{k,j}}
\|_{L^{p(\cdot)}(\Omega)}}\r]^{d}\r\}^{1/{d}}\r\|_{L^\varphi(\rn)} \nonumber\\
&\quad\lesssim\lf\|\lf\{\sum_{k\in\zz,j\in\nn}\lf(2^k
\mathbf{1}_{B_{k,j}}\r)^{d}\r\}^{1/d}\r\|_{L^\varphi(\rn)}\lesssim\lf\|\lf\{\sum_{k\in\zz}\lf(2^k
\mathbf{1}_{\Omega_k}\r)^{d}\r\}^{1/{d}}\r\|_{L^\varphi(\rn)}\nonumber\\
&\quad\lesssim\lf\|\lf\{\sum_{k\in\zz}\lf(2^k
\mathbf{1}_{\Omega_k\backslash \Omega_{k+1}}\r)^{d}\r\}^{1/d}\r\|_{L^\varphi(\rn)}
\lesssim\lf\|\sum_{k\in\zz}f^*_m
\mathbf{1}_{\Omega_k\backslash \Omega_{k+1}}\r\|_{L^\varphi(\rn)}\nonumber\\
&\quad\leq \lf\|f^*_m\r\|_{L^\varphi(\rn)}\leq \lf\|f^*\r\|_{L^\varphi(\rn)}
=\|f\|_{H^{\varphi}_{A}(\mathbb{R}^{n})}<\infty,
\end{align}
further implies that \eqref{e4.1} holds true.

Next, we prove \eqref{e4.2}. Indeed, from \eqref{e4.5}, Lemma \ref{l4.1}, and
\cite[Lemma 3.16]{jwyyz23}, it follows that, for any given
$d\in(0,\min\{1,i(\varphi)\})$ and for any $k\in\zz$ and $N, j_0\in\nn$,
$$\lf\|\sum_{j=N}^{N+j_0}\lambda_{k,j}a_{k,j}\r\|_{H^{\varphi}_{A}(\mathbb{R}^{n})}
\lesssim\lf\|\lf\{\sum_{j=N}^{N+j_0}\lf[\frac{\lambda_{k,j}\mathbf{1}_{B_{k,j}}}
{\|\mathbf{1}_{B_{k,j}}\|_{L^\varphi(\rn)}}\r]^d\r\}^{\frac1{d}}\r\|_{L^\varphi(\rn)}\rightarrow0$$
as $N\rightarrow\infty$, which implies that, for any $k\in\zz$, $\{\sum_{j=1}^N\lambda_{k,j}a_{k,j}\}_{N\in\nn}$
is a Cauchy sequence in $H^{\varphi}_{A}(\mathbb{R}^{n})$.
By this and the completeness of $H^{\varphi}_{A}(\mathbb{R}^{n})$ (the proof of which is similar to that of
\cite[Proposition 6]{lyy14} and \cite[Proposition 2.8]{lyy16}), we know that,
for any $k\in\zz$, $\sum_{j\in\nn}\lambda_{k,j}a_{k,j}$ converges in $H^{\varphi}_{A}(\mathbb{R}^{n})$.
Similarly, we deduce that $\sum_{k\in\zz}\sum_{j\in\nn}\lambda_{k,j}a_{k,j}$ converges in $H^{\varphi}_{A}(\mathbb{R}^{n})$,
which, together with $f\in H^{\varphi}_{A}(\mathbb{R}^{n})$, \eqref{e4.5}, Lemma \ref{l4.1}, and
\cite[Lemma 3.16]{jwyyz23}, implies that, for any given
$d\in(0,\min\{1,i(\varphi)\})$ and for any $N_1, N_2, N\in\nn$,
\begin{align}\label{e4.6}
&\lf\|f-\sum_{k=-N_1}^{N_2}\sum_{j=1}^{N}\lambda_{k,j}a_{k,j}\r\|_{H^{\varphi}_{A}(\mathbb{R}^{n})} \nonumber\\
&\quad\lesssim\lf\|\sum_{k=-N_1}^{N_2}\sum_{j=N+1}^{\infty}\lambda_{k,j}a_{k,j}\r\|_{H^{\varphi}_{A}(\mathbb{R}^{n})}
+\lf\|\sum_{k\in R_N}\sum_{j\in\nn}\lambda_{k,j}a_{k,j}\r\|_{H^{\varphi}_{A}(\mathbb{R}^{n})} \nonumber\\
&\quad\lesssim\lf\|\lf\{\sum_{k=-N_1}^{N_2}\sum_{j=N+1}^{\infty}\lf[\frac{\lambda_{k,j}\mathbf{1}_{B_{k,j}}}
{\|\mathbf{1}_{B_{k,j}}\|_{L^\varphi(\rn)}}\r]^d\r\}^{\frac1{d}}\r\|_{L^{\varphi}(\mathbb{R}^{n})}
+\lf\|\lf\{\sum_{k\in R_N}\sum_{j\in\nn}\lf[\frac{\lambda_{k,j}\mathbf{1}_{B_{k,j}}}
{\|\mathbf{1}_{B_{k,j}}\|_{L^\varphi(\rn)}}\r]^d\r\}^{\frac1{d}}\r\|_{L^{\varphi}(\mathbb{R}^{n})} \nonumber\\
&\quad\rightarrow0
\end{align}
as $N_1, N_2, N\rightarrow\infty$, where $R_N:=[(-\infty,-N_1)\cup(N_2,\infty)]\cap\zz$,
further implies that \eqref{e4.2} holds true and hence completes the proof of (i).

Finally, we prove (ii). By the fact that any element in $L^\infty_{\mathrm{c},s}(\rn)$ is an
anisotropic $(\varphi,\infty,s)$-atom up to a constant multiple and Lemma \ref{l4.1},
we know that $L^\infty_{\mathrm{c},s}(\rn)\subset H^{\varphi}_{A}(\mathbb{R}^{n})$,
which combined with any anisotropic $(\varphi,\infty,s)$-atom belongs to $L^\infty_{\mathrm{c},s}(\rn)$
and \eqref{e4.6}, further implies that $L^\infty_{\mathrm{c},s}(\rn)$ is a dense subspace of $H^{\varphi}_{A}(\mathbb{R}^{n})$.
This completes the proof of (ii) and hence of Theorem \ref{t4.1}.
\end{proof}

The following lemma is just from \cite[Theorem A.3]{blyz10} (see also \cite[Lemma 4.2]{lffy16}).
\begin{lemma}\label{l4.3}
Let $p\in(1,\infty)$ and $\vz\in\mathbb{A}_p(A)$. Then anisotropic Calder\'on-Zygmund operator $T$
as in Definition \ref{d4.1} is bounded on $L^p_{\varphi(\cdot,\lambda)}(\rn)$ uniformly for any
$\lambda\in(0,\infty)$ and, moreover, there exists a positive constant $C$ such that, for any
$\lambda\in(0,\infty)$ and $f\in L^p_{\varphi(\cdot,\lambda)}(\rn)$,
$$\lf\|T(f)\r\|_{L^p_{\varphi(\cdot,\lambda)}(\rn)}\leq C\|f\|_{L^p_{\varphi(\cdot,\lambda)}(\rn)}.$$
\end{lemma}

The following lemma shows that $L^\infty_{\mathrm{c},s}(\rn)$ is a dense subspace of $H^{\varphi,q}_{A}(\mathbb{R}^{n})$
and $\mathcal{H}^{\varphi,\infty}_{A}(\mathbb{R}^{n})$, the proof of which is a slight modification of
\cite[Lemma 5.4]{jwyyz23}; we omit the details here.
\begin{lemma}\label{l4.4}
Let $\vz\in\mathbb{A}_\infty(A)$ be a Musielak-Orlicz function with $0<i(\varphi)\leq I(\varphi)<\infty$,
$q\in(0,\infty)$, and $s\in[m(\varphi),\infty)\cap\zz_+$, where $i(\varphi)$, $I(\varphi)$, and $m(\varphi)$ are,
respectively, as in \eqref{e1.1}, \eqref{e1.2}, and \eqref{e1.7}.
Then $L^\infty_{\mathrm{c},s}(\rn)$ is a dense subspace of $H^{\varphi,q}_{A}(\mathbb{R}^{n})$
and $\mathcal{H}^{\varphi,\infty}_{A}(\mathbb{R}^{n})$.
\end{lemma}

The following lemma is an anisotropic version of the weak-type weighted Fefferman-stein vector-valued
inequality of the Hardy-Littlewood maximal operator, which is just from \cite[Theorem 1.2]{gly09}.
\begin{lemma}\label{l4.5}
Let $p\in(1,\infty)$ and $\vz\in\mathbb{A}_1(A)$.
Then there exists a positive constant $C$ such that, for any $\lambda, t\in(0,\infty)$ and any
measurable functions $\{f_j\}_{j\in\nn}$ on $\rn$,
$$\varphi\lf(\lf\{x\in\rn:\lf[\sum_{j\in\nn}\lf\{M_{\mathrm{HL}}(f_j)(x)\r\}^p\r]^{\frac1{p}}>\lambda\r\},t\r)
\leq\frac{C}{\lambda}\int_\rn\lf[\sum_{j\in\nn}\lf|f_j(x)\r|^p\r]^{\frac1{p}}\varphi(x,t)dx.$$
\end{lemma}

Using Lemma \ref{l4.5}, we obtain the following anisotropic version of the weak-type vector-valued Fefferman-stein
inequality on $L^\varphi(\rn)$.
\begin{lemma}\label{l4.6}
Let $p\in(1,\infty)$ and $\vz\in\mathbb{A}_1(A)$ be a
Musielak-Orlicz function of uniformly lower type $p^-_\varphi\in(0,\infty)$ and of
uniformly upper type $p^+_\varphi\in(0,\infty)$. If $1\leq p^-_\varphi\leq p^+_\varphi<\infty$,
then there exists a positive constant $C$ such that, for any $\{f_k\}_{k\in\nn}\subset L^\varphi(\rn)$,
$$\lf\|\lf\{\sum_{k\in\nn}\lf[M_{\mathrm{HL}}(f_k)\r]^p\r\}^{\frac1{p}}\r\|_{L^{\varphi,\infty}(\rn)}\leq C
\lf\|\lf(\sum_{k\in\nn}|f_k|^p\r)^{\frac1{p}}\r\|_{L^\varphi(\rn)}.$$
\end{lemma}

\begin{proof}
Let $\{f_k\}_{k\in\nn}\subset L^\varphi(\rn)$.
If $p^-_\varphi\in(1,\infty)$, then by \cite[Remark 2.5(iii)]{jwyyz23},
\cite[Lemma 4.3]{lsll20}, and \cite[Lemma 3.17]{jwyyz23},
we obtain the desired result, the details being omitted here.
If $p^-_\varphi=1$, then by Lemma \ref{l4.5} and an argument similar to the proof of \cite[Lemma 5.6]{jwyyz23},
we obtain the desired result, the details being omitted here.
This completes the proof of Lemma \ref{l4.6}.
\end{proof}

Now, we show Theorem \ref{t1.5} via using Theorems \ref{t4.1} and \ref{t4.2}, Lemmas \ref{l4.3}, \ref{l4.4}, and \ref{l4.6}.
\begin{proof}[Proof of Theorem \ref{t1.5}]
We show this theorem by borrowing some ideas from the proof of \cite[Theorem 5.2]{jwyyz23}.
We first show (i). From Lemma \ref{l4.4}, it follows that the subspace
$L^\infty_{\mathrm{c},s}(\rn)$ is dense in $H^{\varphi,q}_{A}(\mathbb{R}^{n})$
and $\mathcal{H}^{\varphi,\infty}_{A}(\mathbb{R}^{n})$, which implies that, for any given $r\in(1,\infty)$,
$H^{\varphi,q}_{A}(\mathbb{R}^{n})\cap L^r(\rn)$ and $\mathcal{H}^{\varphi,\infty}_{A}(\mathbb{R}^{n})\cap L^r(\rn)$
are, respectively, dense in $H^{\varphi,q}_{A}(\mathbb{R}^{n})$ and $\mathcal{H}^{\varphi,\infty}_{A}(\mathbb{R}^{n})$.
Moreover, by $H^{\varphi,q}_{A}(\mathbb{R}^{n})$, $\mathcal{H}^{\varphi,\infty}_{A}(\mathbb{R}^{n})$
are complete (whose proof is similar to that of \cite[Proposition 2.12 and Remark 2.10(i)]{jwyyz23})
and a standard density argument, to prove (i), it suffices to show that $T$ is bounded from
$H^{\varphi,q}_{A}(\mathbb{R}^{n})\cap L^r(\rn)$ to $H^{\varphi,q}_{A}(\mathbb{R}^{n})$
and from $\mathcal{H}^{\varphi,\infty}_{A}(\mathbb{R}^{n})\cap L^r(\rn)$ to
$\mathcal{H}^{\varphi,\infty}_{A}(\mathbb{R}^{n})$. In what follows, we only show the case $H^{\varphi,q}_{A}(\mathbb{R}^{n})$
since the proof of the case $\mathcal{H}^{\varphi,\infty}_{A}(\mathbb{R}^{n})$ is similar.

To this end, let $r\in(1,\infty)$ and $f\in[H^{\varphi,q}_{A}(\mathbb{R}^{n})\cap L^r(\rn)]$.
From Theorem \ref{t1.4}, to prove (i) of this theorem, it suffices to show that, for any given
$m\in[s,\infty)\cap[m(\varphi),\infty)\cap\zz_+$,
\begin{align}\label{e4.7}
\lf\|\lf(T(f)\r)^*_m\r\|_{L^{\varphi,q}(\rn)}\lesssim\|f\|_{H^{\varphi,q}_{A}(\mathbb{R}^{n})},
\end{align}
where $m(\varphi)$ is as in \eqref{e1.7}. Indeed, by Theorems \ref{t1.2} and \ref{t4.2},
we obtain that there exists a positive constant $C_{(n)}$ (depending only on $n$) and
a sequence $\{a_{k,j}\}_{k\in\zz,j\in\nn}$ of anisotropic $(\varphi,\infty,s)$-atoms associated, respectively,
with the dilated balls $\{B_{k,j}\}_{k\in\zz,j\in\nn}\subset\mathfrak{B}(\rn)$, where $B_{k,j}:=x_{k,j}+B_{\ell_{k,j}}$
with $x_{k,j}\in\rn$ and $\ell_{k,j}\in\zz$,
such that, for any $k\in\zz$, $\sum_{j\in\nn}\mathbf{1}_{B_{k,j}}\leq C_{(n)}$,
\begin{align}\label{e4.8}
f=\sum_{k\in\zz}\sum_{j\in\nn}\lambda_{k,j}a_{k,j}
\end{align}
almost everywhere on $\rn$ and in $L^r(\rn)$, and
\begin{align}\label{e4.9}
\|f\|_{H^{\varphi,q}_{A}(\mathbb{R}^{n})}\sim\lf[\sum_{k\in\zz}2^{kq}
\lf\|\sum_{j\in\nn}\mathbf{1}_{B_{k,j}}\r\|^q_{L^\varphi(\rn)}\r]^{\frac1{q}},
\end{align}
where $\lambda_{k,j}:=C_12^k\|\mathbf{1}_{B_{k,j}}\|_{L^{\varphi}(\rn)}$ for any $k\in\zz$ and $j\in\nn$,
and $C_1$ is a positive constant independent of $f$.
Moreover, by an argument similar to the proof of \cite[Remark 3.42]{jwyyz23}, we conclude that, for any $\widetilde{k}\in\zz$,
\begin{align}\label{e4.10}
f=\sum_{k=-\infty}^{\widetilde{k}-1}\sum_{j\in\nn}\lambda_{k,j}a_{k,j}
+\sum_{k=\widetilde{k}}^{\infty}\sum_{j\in\nn}\lambda_{k,j}a_{k,j}=:f_1+f_2
\end{align}
almost everywhere on $\rn$ and in $L^r(\rn)$. Furthermore, from \eqref{e4.10} and the  boundedness of $T$ on
$L^r(\rn)$ (see, for instance, \cite[Lemma 4.2]{lffy16}), we deduce that, for any $\widetilde{k}\in\zz$,
$$\lf(T(f)\r)^*_m\leq\lf(T(f_1)\r)^*_m+\sum_{k=\widetilde{k}}^\infty\sum_{j\in\nn}\lambda_{k,j}\lf(T(a_{k,j})\r)^*_m.$$
In addition, for any $\widetilde{k}\in\zz$, let $G_{\widetilde{k}}:=\sum_{k=\widetilde{k}}^\infty\sum_{j\in\nn}A^{4\tau}B_{k,j}$,
$$E_1:=\lf\{x\in G_{\widetilde{k}}:\sum_{k=\widetilde{k}}^\infty\sum_{j\in\nn}\lambda_{k,j}\lf(T(a_{k,j})\r)^*_m
(x)>2^{\widetilde{k}}\r\},$$
and $E_2:=\{x\in \rn\setminus G_{\widetilde{k}}:\sum_{k=\widetilde{k}}^\infty\sum_{j\in\nn}\lambda_{k,j}(T(a_{k,j}))^*_m
(x)>2^{\widetilde{k}}\}$. Since $\|\cdot\|_{L^\varphi(\rn)}$ is a quasi-norm of $L^\varphi(\rn)$ (see, for instance,
\cite[Lemma 3.2.2]{jwyyz23}), we know that, for any $\widetilde{k}\in\zz$,
\begin{align}\label{e4.11}
\lf\|\mathbf{1}_{\{x\in\rn:\ \lf(T(f)\r)^*_m(x)>2^{\widetilde{k}+1}\}}\r\|_{L^\varphi(\rn)}
&\lesssim\lf\|\mathbf{1}_{\{x\in\rn:\ \lf(T(f_1)\r)^*_m(x)>2^{\widetilde{k}}\}}\r\|_{L^\varphi(\rn)}
+\lf\|\mathbf{1}_{E_1}\r\|_{L^\varphi(\rn)}+\lf\|\mathbf{1}_{E_2}\r\|_{L^\varphi(\rn)}\nonumber\\
&=:{\rm{I_1+I_2+I_3}}.
\end{align}

Firstly, we deal with ${\rm{I_1}}$. Let $\widetilde{q}\in(\max\{q(\varphi),I(\varphi)\},\infty)$,
$p^+_\varphi\in(\max\{q(\varphi),I(\varphi)\},\widetilde{q})$, $a\in(0,1-\frac{p^+_\varphi}{\widetilde{q}})$,
$p^-_\varphi\in(0,\min\{1,q,i(\varphi)\})$. Then, by the facts that
$$\mathbf{1}_{\{x\in\rn:\ \lf(T(f_1)\r)^*_m(x)>2^{\widetilde{k}+1}\}}\leq
2^{-\widetilde{k}\widetilde{q}}\lf[\lf(T(f_1)\r)^*_m\r]^{\widetilde{q}}$$
and $(T(f_1))^*_m\lesssim M_{\mathrm{HL}}(T(f_1))$, $\vz\in\mathbb{A}_{\widetilde{q}}(A)$,
Lemmas \ref{l2.2}(ii) and \ref{l4.3}, we conclude that, for any $\widetilde{k}\in\zz$ and $\lambda\in(0,\infty)$,
\begin{align*}
F:=\varphi\lf(\lf\{x\in\rn:\ \lf(T(f_1)\r)^*_m(x)>2^{\widetilde{k}}\r\},\frac{2^{\widetilde{k}}}{\lambda}\r)
&=\int_{\{x\in\rn:\ \lf(T(f_1)\r)^*_m(x)>2^{\widetilde{k}}\}}\varphi\lf(x,\frac{2^{\widetilde{k}}}{\lambda}\r)dx\\
&\lesssim2^{-\widetilde{k}\widetilde{q}}\int_\rn\lf[\lf(T(f_1)\r)^*_m(x)\r]^{\widetilde{q}}
\varphi\lf(x,\frac{2^{\widetilde{k}}}{\lambda}\r)dx\\
&\lesssim2^{-\widetilde{k}\widetilde{q}}\int_\rn\lf[M_{\mathrm{HL}}\lf(T(f_1)\r)(x)\r]^{\widetilde{q}}
\varphi\lf(x,\frac{2^{\widetilde{k}}}{\lambda}\r)dx\\
&\lesssim2^{-\widetilde{k}\widetilde{q}}\int_\rn\lf[\sum_{k=-\infty}^{\widetilde{k}-1}
\sum_{j\in\nn}\lf|\lambda_{k,j}a_{k,j}(x)\r|\r]^{\widetilde{q}}
\varphi\lf(x,\frac{2^{\widetilde{k}}}{\lambda}\r)dx,
\end{align*}
which, combined with the H\"{o}lder inequality for $\frac1{\widetilde{q}'}+\frac1{\widetilde{q}}=1$,
$\sum_{j\in\nn}\mathbf{1}_{B_{k,j}}\lesssim1$ for any $k\in\zz$, the Fubini theorem,
$\|a_{k,j}\|_{L^\infty(\rn)}\le \|\mathbf{1}_{B_{k,j}}\|^{-1}_{L^\varphi(\rn)}$,
$\lambda_{k,j}:=C_12^k\|\mathbf{1}_{B_{k,j}}\|_{L^{\varphi}(\rn)}$, and the uniformly upper type $p^+_\varphi$
property of $\varphi$, further implies that
\begin{align*}
F&\lesssim2^{-\widetilde{k}\widetilde{q}}\int_\rn\lf(\sum_{k=-\infty}^{\widetilde{k}-1}2^{ka\widetilde{q}'}\r)
^{\frac{\widetilde{q}}{\widetilde{q}'}}\sum_{k=-\infty}^{\widetilde{k}-1}2^{-ka\widetilde{q}}
\lf[\sum_{j\in\nn}\lf|\lambda_{k,j}a_{k,j}(x)\r|\r]^{\widetilde{q}}\varphi\lf(x,\frac{2^{k}}{\lambda}\r)dx\\
&\lesssim2^{-\widetilde{k}\widetilde{q}(1-a)}\sum_{k=-\infty}^{\widetilde{k}-1}2^{-ka\widetilde{q}}\int_\rn\sum_{j\in\nn}
\lf|\lambda_{k,j}a_{k,j}(x)\r|^{\widetilde{q}}\varphi\lf(x,\frac{2^{k}}{\lambda}\r)dx\\
&\lesssim2^{-\widetilde{k}\widetilde{q}(1-a)}\sum_{k=-\infty}^{\widetilde{k}-1}2^{-ka\widetilde{q}}\sum_{j\in\nn}
2^{k\widetilde{q}}\varphi\lf(B_{k,j},\frac{2^{k}}{\lambda}\r)\\
&\lesssim2^{-\widetilde{k}[\widetilde{q}(1-a)-p^+_\varphi]}\sum_{k=-\infty}^{\widetilde{k}-1}
2^{k[\widetilde{q}(1-a)-p^+_\varphi]}\sum_{j\in\nn}\lf(B_{k,j},\frac{2^{k}}{\lambda}\r).
\end{align*}
By this and \eqref{e2.11}, we find that, for any $\widetilde{k}\in\zz$ and $\lambda\in(0,\infty)$,
\begin{align}\label{e4.12}
F&\lesssim\max\lf\{\frac{2^{-\widetilde{k}[\widetilde{q}(1-a)-p^+_\varphi]}}{\lambda^{p^+_\varphi}}
\sum_{k=-\infty}^{\widetilde{k}-1}2^{k\widetilde{q}(1-a)}\lf\|\sum\limits_{j\in\nn}
\mathbf{1}_{B_{k,j}}\r\|^{p^+_\varphi}_{L^\varphi(\rn)},\r.\nonumber\\
&\hs\lf.\frac{2^{-\widetilde{k}[\widetilde{q}(1-a)-p^+_\varphi]}}{\lambda^{p^-_\varphi}}
\sum_{k=-\infty}^{\widetilde{k}-1}2^{k[\widetilde{q}(1-a)-p^+_\varphi+p^-_\varphi]}\lf\|\sum\limits_{j\in\nn}
\mathbf{1}_{B_{k,j}}\r\|^{p^-_\varphi}_{L^\varphi(\rn)}\r\}.
\end{align}
Let $\lambda:=2^{\widetilde{k}}\|\mathbf{1}_{\{x\in\rn:\ (T(f_1))^*_m(x)>2^{\widetilde{k}}\}}\|_{L^\varphi(\rn)}$. Then, by the
definition of $F$, and Lemma \ref{l2.6}, we conclude that $F=1$, which, together with
the definition of $\lambda$ and ${\rm{I_1}}$, and \eqref{e4.12}, further implies that
\begin{align*}
2^{\widetilde{k}}{\rm{I_1}}&\lesssim2^{\widetilde{k}}\lf\|\mathbf{1}_{\{x\in\rn:\ \lf(T(f_1)\r)^*_m(x)>2^{\widetilde{k}}\}}\r\|_{L^\varphi(\rn)}=\lambda\\
&\lesssim\max\lf\{2^{-\frac{\widetilde{k}[\widetilde{q}(1-a)-p^+_\varphi]}{p^+_\varphi}}\lf[
\sum_{k=-\infty}^{\widetilde{k}-1}2^{k\widetilde{q}(1-a)}
\lf\|\sum\limits_{j\in\nn}\mathbf{1}_{B_{k,j}}\r\|^{p^+_\varphi}_{L^\varphi(\rn)}\r]^{\frac{1}{p^+_\varphi}},\r.\\
&\quad \quad \quad \quad\lf.2^{-\frac{\widetilde{k}[\widetilde{q}(1-a)-p^+_\varphi]}{p^-_\varphi}}
\lf[\sum_{k=-\infty}^{\widetilde{k}-1}2^{\frac{k[\widetilde{q}(1-a)-p^+_\varphi]}{p^-_\varphi}}2^{kp^-_\varphi}
\lf\|\sum\limits_{j\in\nn}\mathbf{1}_{B_{k,j}}\r\|^{p^-_\varphi}_{L^\varphi(\rn)}\r]^{\frac1{p^-_\varphi}}\r\}.
\end{align*}
By this, $a\in(0,1-\frac{p^+_\varphi}{\widetilde{q}})$, Lemma \ref{l2.7}, and \eqref{e4.9}, we find that
\begin{align}\label{e4.13}
\lf[\sum_{\widetilde{k}\in\zz}\lf(2^{\widetilde{k}}{\rm{I_{1}}}\r)^q\r]^{\frac1{q}}\ls
\lf[\sum_{k\in\zz}2^{kq}\lf\|\sum_{j\in\nn}\mathbf{1}_{B_{k,j}}\r\|^q_{L^\varphi(\rn)}\r]^{\frac1{q}}
\sim\|f\|_{H^{\varphi,q}_{A}(\mathbb{R}^{n})},
\end{align}
which implies the desired result.

Next we deal with ${\rm{I_{2}}}$. Let $p^-_\varphi\in(0,\min\{1,q,i(\varphi)\})$. Then, from
$G_{\widetilde{k}}:=\sum_{k=\widetilde{k}}^\infty\sum_{j\in\nn}A^{4\tau}B_{k,j}$,
Lemma \ref{l2.12}, \ref{l2.9}(i), and \ref{l2.12}, the fact that $(\sum_{k\in\nn}|\lambda_k|)^\theta\leq\sum_{k\in\nn}|\lambda_k|^\theta$
for any $\{\lambda_k\}_{k\in\nn}\subset\mathbb{C}$ and $\theta\in(0,1]$, Lemma \ref{l2.10},
$\sum_{j\in\nn}\mathbf{1}_{B_{k,j}}\lesssim1$ for any $k\in\zz$, and Lemma \ref{l2.9}(ii),
we deduce that, for any $\widetilde{k}\in\zz$,
\begin{align}\label{e4.14}
{\rm{I_{2}}}&\leq\lf\|\mathbf{1}_{G_{\widetilde{k}}}\r\|_{L^\varphi(\rn)}
\leq\lf\|\sum_{k=\widetilde{k}}^\infty\sum_{j\in\nn}\mathbf{1}_{A^{4\tau}B_{k,j}}\r\|_{L^\varphi(\rn)}
\lesssim\lf\|\sum_{k=\widetilde{k}}^\infty\sum_{j\in\nn}\mathbf{1}_{B_{k,j}}\r\|_{L^\varphi(\rn)}\nonumber\\
&=\lf\|\lf(\sum_{k=\widetilde{k}}^\infty\sum_{j\in\nn}\mathbf{1}_{B_{k,j}}\r)
^{p^-_\varphi}\r\|^{\frac1{p^-_\varphi}}_{L^{\varphi_{1/p^-_\varphi}}(\rn)}
\leq\lf\|\sum_{k=\widetilde{k}}^\infty\lf(\sum_{j\in\nn}\mathbf{1}_{B_{k,j}}\r)
^{p^-_\varphi}\r\|^{\frac1{p^-_\varphi}}_{L^{\varphi_{1/p^-_\varphi}}(\rn)}\nonumber\\
&\lesssim\lf\|\sum_{k=\widetilde{k}}^\infty\sum_{j\in\nn}\mathbf{1}_{B_{k,j}}
\r\|^{\frac1{p^-_\varphi}}_{L^{\varphi_{1/p^-_\varphi}}(\rn)}
\lesssim\lf[\sum_{k=\widetilde{k}}^\infty\lf\|\sum_{j\in\nn}\mathbf{1}_{B_{k,j}}
\r\|_{L^{\varphi_{1/p^-_\varphi}}(\rn)}\r]^{\frac1{p^-_\varphi}}
\lesssim\lf[\sum_{k=\widetilde{k}}^\infty\lf\|\sum_{j\in\nn}\mathbf{1}_{B_{k,j}}
\r\|^{p^-_\varphi}_{L^{\varphi}(\rn)}\r]^{\frac1{p^-_\varphi}}.
\end{align}
By \eqref{e4.14} and Lemma \ref{l2.7} with $\widetilde{\alpha}:=1$, $\widetilde{\beta}:=q$, and
$\mu_k:=2^k\|\mathbf{1}_{B_{{k,j}}}\|_{L^{\varphi}(\rn)}$, and \eqref{e4.9}, we conclude that
\begin{align}\label{e4.15}
\lf[\sum_{\widetilde{k}\in\zz}\lf(2^{\widetilde{k}}{\rm{I_{2}}}\r)^q\r]^{\frac1{q}}\ls
\lf[\sum_{k\in\zz}2^{kq}\lf\|\sum_{j\in\nn}\mathbf{1}_{B_{k,j}}\r\|^q_{L^\varphi(\rn)}\r]^{\frac1{q}}
\sim\|f\|_{H^{\varphi,q}_{A}(\mathbb{R}^{n})},
\end{align}
which implies the desired result.

Finally, we deal with $\rm{I_{3}}$. From $m\in[s,\infty)\cap\zz_+$ and an argument similar to that
used in the estimate of \cite[pp.1710-1712]{lyy16}, it follows that, for any $\phi\in\cs_m(\rn)$
and $x\in(A^{4\tau}B_{k,j})^\complement$,
$$\sup_{t\in(0,\infty)}\lf|T(a_{k,j})\ast\phi_t(x)\r|
\lesssim\lf[M_{\mathrm{HL}}\lf(\mathbf{1}_{B_{k,j}}\r)(x)\r]^{1+\delta}\lf\|\mathbf{1}_{B_{{k,j}}}\r\|^{-1}_{L^{\varphi}(\rn)},$$
which, combined with Lemma \ref{l2.8}, further implies that
\begin{align}\label{e4.16}
\lf(T(a_{k,j})\r)^*_m(x)\lesssim\lf[M_{\mathrm{HL}}\lf(\mathbf{1}_{B_{k,j}}\r)(x)\r]^{1+\delta}
\lf\|\mathbf{1}_{B_{{k,j}}}\r\|^{-1}_{L^{\varphi}(\rn)}.
\end{align}
Moreover, since $\frac{i(\varphi)}{q(\varphi)}\in(\frac{1}{1+\delta},\infty)$, we can choose an
$r_1\in(\frac{q(\varphi)}{i(\varphi)(1+\delta)},1)\cap(\frac1{1+\delta},1)$. Meanwhile,
let $p^-_\varphi\in(0,\min\{1,q,i(\varphi)\})$. Then, by $r_1\in(\frac{q(\varphi)}{i(\varphi)(1+\delta)},1)\cap(\frac1{1+\delta},1)$, Lemma \ref{l2.9}(i),
$p^-_\varphi\in(0,1)$, the fact that $(\sum_{k\in\nn}|\lambda_k|)^\theta\leq\sum_{k\in\nn}|\lambda_k|^\theta$
for any $\{\lambda_k\}_{k\in\nn}\subset\mathbb{C}$ and $\theta\in(0,1]$,
$\lambda_{k,j}:=C_12^k\|\mathbf{1}_{B_{k,j}}\|_{L^{\varphi}(\rn)}$, Lemma \ref{l2.10}, and
\eqref{e4.16}, we know that, for any $\widetilde{k}\in\zz$,
\begin{align*}
{\rm{I_{3}}}&\leq2^{-\widetilde{k}r_1}\lf\|\lf[\sum_{k=\widetilde{k}}^\infty
\sum_{j\in\nn}\lambda_{k,j}(T(a_{k,j}))^*_m\r]^{r_1}\mathbf{1}_{\rn\setminus G_{\widetilde{k}}}\r\|_{L^{\varphi}(\rn)}\\
&\leq2^{-\widetilde{k}r_1}\lf\|\sum_{k=\widetilde{k}}^\infty
\sum_{j\in\nn}\lf[\lambda_{k,j}(T(a_{k,j}))^*_m\r]^{r_1}\mathbf{1}_{\rn\setminus G_{\widetilde{k}}}\r\|_{L^{\varphi}(\rn)}\\
&=2^{-\widetilde{k}r_1}\lf\|\lf\{\sum_{k=\widetilde{k}}^\infty
\sum_{j\in\nn}\lf[\lambda_{k,j}(T(a_{k,j}))^*_m\r]^{r_1}\mathbf{1}_{\rn\setminus G_{\widetilde{k}}}\r\}^{p^-_\varphi}\r\|^{\frac1{p^-_\varphi}}_{L^{\varphi_{1/{p^-_\varphi}}}(\rn)}\\
&\leq2^{-\widetilde{k}r_1}\lf\|\sum_{k=\widetilde{k}}^\infty
\lf\{\sum_{j\in\nn}\lf[\lambda_{k,j}(T(a_{k,j}))^*_m\r]^{r_1}\mathbf{1}_{\rn\setminus G_{\widetilde{k}}}\r\}^{p^-_\varphi}\r\|^{\frac1{p^-_\varphi}}_{L^{\varphi_{1/{p^-_\varphi}}}(\rn)}\\
&\lesssim2^{-\widetilde{k}r_1}\lf[\sum_{k=\widetilde{k}}^\infty\lf\|
\lf\{\sum_{j\in\nn}\lf[\lambda_{k,j}(T(a_{k,j}))^*_m\r]^{r_1}\mathbf{1}_{\rn\setminus G_{\widetilde{k}}}\r\}^{p^-_\varphi}\r\|_{L^{\varphi_{1/{p^-_\varphi}}}(\rn)}\r]^{\frac1{p^-_\varphi}}\\
&\lesssim2^{-\widetilde{k}r_1}\lf[\sum_{k=\widetilde{k}}^\infty2^{kr_1p^-_\varphi}\lf\|
\lf\{\sum_{j\in\nn}\lf[M_{\mathrm{HL}}\lf(\mathbf{1}_{B_{k,j}}\r)\r]^{r_1(1+\delta)}
\r\}^{p^-_\varphi}\r\|_{L^{\varphi_{1/{p^-_\varphi}}}(\rn)}\r]^{\frac1{p^-_\varphi}}.
\end{align*}
From this, Lemma \ref{l2.9}(i), $q(\varphi)<r_1i(\varphi)(1+\delta)\leq r_1I(\varphi)(1+\delta)<\infty$,
Lemma \ref{l2.9}(ii), $r_1(1+\delta)\in(1,\infty)$, and Lemma \ref{l2.11}, it follows that, for any $\widetilde{k}\in\zz$,
\begin{align*}
{\rm{I_{3}}}&\lesssim2^{-\widetilde{k}r_1}\lf[\sum_{k=\widetilde{k}}^\infty2^{kr_1p^-_\varphi}\lf\|
\lf\{\sum_{j\in\nn}\lf[M_{\mathrm{HL}}\lf(\mathbf{1}_{B_{k,j}}\r)\r]^{r_1(1+\delta)}
\r\}^{\frac1{r_1(1+\delta)}}\r\|^{p^-_\varphi r_1(1+\delta)}_{L^{\varphi_{{r_1(1+\delta)}}}(\rn)}\r]^{\frac1{p^-_\varphi}}\\
&\lesssim2^{-\widetilde{k}r_1}\lf[\sum_{k=\widetilde{k}}^\infty2^{kr_1p^-_\varphi}
\lf\|\lf(\sum_{j\in\nn}\mathbf{1}_{B_{k,j}}\r)^{\frac1{r_1(1+\delta)}}\r\|
^{p^-_\varphi r_1(1+\delta)}_{L^{\varphi_{{r_1(1+\delta)}}}(\rn)}\r]^{\frac1{p^-_\varphi}}\\
&\sim2^{-\widetilde{k}r_1}\lf[\sum_{k=\widetilde{k}}^\infty2^{kr_1p^-_\varphi}
\lf\|\sum_{j\in\nn}\mathbf{1}_{B_{k,j}}\r\|^{p^-_\varphi}_{L^{\varphi}(\rn)}\r]^{\frac1{p^-_\varphi}}.
\end{align*}
By this, $r_1\in(0,1)$, and Lemma \ref{l2.7} with $\widetilde{\alpha}:=1-r_1$, $\widetilde{\beta}:=p^-_\varphi$, and
$\mu_k:=2^k\|\mathbf{1}_{B_{{k,j}}}\|_{L^{\varphi}(\rn)}$, and \eqref{e4.9}, we know that
\begin{align}\label{e4.17}
\lf[\sum_{\widetilde{k}\in\zz}\lf(2^{\widetilde{k}}{\rm{I_{3}}}\r)^q\r]^{\frac1{q}}\ls
\lf[\sum_{k\in\zz}2^{kq}\lf\|\sum_{j\in\nn}\mathbf{1}_{B_{k,j}}\r\|^q_{L^\varphi(\rn)}\r]^{\frac1{q}}
\sim\|f\|_{H^{\varphi,q}_{A}(\mathbb{R}^{n})},
\end{align}
which implies the desired result.

Thus, by the fact that $\|\cdot\|_{L^{\varphi,q}(\rn)}$ is a quasi-norm of ${L^{\varphi,q}(\rn)}$
(see, for instance, \cite[Remark 2.5(ii)]{jwyyz23}),
\eqref{e4.11}, \eqref{e4.13}, \eqref{e4.15}, and \eqref{e4.17}, we conclude that
\begin{align*}
\lf\|\lf(T(f)\r)^*_m\r\|_{L^{\varphi,q}(\rn)}
&\lesssim\lf[\sum_{\widetilde{k}\in\zz}2^{(\widetilde{k}+1)q}
\lf\|\mathbf{1}_{\{x\in\rn:\ \lf(T(f)\r)^*_m(x)>2^{\widetilde{k}+1}\}}\r\|^q_{L^\varphi(\rn)}\r]^{\frac1{q}}\\
&\lesssim\lf[\sum_{\widetilde{k}\in\zz}\lf(2^{\widetilde{k}}{\rm{I_{1}}}\r)^q\r]^{\frac1{q}}
+\lf[\sum_{\widetilde{k}\in\zz}\lf(2^{\widetilde{k}}{\rm{I_{2}}}\r)^q\r]^{\frac1{q}}
+\lf[\sum_{\widetilde{k}\in\zz}\lf(2^{\widetilde{k}}{\rm{I_{3}}}\r)^q\r]^{\frac1{q}}\\
&\lesssim\|f\|_{H^{\varphi,q}_{A}(\mathbb{R}^{n})},
\end{align*}
which implies that \eqref{e4.7} holds true. Therefore, from \eqref{e4.7}, Lemma \ref{l4.4}, and a density argument,
we finishes the proof of the case $H^{\varphi,q}_{A}(\mathbb{R}^{n})$ and hence $(\mathrm{i})$.

Now, we show (ii). By Theorem \ref{t4.1}(ii), we conclude that the subspace
$L^\infty_{\mathrm{c},s}(\rn)$ is dense in $H^{\varphi}_{A}(\mathbb{R}^{n})$, which implies that, for any $r\in(1,\infty)$,
$H^{\varphi}_{A}(\mathbb{R}^{n})\cap L^r(\rn)$ is dense in $H^{\varphi}_{A}(\mathbb{R}^{n})$. By this, the
fact that the completeness of $H^{\varphi}_{A}(\mathbb{R}^{n})$
(which proof is similar to that of \cite[Proposition 6]{lyy14} and \cite[Proposition 2.8]{lyy16}),
and a standard density argument, to prove (ii), it suffices to show that, for any $r\in(1,\infty)$,
$T$ is bounded from $H^{\varphi}_{A}(\mathbb{R}^{n})\cap L^r(\rn)$ to $H^{\varphi,\infty}_{A}(\mathbb{R}^{n})$.

Let $r\in(1,\infty)$ and $f\in H^{\varphi}_{A}(\mathbb{R}^{n})$. Noticing that
$H^{\varphi}_{A}(\mathbb{R}^{n})\subset H^{\varphi,\infty}_{A}(\mathbb{R}^{n})$
(see, for instance, \cite[Remark 3.24]{jwyyz23}), we can decompose $f$ in the same way as in
\eqref{e4.8} with $q:=\infty$. Then we have
\begin{align}\label{e4.18}
\|f\|_{H^{\varphi,\infty}_{A}(\mathbb{R}^{n})}\sim
\sup_{k\in\zz}2^k\lf\|\sum_{j\in\nn}\mathbf{1}_{B_{k,j}}\r\|_{L^\varphi(\rn)}.
\end{align}
By Theorem \ref{t1.4}, to show (ii), it suffices to show that,
for any given $m\in[s,\infty)\cap\zz_+$,
\begin{align}\label{e4.19}
\lf\|\lf(T(f)\r)^*_m\r\|_{L^{\varphi,\infty}(\rn)}
\lesssim\|f\|_{H^{\varphi,\infty}_{A}(\mathbb{R}^{n})}.
\end{align}
To this end, for any $\widetilde{k}\in\zz$, let ${\rm{I_{1}}}$, ${\rm{I_{2}}}$, and ${\rm{I_{3}}}$
be as in \eqref{e4.11}. Then
\begin{align}\label{e4.20}
\lf\|\lf(T(f)\r)^*_m\r\|_{L^{\varphi,\infty}(\rn)}
\lesssim\sup_{\widetilde{k}\in\zz}2^{\widetilde{k}}\lf({\rm{I_{1}+I_2+I_3}}\r).
\end{align}
Moreover, by \eqref{e4.18} and $H^{\varphi}_{A}(\mathbb{R}^{n})\subset H^{\varphi,\infty}_{A}(\mathbb{R}^{n})$,
similarly to the estimates of \eqref{e4.13} and \eqref{e4.15}, we conclude that
\begin{align}\label{e4.21}
\sup_{\widetilde{k}\in\zz}2^{\widetilde{k}}\lf({\rm{I_{1}+I_2}}\r)
\lesssim\|f\|_{H^{\varphi,\infty}_{A}(\mathbb{R}^{n})}\lesssim\|f\|_{H^{\varphi}_{A}(\mathbb{R}^{n})},
\end{align}
which implies the desired result.

Finally, we deal with ${\rm{I_{3}}}$. Indeed, from Theorem \ref{t4.2}(i), we deduce that
\begin{align}\label{e4.22}
\lf\|\sum_{k\in\zz}\sum_{j\in\nn}\frac{\lambda_{k,j}\mathbf{1}_{B_{k,j}}}
{\|\mathbf{1}_{B_{k,j}}\|_{L^\varphi(\rn)}}\r\|_{L^\varphi(\rn)}
\lesssim\|f\|_{H^{\varphi}_{A}(\mathbb{R}^{n})}.
\end{align}
Moreover, from Lemma \ref{l2.9}(ii), we deduce that $i(\varphi_{1+\delta})=(1+\delta)i(\varphi)$ is
attainable and $q(\varphi_{1+\delta})=q(\varphi)=1$. By this, \eqref{e4.16},
Lemmas \ref{l2.9}(i) and \ref{l4.6}, and \eqref{e4.22}, we conclude that, for any $\widetilde{k}\in\zz$,
\begin{align*}
2^{\widetilde{k}}{\rm{I_{3}}}
&\lesssim2^{\widetilde{k}}\lf\|\mathbf{1}_{A_1}\r\|_{L^\varphi(\rn)}
\sim2^{\widetilde{k}}\lf\|\mathbf{1}_{A_2}\r\|^{1+\delta}_{L^{\varphi_{1+\delta}}(\rn)}\\
&\lesssim\lf\|\lf\{\sum_{k\in\zz}\sum_{j\in\nn}\frac{\lambda_{k,j}}
{\|\mathbf{1}_{B_{k,j}}\|_{L^\varphi(\rn)}}\lf[M_{\mathrm{HL}}\lf(\mathbf{1}_{B_{k,j}}\r)\r]^{1+\delta}
\r\}^{\frac1{1+\delta}}\r\|^{1+\delta}_{L^{\varphi_{1+\delta},\infty}(\rn)}\\
&\lesssim\lf\|\lf\{\sum_{k\in\zz}\sum_{j\in\nn}\frac{\lambda_{k,j}\mathbf{1}_{B_{k,j}}}
{\|\mathbf{1}_{B_{k,j}}\|_{L^\varphi(\rn)}}
\r\}^{\frac1{1+\delta}}\r\|^{1+\delta}_{L^{\varphi_{1+\delta}}(\rn)}\\
&\sim\lf\|\sum_{k\in\zz}\sum_{j\in\nn}\frac{\lambda_{k,j}\mathbf{1}_{B_{k,j}}}
{\|\mathbf{1}_{B_{k,j}}\|_{L^\varphi(\rn)}}\r\|_{L^{\varphi}(\rn)}
\lesssim\|f\|_{H^{\varphi}_{A}(\mathbb{R}^{n})},
\end{align*}
where $A_1:=\{x\in G_{\widetilde{k}}:\ \sum_{k\in\zz}\sum_{j\in\nn}\lambda_{k,j}
\{M_{\mathrm{HL}}(\mathbf{1}_{B_{k,j}})(x)\}^{1+\delta} \|\mathbf{1}_{B_{{k,j}}}\|^{-1}_{L^{\varphi}(\rn)}
>2^{\widetilde{k}}\}$ and
$$A_2:=\lf\{x\in \rn\backslash G_{\widetilde{k}}:\ \sum_{k\in\zz}\sum_{j\in\nn}\lambda_{k,j}
\lf\{M_{\mathrm{HL}}\lf(\mathbf{1}_{B_{k,j}}\r)(x)\r\}^{1+\delta} \lf\|\mathbf{1}_{B_{{k,j}}}\r\|^{-1}_{L^{\varphi}(\rn)}
>2^{\widetilde{k}}\r\}.$$
By this, \eqref{e4.20}, and \eqref{e4.21}, we know that \eqref{e4.19} holds true.
This completes the proof of (ii) and hence of Theorem \ref{t1.5}.
\end{proof}






\bigskip

\noindent Xiong Liu: School of Mathematics and Physics,
Gansu Center for Fundamental Research in Complex Systems Analysis and Control,
Lanzhou Jiaotong University, Lanzhou 730070, P. R. China

\medskip

\noindent Wenhua Wang (Corresponding author): Institute for Advanced Study in Mathematics,
Harbin Institute of Technology, Harbin 150001, P. R. China

\medskip

\smallskip

\noindent{E-mails}:\\
\texttt{liuxmath@126.com} (Xiong Liu)  \\
\texttt{whwangmath@whu.edu.cn} (Wenhua Wang)
\bigskip \medskip

\end{document}